\documentclass[11pt,letterpaper]{article}

\usepackage{anysize,amsmath,amsthm,amsfonts,amssymb,graphics,graphicx,enumerate,multirow}

\usepackage[T1]{fontenc}
\usepackage{xcolor}

\usepackage[shortlabels,inline]{enumitem}

\usepackage{lscape}

\usepackage{tikz}
\usetikzlibrary{arrows,shapes}
\usepackage{subfigure}
\usetikzlibrary{patterns}
\usetikzlibrary{decorations.pathreplacing}
\usepackage{pgfplots}

\usepackage{setspace}
\usepackage[margin=0.9in]{geometry}

\usepackage{natbib}
\usepackage{commath}
\bibpunct[, ]{(}{)}{,}{a}{}{,}%

\usepackage[compact]{titlesec}  

\usepackage{bm}
\usepackage{booktabs}
\usepackage{tabularx}
\usepackage{array}
\usepackage{color}
\usepackage{url}
\usepackage{algorithm}
\usepackage{algpseudocode}
\usepackage{endnotes}
\usepackage{enumerate}
\usepackage{enumitem}
\usepackage{rotating}
\usepackage{subfigure}
\usepackage{graphicx}
\usepackage{multirow}
\usepackage{multicol}
\usepackage{hhline}

\usepackage{natbib}
\usepackage{commath}
\bibpunct[, ]{(}{)}{,}{a}{}{,}%

\usepackage{booktabs} 
\usepackage[english]{babel}
\usepackage{amsfonts,dsfont}
\usepackage{cleveref}
\usepackage{xcolor}
\usepackage{commath}
\usepackage{float}
\usepackage{graphicx}
\usepackage{soul}
\setstcolor{red}

\usepackage{thm-restate}
\usepackage{thmtools}

\usepackage{verbatim}
\usepackage{kbordermatrix}
\usepackage{blkarray}

\allowdisplaybreaks
\newlist{compactitem}{itemize}{3}
\setlist[compactitem]{topsep=0pt,partopsep=0pt,itemsep=0pt,parsep=0pt}
\setlist[compactitem,1]{label=\textbullet}
\setlist[compactitem,2]{label=---}
\setlist[compactitem,3]{label=*}

\newlist{compactdesc}{description}{3}
\setlist[compactdesc]{topsep=0pt,partopsep=0pt,itemsep=0pt,parsep=0pt}

\newlist{compactenum}{enumerate}{3}
\setlist[compactenum]{topsep=0pt,partopsep=0pt,itemsep=0pt,parsep=0pt}
\setlist[compactenum,1]{label=\arabic*}
\setlist[compactenum,2]{label=\alph*}
\setlist[compactenum,3]{label=\roman*}

\newtheorem{thm}{Theorem}
\newtheorem{cor}{Corollary}
\newtheorem{lem}{Lemma}
\newtheorem{asm}{Assumption}

\newtheorem{fact}{Fact}
\newtheorem{prop}{Proposition}

\newtheorem{defn}{Definition}

\newtheorem{rem}{Remark}

\newcommand{\yk}[1]{{\color{red}[YK: #1]}}
\newcommand{\edit}[1]{{\color{teal} #1}}

\newcommand{\yka}[1]{}

\newcommand{\sida}[1]{}

\newcommand{\pqa}[1]{}
\newcommand{\pqd}[1]{}

\newcommand{\prob}{{\mathbb{P}}}
\newcommand{\eps}{{\epsilon}}
\newcommand{\Ex}{{\mathbb{E}}}

\newcommand{\N}{K} 

\newcommand{\Iem}{I_{\textup{\tiny em}}}
\makeatletter
\newcommand\footnoteref[1]{\protected@xdef\@thefnmark{\ref{#1}}\@footnotemark}
\makeatother

\newcommand{\X}{\mathbf{X}} 
\newcommand{\e}{\mathbf{e}} 
\newcommand{\V}{L} 
\newcommand{\T}{\top} 

\newcommand{\bA}{\mathbf{A}}

\newcommand{\bD}{\mathbf{D}}
\newcommand{\bE}{\mathbf{E}}
\newcommand{\bof}{\mathbf{f}}

\newcommand{\bv}{\mathbf{v}}
\newcommand{\bx}{\mathbf{x}}
\newcommand{\bX}{\mathbf{X}}

\newcommand{\by}{\mathbf{y}}

\newcommand{\bz}{\mathbf{z}}

\newcommand{\bbp}{\mathbb{P}}

\newcommand{\fA}{\bar{\mathbf{A}}} 
\newcommand{\fE}{\bar{\mathbf{E}}} 
\newcommand{\fX}{\bar{\mathbf{X}}} 
\newcommand{\fKX}{\bar{\mathbf{X}}^{\N}} 
\newcommand{\fKA}{\bar{\mathbf{A}}^{\N}} 

\newcommand{\PsiU}{\Psi^{U}} 
\newcommand{\PsiKU}{\Psi^{K,U}} 

\newcommand{\cJ}{\mathcal{J}}
\newcommand{\cT}{\mathcal{T}}
\newcommand{\cU}{\mathcal{U}}

\newcommand{\ind}{{\mathbb{I}}}

\newcommand{\w}{\alpha} 
\newcommand{\alphatwo}{w}

\newcommand{\bzero}{\mathbf{0}}

\newcommand{\balpha}{\boldsymbol{\alpha}}

\newcommand{\gammaopt}{{\bar{\gamma}}}
\newcommand{\blambda}{\boldsymbol{\lambda}}

\newcommand{\bmu}{\boldsymbol{\mu}}

\newcommand{\bphi}{\boldsymbol{\phi}}

\makeatletter
\newsavebox\myboxA
\newsavebox\myboxB
\newlength\mylenA

\newcommand*\xoverline[2][0.82]{%
	\sbox{\myboxA}{$\m@th#2$}%
	\setbox\myboxB\null
	\ht\myboxB=\ht\myboxA%
	\dp\myboxB=\dp\myboxA%
	\wd\myboxB=#1\wd\myboxA
	\sbox\myboxB{$\m@th\overline{\copy\myboxB}$}
	\setlength\mylenA{\the\wd\myboxA}
	\addtolength\mylenA{-\the\wd\myboxB}%
	\ifdim\wd\myboxB<\wd\myboxA%
	\rlap{\hskip 0.5\mylenA\usebox\myboxB}{\usebox\myboxA}%
	\else
	\hskip -0.5\mylenA\rlap{\usebox\myboxA}{\hskip 0.5\mylenA\usebox\myboxB}%
	\fi}
\makeatother

\newtheorem*{theorem*}{Theorem}

\pgfplotsset{
	compat=1.8,
	tick label style = {font=\scriptsize},
	every axis label = {font=\scriptsize},
	legend style = {font=\scriptsize},
	label style = {font=\scriptsize},
	/pgfplots/xlabel near ticks/.style={
		/pgfplots/every axis x label/.style={
			at={(ticklabel cs:0.5)},anchor=near ticklabel
		}
	},
	/pgfplots/ylabel near ticks/.style={
		
		/pgfplots/every axis y label/.style={
			at={(ticklabel cs:0.5)},rotate=90,anchor=near ticklabel}
	}
	
}

\expandafter\def\expandafter\normalsize\expandafter{%
	\normalsize
	\setlength\abovedisplayskip{5pt}
	\setlength\belowdisplayskip{5pt}
	\setlength\abovedisplayshortskip{4pt}
	\setlength\belowdisplayshortskip{4pt}
}

\pdfoutput=1

\RequirePackage[normalem]{ulem} 
\RequirePackage{color}\definecolor{RED}{rgb}{1,0,0}\definecolor{BLUE}{rgb}{0,0,1} 

\begin{document}

\vspace{-80pt}
\title{
Large Deviations Optimal Scheduling of Closed Queueing Networks
}
\author{
    Siddhartha Banerjee\thanks{
		School of ORIE, Cornell University. Email: \texttt{sbanerjee@cornell.edu}}
\and Yash Kanoria\thanks{
		Graduate School of Business, Columbia University. Email: \texttt{ykanoria@columbia.edu}}
\and Pengyu Qian\thanks{
Krannert School of Management, Purdue University. Email: \texttt{qianp@purdue.edu}}
}
\date{}	


\maketitle

\vspace{-20pt}
\begin{abstract}
\singlespacing
\vspace{-10pt}
We study the design of dynamic scheduling controls in closed queueing networks with a fixed number of jobs. Each time a server becomes available, the controller has (limited) flexibility in choosing the buffer from which to serve a job. If no jobs are available at any compatible buffer, the server idles. If the job is served, it relocates to a ``destination'' buffer. We study how to maximize throughput in steady state via a large deviations analysis.

We propose a family of simple state-dependent policies called Scaled MaxWeight (SMW) policies that dynamically manage the distribution of jobs in the network. We prove that under a complete resource pooling condition (analogous to the condition in Hall's marriage theorem), any SMW policy leads to exponential decay of throughput-loss probability as the number of jobs scales to infinity. Further, there is an SMW policy that achieves the \textbf{optimal} loss exponent among all scheduling policies, and we analytically specify this policy in terms of the service rates and routing probabilities. The optimal SMW policy maintains high job levels adjacent to structurally under-supplied servers.
Our methodology also applies to the open network setting and leads to exponent-optimal policies.

\medskip

\end{abstract}


\section{Introduction}\label{sec:intro}

The study of the control of closed queueing networks has a long history \citep[see, e.g.,][]{harrison1990scheduling}, and has seen a recent resurgence of activity motivated by applications such as shared transportation systems \citep[see, e.g.,][]{waserhole2016pricing,banerjee2016pricing,braverman2016empty}. The hallmark of such systems is that serving a job causes it to relocate. The system operator makes tactical control decisions with the aim of maximizing longer-term system performance, which necessitates that the operator \emph{manage the distribution of the jobs} to avoid server starvation throughout the network. In this paper, we focus on \emph{dynamic scheduling} of a closed queueing network given \emph{limited flexibility}, i.e., when a server becomes available, which compatible (e.g., nearby) buffer should it serve?  
A central challenge in such systems is that of distributional mismatch between the supply and demand of jobs: to keep a server busy, there has to be an available job at a compatible buffer when the server becomes available. There are two sources of distributional supply-demand asymmetry: \emph{structural imbalance} (some buffers may have a tendency to have a systematic net inflow, or outflow, of jobs) and \emph{stochasticity}.
Many previous works have studied control decisions made in a \emph{state-independent} manner where structural imbalance is handled by solving the fluid limit problem which arises as the number of jobs $K$ is taken to $\infty$. However, this approach fails to react to stochasticity leading to optimality gap (fraction of demand lost) which shrinks to zero only (slowly) as $1/K$ \citep{banerjee2016pricing} as $K$ grows if service rates and routing probabilities are \emph{exactly} known, and non-vanishing optimality gap as $K \to \infty$ if these parameters are not perfectly known (see Proposition~\ref{prop:state_ind_no_exp} in Section~\ref{subsec:neg_result_state_independent}).

In this paper we propose simple and practical \textit{state-dependent} scheduling policies which automatically handle both structural imbalance and stochasticity, and come with strong performance guarantees. We focus on system parameters satisfying the Complete Resource Pooling condition, which ensures that in the absence of stochasticity (i.e., in the fluid limit), all the servers are always busy. The control problem remains non-trivial: all state-independent policies provide unsatisfactory performance as summarized above. 
We provide a very simple ``maximum weight'' (MaxWeight) control policy which does not use system parameters and achieves optimality gap (loss) which decays \emph{exponentially} in $K$. This result motivates the large deviations question: \emph{Which policy maximizes the loss exponent?} We propose a natural family of Scaled MaxWeight (SMW) policies generalizing MaxWeight, and show that all SMW policies achieve exponentially small loss. We then prove the surprising result that there is always an SMW policy which is \emph{exponent-optimal} among all dynamic scheduling policies, and characterize how the parameters of the optimal SMW policy are determined by the system parameters.

{\bf Informal description of our model.}
In our model, the system consists of a network with two sets of nodes, namely, buffers and servers. A fixed number of jobs circulate among the buffers. We use the service token setup to model the service and routing process \citep[see, e.g.,][]{tsitsiklis2013power}: Service tokens arrive stochastically at servers with buffer destinations, at some time-invariant rates.
For each server, a subset of the buffers are \emph{compatible} with it, and the system dynamically decides from which compatible buffer to serve a job using the incoming service token. Thus, compatibilities capture the \emph{limited flexibility} available to the system.
After a job is scheduled to a service token, it relocates to the destination of the service token. Jobs do not enter or leave the system.
The system's goal is to waste as few service tokens as possible in steady state, i.e., to maximize throughput.

The controller's challenge is that of managing the distribution of the $K$ jobs to ensure the continued availability of jobs throughout the network. To obtain tight characterizations, we consider the asymptotic regime where the number of jobs in the system $K$ goes to infinity, and perform a large deviations analysis.

\subsection{Main contributions}
We show that a simple and practical MaxWeight scheduling policy effectively manages the distribution of jobs in the network, leading to a fraction of throughput lost that decays exponentially fast in $K$. Each time a service token arrives, MaxWeight simply uses it to serve a job from the compatible buffer which currently has the largest number of jobs. In particular, MaxWeight requires no knowledge of service token arrival rates. 

This finding motivates a thorough \emph{large deviations analysis} which yields surprisingly elegant results. As a function of system primitives, we derive a large deviations rate-optimal scheduling policy that minimizes throughput-loss. Our optimal policy is a close cousin of MaxWeight and its parameters depend in a natural way on service token arrival rates. Our contribution is threefold:
\begin{compactenum}[label=\arabic*.,leftmargin=*]
	\item {\bf A family of simple policies.}
	We propose a family of state-dependent scheduling policies called Scaled MaxWeight (SMW) policies, and prove that all of them guarantee exponential decay of throughput-loss probability under the CRP condition. An SMW policy is parameterized by a vector of scaling factors, one for each buffer; each service token serves a job from the compatible buffer with the largest scaled number of jobs. SMW policies are simple, explicit and promising for practical applications.
	\item {\bf The value of (intelligent) state-dependent control.} We show (Proposition \ref{prop:state_ind_no_exp}) that no state-independent scheduling policy can achieve loss which decays exponentially in $K$, and that if service token arrival rates are not perfectly known, then  the loss of a state-independent policy (generically) does not vanish as $K \to \infty$. 	Our SMW policies provide vastly superior performance: even the naive unscaled (``vanilla'') MaxWeight scheduling policy requiring no knowledge of service token arrival rates achieves loss which decays exponentially in $K$.
	\item {\bf Exponent-optimal policy and qualitative insights.}
	For general network structures, we obtain an explicit specification for the optimal scaling factors for SMW based on compatibilities and service token arrival rates. Further, we obtain the surprising finding that the optimal SMW policy is, in fact, {\em exponent-optimal} among all state-dependent policies (Theorem \ref{thm:main_tight}). A key ingredient of this result is that SMW policies satisfy the \textit{critical subset} property: for each SMW policy, there is a corresponding (fluid) equilibrium state, and for this state there are ``critical'' subsets of servers that are most vulnerable to the depletion of jobs in compatible buffers. Each SMW policy simultaneously ``protects'' all critical subsets maximally by maintaining high job levels near structurally under-supplied servers. 
\end{compactenum}
\medskip

{\bf Technical contributions.} To the best of our knowledge, we are the first to perform a large deviations analysis under CRP, leading to the challenging problem of deriving an exponent optimal control.
One key difficulty in the mathematical analysis is the necessity to deal with a multi-dimensional system even in the limit. Usually CRP renders the control problem ``easy'' because it leads to the ``collapse'' of the system state to a lower dimensional space in the heavy traffic limit, as in many existing works that establish the asymptotic optimality of a certain policy in minimizing the workload of a queueing system. In contrast, in our setting, the limit system remains $m$-dimensional, where $m$ is the number of buffers.
A second key challenge we face is that the ideal state for the system is a priori unknown, making it unclear how to define a Lyapunov function. We overcome these difficulties via a novel approach. We construct a  \emph{policy-specific} Lyapunov function to facilitate a sharp large deviations analysis of a given SMW policy leveraging the machinery of \cite{venkataramanan2013queue}. The analysis applies to general network structures, and reveals that the SMW policy maximally protects all the ``critical subsets'' of servers. We deduce the existence of an exponent optimal SMW policy, and characterize its scaling factors in terms of service token arrival rates. Happily, the fluid equilibrium for this optimal policy is revealed as the ideal state.

In addition to closed networks, our technical machinery can also be applied to the open network setting. To demonstrate the versatility of our methodology, in Section \ref{sec:open_network} we study an open queueing network with limited flexibility and finite buffer space, where the system operator needs to (i) choose buffer sizes beforehand subject to an upper bound on the sum of buffer sizes, and (ii) make dynamic scheduling decisions, in order to minimize the probability of buffer overflow.
Using a similar analysis as in the closed network setting, we design
an exponent-optimal policy.

\subsection{Applications}\label{subsec:intro-application}
Our main model and analysis can serve as a building block towards studying various applications. 

\textbf{Shared transportation systems.}
Shared transportation platforms such as those for ride-hailing and bikesharing make assignment control decisions under limited flexibility to manage the distribution of supply. In these applications, the nodes in our model correspond to geographical locations,\footnote{The set of buffers and servers are replicas of each other in these applications.} while jobs and service tokens correspond to vehicles and customers, respectively. The assignment control in ride-hailing takes the form of dispatch, i.e., the platform can decide where (near the demand's origin) to dispatch a car from. Bikesharing platforms can execute assignment control by suggesting to the customer where (near the customer's origin or destination) to pick up (or drop off) a bike.\footnote{For example, the Bike Angels program of CitiBike implicitly makes these suggestions to members by awarding ``points for taking bikes from crowded stations and bringing them to empty ones or stations expected to soon become empty''. Notice the resemblance to a MaxWeight approach. A live map of point awards is shown to customers.} 

\textbf{Quality-of-service with finite buffers.}
An important metric in queueing networks is quality-of-service (QoS). When queueing space is limited, QoS can be measured by the buffer overflow probability.
In many applications such as make-to-order production systems \citep{shi2015process}, the total buffer size of multiple queues is bounded by a physical constraint (e.g., the size of the warehouse), but it is a priori unclear how to divide the buffer space among different queues to maximize QoS.
We present an exponent optimal buffer sizing and dynamic scheduling policy for this problem in Section \ref{sec:open_network}.

\subsection{Literature review}
\textbf{MaxWeight scheduling.}
MaxWeight is a simple scheduling policy in constrained queueing networks which (roughly speaking) chooses the feasible control decision that serves the queues with largest total weight (e.g. queue length, waiting time, etc.), at each time.
MaxWeight scheduling has been shown to exhibit good performance in various settings (see, e.g., \citealt{tassiulas1992stability,dai2005maximum,stolyar2004maxweight,eryilmaz2012asymptotically,maguluri2016heavy}), including by \cite{shi2015process} who study an open one-hop network version of our setting.
In contrast, we find that MaxWeight achieves a suboptimal exponent in our closed network setting.

\textbf{Large deviations in queueing systems.}
There is a large literature on characterizing the probability of building up long queues in \emph{open} queueing networks, including controlled \citep[see, e.g.,][]{stolyar2001largest,stolyar2003control} and uncontrolled \citep[see, e.g.,][]{majewski2008large,blanchet2013optimal} networks.
The work closest to ours is that of \cite{venkataramanan2013queue}, who established the relationship between Lyapunov functions and buffer overflow probability for open queueing networks.
The key difficulty in extending the Lyapunov approach to closed queueing networks is the lack of a natural reference state where the Lyapunov function equals to $0$ (in an open queueing network the reference state is simply $\mathbf{0}$).
It turns out that as we optimize the MaxWeight parameters we are also solving for the best reference state.

\textbf{Applications: shared transportation systems.}
\cite{ozkan2016dynamic} studied revenue-maximizing state-independent assignment control by solving a minimum cost flow problem in the fluid limit.
\cite{braverman2016empty} modeled the system by a closed queueing network and derived the optimal static routing policy that sends empty vehicles to under-supplied locations.
\cite{banerjee2016pricing} adopted the Gordon-Newell closed queueing network model and considered static pricing/repositioning/matching policies that maximizes throughput, welfare or revenue.
In contrast to our work, which studies state-dependent control, these works consider static control that completely relies on system parameters.
In terms of convergence rate to the fluid-based solution, \cite{ozkan2016dynamic} did not study the convergence rate of their policy, \cite{braverman2016empty} observed from simulation an $O(1/\sqrt{\N})$ convergence rate as the number of supply units in the closed system $\N$ goes to infinity,\footnote{In the setting of \cite{braverman2016empty}, the loss probability can remain positive even as $K$ grows, in contrast with our setting where the loss probability can always be sent to $0$ because of our CRP condition under which the flows in the network can potentially be balanced.
The comparison of convergence rates is most meaningful if we restrict attention to instances in their setting where the loss probability goes to zero as $K$ grows.}
while \cite{banerjee2016pricing} showed finite system bounds with an $O(1/{\N})$ convergence rate as $\N\to\infty$ in the absence of service times and an $O(1/\sqrt{\N})$ convergence rate with service times. All these works propose static policies, and we show that no static policy can achieve exponentially small loss.
In contrast, under the CRP condition, we obtain exponentially small loss in $K$, and further obtain the optimal exponent. 

Our approach of studying control while initially ignoring travel delays is mirrored in several papers in this literature, starting with \cite{waserhole2016pricing}. The main model in \cite{banerjee2016pricing} ignores delays, and the paper subsequently shows that all its findings are robust to that assumption. Similarly, subsequent to the present paper, \cite{balseiro2019dynamic} study the control of (large) networks of circulating resources by ignoring travel delays and then show robustness of their results to delays. 

\textbf{Online stochastic bipartite matching.}
There is a related stream of research on online stochastic bipartite matching, see, e.g., \cite{caldentey2009fcfs,adan2012loss,buyic2015approximate,mairesse2016stability}.
Different types of supplies and demands arrive over time, and the system manager matches supplies with demands of compatible types using a specific matching policy, and then discharges the matched pairs from the system.
Our work is different in that we study a \emph{closed} system where supply units never enter or leave the system. 

\textbf{Other related work.} \cite{jordan1995principles,desir2016sparse,shi2015process} and others study how process flexibility can facilitate improved performance, analogous to our use of dispatch control to improve demand fulfillment.
Along similar lines, network revenue management is a classical dynamic resource allocation problem, see, e.g., \citet{gallego1994optimal,talluri2006theory}, and recent works, e.g., \citet{jasin2012re,bumpensanti2018re}.
Different types of demands arrive over time, and a centralized decision is made at each arrival.
Again, each of these settings is ``open'' in that each service token or supply unit can be used only once, in contrast to our closed setting.

\subsection{Organization of the paper}
The remainder of our paper is organized as follows.
In Section \ref{sec:model} we introduce the basic notation and formally describe our baseline model together with the performance metric.
In Section \ref{sec:smw} we introduce the family of Scaled MaxWeight policies.
In Section \ref{sec:main_results} we present our main theoretical result, i.e., that there is an exponent optimal SMW policy for any set of primitives satisfying our main assumption.
In Section \ref{sec:smw_analysis} we prove the exponent optimality of SMW policies.
In Section \ref{sec:open_network} we apply our technical machinery to an open network setting.
We conclude in Section \ref{sec:conclu}.

\noindent\textbf{Notation.} We use $\e_i$ to denote the $i$-th unit vector, and $\mathbf{1}$ the all-$1$ vector.
The dimensions of the vectors will be clear from the context.
For a finite index set $A$, define $\mathbf{1}_{A}\triangleq\sum_{i\in A}\e_i$.
For a set $\Omega$ in Euclidean space $\mathbb{R}^n$, denote its relative interior by $\textup{relint}(\Omega)$.
For event $C$, we define the indicator random variable $\ind\{C\}$\yka{Better to use  $\mathbb{I}$  instead, to clearly distinguish from $\mathbf{1}$.} to equal $1$ when $C$ is true, else $0$.
All vectors are column vectors if not specified otherwise.

\section{The Model and Preliminaries}\label{sec:model}

\subsection{Basic setting}
We study the dynamic scheduling problem in closed queueing networks. We consider an infinite-horizon continuous-time model, with a fixed number $K$ of identical \emph{jobs} that circulate in the network. Formally, we consider a sequence of systems indexed by $K \in \mathbb{Z}_+$.

\textbf{The compatibility graph.} The compatibility structure for scheduling is described by a bipartite \emph{compatibility graph} $G= (V_B\cup V_S,E)$, where the $K$ jobs are distributed over the \emph{buffers} $V_B$, and the set of \emph{servers} is $V_S$.
We add a prime symbol to the indices of servers (in $V_S$) to distinguish between the two.
Let $m\triangleq |V_B|$ and $n\triangleq |V_S|\in \mathbb{Z}_+$ be the number of buffers and servers, respectively.
Each edge $(i,j')\in E$ represents a compatible pair of buffer and server, i.e., a job currently stationed at $i \in V_B$ can be served by server $j' \in V_S$. See Figure~\ref{fig:bipartite_matching} for an illustration. We denote the neighborhood of a buffer $i\in V_B$ (resp. server $j'\in V_S$) in $G$ as $\partial(i)\subseteq V_S$ (resp. $\partial(j') \subseteq V_B$); thus, for a buffer $i$, its compatible servers are given by $\partial(i)=\{j'\in V_S|(i,j')\in E\}$, and similarly for each server.
Moreover, for any set of buffers $A\subseteq V_B$, we also use $\partial(A)$ to denote its server neighborhood (and vice versa).
\begin{figure}[!t]
	\centering
	\includegraphics[width=0.45\columnwidth]{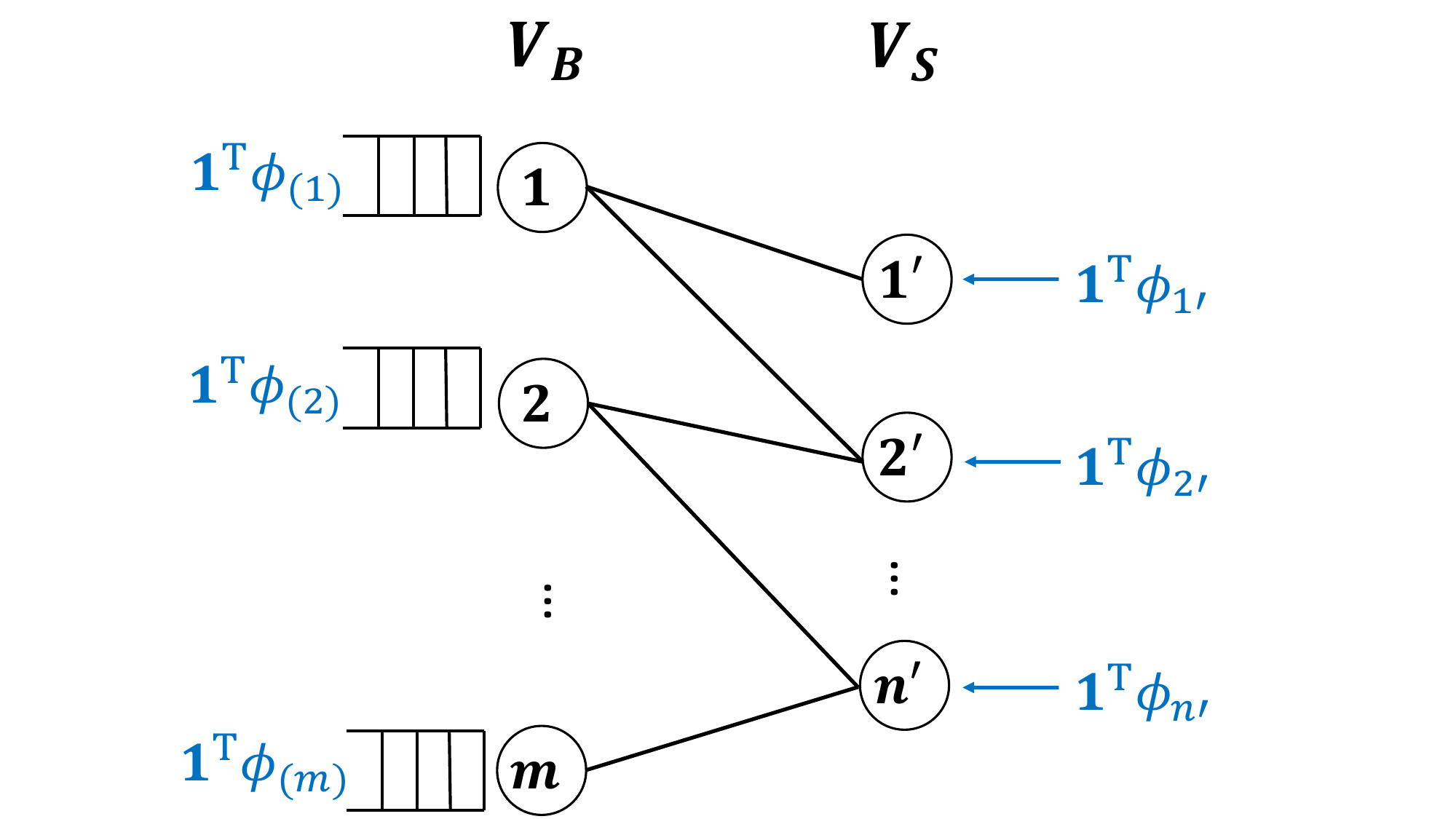}
	\caption{The bipartite (scheduling) compatibility graph:
	On the left are buffers $i \in V_B$, and on the right are servers $j' \in V_S$. The edges entering a server $j'$ encode compatible buffers that can be served by $j'$.
	The (normalized) serviced rate of server $j'$ is $\mathbf{1}^{\T}\phi_{j'}$.\yka{In the picture on the right, the prime signs are not appearing properly in the subscript $\mathbf{1}^{\T}\hat{\phi}_{i'}$.} Assuming no server is idling, the (normalized) rate of arrival of jobs to $i$ is $\mathbf{1}^{\T}\phi_{(i)}$.}\label{fig:bipartite_matching}
\end{figure}

\textbf{Service tokens.} We use the service token setup to model the service process, see, e.g., \cite{tsitsiklis2013power}.
We denote the type of a service token as $(j',k)\in V_S\times V_B$, where $j'$ is its \emph{origin} server and $k$ is its \emph{destination} buffer.
In the $K$-th system, service tokens of each type $(j',k)$ arrive sequentially following independent Poisson processes with rates $K\phi_{j'k}$.
Without loss of generality, we assume that $\mathbf{1}^{\T} \bphi \mathbf{1}=1$.
We call $\bphi\in\mathbb{R}^{n\times m}$ the matrix of service token type probabilities.
We denote the $k$-th column of $\bphi$ as $\phi_{(k)}$, and the transpose of the $j'$-th row of $\bphi$ as $\phi_{j'}$.
Thus, the (normalized) service rate at server $j'$ is $\mathbf{1}^{\T}\phi_{j'}$, and, assuming all service tokens are utilized, the rate of jobs arriving at queue $k$ is $\mathbf{1}^{\T}\phi_{(k)}$.
We assume all servers have positive service rate $\mathbf{1}^{\T} \phi_{j'}$.

We use the term \emph{network} to refer to a given set of primitives: a compatibility graph $G$ and service token type distribution matrix $\bphi$. We make two mild assumptions on the network.

\begin{asm}[Connectedness]
\label{asm:connectivity}
A network $(G,\bphi)$ is \emph{connected} if for every ordered pair of distinct buffers $(k_0,i) \in V_B\times V_B$, $k_0 \neq i$, there is a finite sequence of service token types $(j_1',k_1),\cdots,(j_{\ell}',k_{\ell}=i)$ such that $\phi_{j_r' k_r}>0$ for all $r=1,\cdots,\ell$, and $k_{r-1} \in \partial(j_{r}')$ for all $r=1,\cdots,\ell$.
\end{asm}

Assumption \ref{asm:connectivity} requires that for every pair of buffers, there is a sequence of service token types with positive arrival rates and corresponding compatible buffers that would take a job from one buffer eventually to the other buffer.

We now observe that if the compatibility graph affords ample flexibility, specifically, if the destination for every service token type belongs to the compatible neighborhood of the origin, then the control problem is trivial.
\begin{prop}[Ample flexibility renders the control problem trivial]
\label{prop:NT-is-necessary}
Consider any network $(G, \bphi)$ which satisfies Assumption~\ref{asm:connectivity} and such that for all  $j' \in V_S$, $k \in V_B$ such  $\phi_{j'k}>0$ it holds that $k \in \partial(j')$.
Then for any $K \geq n \triangleq |V_S|$, there is a control policy which wastes an identically zero fraction of service tokens in the long run. Formally, there is a policy $U$ such that $\mathbb{P}^{\N,U} = 0$, for $\mathbb{P}^{\N,U}$ defined in \eqref{eq:lb_performance_measure} below.
\end{prop}
The reason is simple: we can ``reserve'' a job for each server $j' \in V_S$, and each reserved job will never leave the corresponding neighborhood $\partial(j')$, ensuring that no service token is ever wasted. The proof of Proposition~\ref{prop:NT-is-necessary} is in Appendix \ref{appen:sec:crp}.
Proposition~\ref{prop:NT-is-necessary} motivates the following assumption to ensure that the flexibility available is sufficiently limited that the scheduling control problem at hand is non-trivial.
\begin{asm}[Limited flexibility]
\label{asm:non_trivial}
A network $(G, \bphi)$ has \emph{limited flexibility} if there exists a server-buffer pair $j' \in V_S$ and $k \in V_B$ such that $k \notin \partial(j')$ and $\phi_{j'k}>0$, i.e., the destination $k$ for these service tokens is not a buffer compatible with their origin $j'$.
\end{asm}

\textbf{System state.}  For the $K$-th system, its state at any time is given by $\X^{K} $, an $m$-dimensional vector that tracks the number of jobs at each buffer. The state space of the $K$-th system is thus given by
$
\Omega_{K}\triangleq
\big \{\mathbf{x}\in \{0, 1, 2, \dots\}^m \, \big | \, \mathbf{1}^{\T} \mathbf{x}=\N\big \}.
$
Note that the normalized state $\frac{1}{K}\X^{K}$ lies in the $m$-probability simplex $\Omega = \{\mathbf{x}\in\mathbb{R}^m|\mathbf{x}\geq 0,\mathbf{1}^{\T} \mathbf{x}=1\}$.  We use $\X^{\N}(0)$ to denote the initial state.

\subsection{Optimal scheduling control}
Given the above setting, the problem we want to study is how to design scheduling policies which maximize the system throughput.
For fixed $\N$, this problem can be formulated as an average cost Markov decision process (MDP) on a finite (albeit, very large) state space, and is thus known to admit a stationary optimal policy (i.e., where the scheduling rule at any time only depends on the current system state $\X^{K}$; see Proposition 5.1.3 in \citealt{bertsekas1995dynamic}). Furthermore, the MDP satisfies the accessibility condition because of Assumption \ref{asm:connectivity}, therefore the optimal cost is independent of the initial state (see \citealt{bertsekas1995dynamic}).

\textbf{Scheduling policies.}
Upon the arrival of a service token of type $(j',k)$, the system must immediately serve a job from a compatible buffer of $j'$;
subsequently, the job relocates to the destination buffer $k$.
If no job is available at any compatible buffer of $j'$, then the service token is wasted.
Let $\mathcal{U}^{K}$ be the set of stationary policies for the $\N$-th system.
A scheduling policy $U\in\mathcal{U}$ consists of, for each $j'\in V_S$, $k\in V_B$, a sequence of mappings $\big(U^{K}\in\mathcal{U}^{K}\big)_{\N=1}^\infty$,  which map the current queue-length vector $\X^{K}$ and service token type $(j',k)$ to $U^{K}[\X^{K}](j',k)\in \partial(j')\cup\{\emptyset\}$.
Here $U^{K}[\X^{K}](j',k)=i$ means given the current state $\X^{K}$, we serve a job from $i \in \partial(j')$ with server $j'$ and the job relocates to $k$, and $U^{K}[\X^{K}](j',k)=\emptyset$ means that the system does not serve jobs with type $(j',k)$ service tokens and hence any such service token is wasted.
When $\X^{K}_i=0$ for all $i\in\partial(j')$, this forces $U^{K}[\X^{K}](j',k)=\emptyset$ since there is no job at buffers compatible to $j'$.
For simplicity of notation, we refer to the policies by $U$ instead of $U^{K}$.

\textbf{System evolution.} Let $t_{r}$ be the $r$-th service token arrival epoch after time $0$.
Denote the state of the system just before $t_r$ by $\X^{K}(t_r^{-})$ (the initial state is $\X^{\N}(0)$); note that this incorporates the state change due to serving with the $(r-1)$-th service token arrival for $r>1$.
Now suppose the system uses a scheduling policy $U$, and the $r$-th service token arrival has origin server $o[r]$ with destination buffer $d[r]$ (sampled from service token type distribution $\bphi$).
Let $S[r] \triangleq U^{K}[\X^{K}(t_r^-)](o[r],d[r])$ be the chosen buffer (potentially $\emptyset$). Then, formally, the system state updates as per
\begin{equation*}
\X^{\N}(t_r)
	\triangleq
\left\{
\begin{array}{ll}
\X^{\N}(t_r^-) - \e_{S[r]} + \e_{d[r]}&\text{if $S[r]\in V_B\, ,$}\\
\X^{\N}(t_r^-)&\text{if $S[r]=\emptyset\, .$}
\end{array}
\right.
\end{equation*}

\textbf{Performance measure.}
The platform's goal is to find a scheduling policy that maximizes system throughput in steady state, which is equivalent to minimizing the \emph{long-run average probability of wasting service tokens}.

Formally, for $U\in\mathcal{U}$ we define
\begin{equation}\label{eq:lb_performance_measure}
{\bbp}^{\N,U}
\triangleq
\lim_{T\to\infty}\frac{1}{T}\ 
\mathbb{E}\left(
\sum_{r=1}^{T}
\ind\left\{U^{\N}[\X^{\N,U}(t_r^{-})](o[r],d[r])= \emptyset\right\}
\right)\, ,
\end{equation}

The limit exists because the MDP has a finite state space. The exact value of \eqref{eq:lb_performance_measure} for fixed $\N$ is challenging to study. To this end, the main performance measure of interest in this work is the decay rate of ${\bbp}^{\N,U}$ as $\N\to\infty$:
\begin{equation}\label{eq:lb_measure_rate}
\gamma(U) \triangleq -\limsup_{\N\to\infty}\frac{1}{\N}\log {\bbp}^{\N,U}\, ,
\end{equation}
For brevity, we henceforth refer to it as the {\em throughput-loss exponent}. 

\subsection{The Complete Resource Pooling (CRP) condition} 
\label{subsec:CRP}
We now make a few additional definitions to allow us to state our main assumption. 
We say that a subset of servers $J \subsetneq V_S$ has \emph{limited flexibility} if there is some server $j' \in J$ and buffer $k \notin \partial(J)$ such that $\phi_{j' k } > 0$. (Informally, there is a service token type which requires jobs to leave the neighborhood of $J$.)
We denote the set of limited-flexibility subsets by $\mathcal{J}$.
Assumption~\ref{asm:non_trivial} guarantees that there is at least one non-trivial singleton $J$ and hence that $\mathcal{J}\neq\emptyset$.
Observe that $J$ has limited flexibility if and only if 
\begin{align}
\label{eq:net-demand}
\mu_J \triangleq \sum_{j'\in J}\sum_{k\notin \partial(J)}\phi_{j'k}>0 \, .
\end{align}
We call $\mu_J$ the \emph{net service rate} of $J$, since it captures the probability that a service token arrival has origin in $J$ and destination outside $\partial(J)$ (and hence requires a job to leave $\partial(J)$). Similarly, we define the (optimistic) \emph{net arrival rate} to $J$ as
\begin{align}
    \lambda_J \triangleq \sum_{j'\notin J}\sum_{k\in \partial(J)}\phi_{j'k} \, .
    \label{eq:net-supply}
\end{align}
Informally, $\lambda_J$ is the probability that a service token arrival is such that it can (depending on the scheduling decision) cause a job to enter $\partial(J)$.

The following is the main assumption of this paper.

\begin{asm}[Complete Resource Pooling]
\label{asm:strict_hall}
We assume that for all subsets of servers $J$ with limited flexibility (i.e., $J\subsetneq V_S$ with positive net service rate $\mu_J > 0$) we have that $\lambda_J > \mu_J$, where the net arrival rate $\lambda_J$ was defined in \eqref{eq:net-supply}, and the net service rate $\mu_J$ was defined in \eqref{eq:net-demand}. 
\end{asm}

The intuition behind this assumption is simple: it assumes the system is ``balanceable'' in that for each subset $J\subsetneq V_S$ of servers, jobs arrive sufficiently fast at neighboring buffers to utilize all the service tokens arriving to $J$, on average.
The control problem under CRP is non-trivial: In Section~\ref{subsec:neg_result_state_independent} we will show that all state-independent policies perform inadequately.

We show that Assumption~\ref{asm:strict_hall} is necessary in order to obtain exponentially small loss.

\begin{prop}\label{prop:hall_is_necessary}
For any $(G,\bphi)$ such that Assumption \ref{asm:strict_hall} is violated, it holds that for any policy $U$, the throughput-loss probability does not decay exponentially, i.e., $\gamma(U) = 0$ where $\gamma(U)$ is defined in \eqref{eq:lb_measure_rate}.
\end{prop}

In other words, if Assumption \ref{asm:strict_hall} is violated, this means the system has significant distributional imbalance between the demand for service and the capacity of the system, and throughput loss is unavoidable. The proof of Proposition~\ref{prop:hall_is_necessary} is in Appendix \ref{appen:sec:crp}. 

\subsection{Sample path large deviation principle}\label{subsec:sample_path_ldp}

Our main theoretical result is the culmination of a sharp \emph{large deviations} analysis, characterizing the best possible throughput-loss exponent. We provide a brief introduction to classical large deviations theory in this subsection.

For each fixed $\N\in\mathbb{Z}_{+}$ and $T \in (0, \infty)$, define a {\em scaled sample path} of accumulated service token arrivals $\fA^{\N}(\cdot)\in (L^{\infty}[0,T])^{n\times m}$\yka{changed $n^2$ to $nm$} as follows.\footnote{Here $L^{\infty}[0,T]$ denotes the space of bounded functions on $[0,T]$ equipped with the supremum norm.}\yka{Reviewer 2 (the technical and good one) writes ``define properly what $L^\infty[0,T]$ means''.}
Let $\{A_{j'k}^K(\cdot)\}_{j'\in V_S,k\in V_B}$ be independent Poisson processes where $A_{j'k}^K(\cdot)$ has rate $K{\phi}_{j'k}$.
Let
\begin{equation}
	\fA^{K}_{j'k}(t)
	\triangleq
	\frac{1}{\N}\bA_{j'k}^K(t)\qquad \forall t\in [0,T] \, .
\label{eq:demand-sample-path}
\end{equation}
Let $\mu_{\N}$ be the law of $\fA^{\N}(\cdot)$ in $(L^{\infty}[0,T])^{n\times m}$.
For all $\bof \in \mathbb{R}_+^{n\times m}$, let
\begin{equation}\label{eq:kl_divergence}
	\Lambda^*(\bof)
	\triangleq
	\left\{
	\begin{array}{ll}
	\sum_{j'\in V_S} \sum_{k\in V_B} \left(f_{j'k} \log\frac{f_{j'k}}{{\phi}_{j'k}}
	-f_{j'k}
	+
	{\phi}_{j'k}
	\right)
	& \text{if } \bof > \bzero\, ,\\
	\infty & \text{otherwise\,.}
	\end{array}
	\right.
\end{equation}
For any set $\Gamma$, let $\bar{\Gamma}$ denote its closure, and $\Gamma^o$ denote its interior.
Below is the sample path large deviation principle (also known as Mogulskii's Theorem, see \citealt{dembo1998large}):
\begin{fact}\label{fact:sample_path_ldp}
	For measures $\{\mu_{\N}\}$ defined above, and any arbitrary measurable set $\Gamma \subseteq (L^{\infty}[0,T])^{n\times m}$, we have
	\begin{equation}\label{eq:ldp}
	-\inf_{\fA\in \Gamma^o}I_T(\fA)
	\leq
		\liminf_{\N\to\infty}\frac{1}{\N}\log \mu_{\N}(\Gamma)
		\leq
		\limsup_{\N\to\infty}\frac{1}{\N}\log \mu_{\N}(\Gamma)
		\leq
		-\inf_{\fA\in \bar{\Gamma}}I_T(\fA)\, ,
	\end{equation}
	where the rate function\footnote{Since absolutely continuous functions are differentiable almost everywhere, the rate function is well-defined.} is:
	\begin{equation}\label{eq:ld_rate}
		I_T(\fA) \triangleq \left\{
		\begin{array}{ll}
		\int_{0}^{T}\Lambda^*\left(\dot{\fA}(t)\right)dt &
		\text{ if }\fA(\cdot)\in \textup{AC}[0,T],\ \fA(0)=\mathbf{0}\, ,\\
		\infty &\text{ otherwise\,.}
		\end{array}
		\right.
	\end{equation}
	Here $\textup{AC}[0,T]$ is the space of absolutely continuous functions on $[0,T]$, and $\dot{\fA}(t)$ is the derivative of $\fA$ at time $t$ when the derivative exists.
\end{fact}

Informally, this fact says the following. (Suppose the leftmost term and rightmost term in \eqref{eq:ldp} are equal.) The probability exponent (with respect to $K$) for the event $\Gamma$ is equal to the exponent for the most likely fluid sample path (a limit of scaled sample paths, see Section~\ref{subsec:FSP_FL}) of demand $\fA$ such that the event occurs. The exponent for $\fA$ is the \emph{time integral of the exponent for its time derivative}, {and the latter is given by the function \eqref{eq:kl_divergence} where the summand is the large deviations exponent of a (sequence of) Poisson random variable(s) with mean ${\phi}_{j'k}$}. 

In the present work, the relevant $\Gamma$ will be the throughput loss event. The reason the sample paths of accumulated service token arrivals fully determine whether this event occurs is because given any deterministic policy (as the policies we propose will be), the arrival process $\bA(\cdot)$ and the initial configuration $\bX(0)$ uniquely determine the evolution of the system state $\bX(\cdot)$, and hence determine throughput loss. The key will be to understand the most likely sample paths of the arrival process which lead to throughput loss. Our converse (impossibility) bound on the exponent will be established by constructing a fluid sample path of service token arrivals that \emph{always} leads to throughput loss regardless of the policy.

\section{Scaled MaxWeight Policies}\label{sec:smw}

The traditional MaxWeight policy is a celebrated approach to scheduling which has been effectively deployed in many applications such as cloud computing, communication networks, traffic management, etc., \citep[see, e.g.,][]{tassiulas1992stability,maguluri2012stochastic}. MaxWeight (hereafter referred to as vanilla MaxWeight) allocates the service capacity to the queue(s) with largest ``weight'' (where weight can be any relevant parameter such as queue length, waiting time, etc.). 
In our setting, vanilla MaxWeight would correspond to serving the compatible buffer with longest queue (with appropriate tie-breaking rules).

Besides its simplicity, one reason for the popularity of MaxWeight is that it is known to be asymptotically optimal in many problem settings (e.g., see \citealt{stolyar2003control,stolyar2004maxweight,shi2015process,maguluri2016heavy}).
In our setting too, we will find that vanilla MaxWeight is asymptotically optimal. In fact, we will show that it achieves an exponentially small loss. However, we will find that, in general, vanilla MaxWeight does not achieve the largest possible loss exponent.
Suboptimality of the exponent prompts us to consider alternate control policies.

We generalize vanilla MaxWeight by attaching a positive scaling parameter $\w_i$ to each buffer $i\in V_B$, and serve the compatible buffer with largest \textit{scaled} queue length $\X_i/\w_i$.
Without loss of generality, we normalize $\balpha$ s.t. 	$\mathbf{1}^{\T}\balpha=1$, or equivalently, $\balpha\in\text{relint}(\Omega).$
We call this family of policies \textit{Scaled MaxWeight (\textup{SMW}) policies}, and use SMW($\balpha$) to denote SMW with parameter $\balpha$.

The formal definition of SMW is as follows.
\begin{defn}[Scaled MaxWeight SMW($\balpha$)]
Fix $\balpha \in \textup{relint}(\Omega)$, i.e., $\balpha \in \mathbb{R}^m$ such that $\alpha_i > 0 \ \forall i \in V_B$ and $\sum_{i \in V_B} \alpha_i =1$.	Given system state $\X(t_r^{-})$ just before the $r$-th service token arrival and for service tokens arriving at server $j'$, \textup{SMW}($\balpha$) serves
	\begin{equation*}
		\textup{argmax}_{i\in \partial(j')}\frac{\X_i(t_r^{-})}{\alpha_i}
	\end{equation*}
	if $\max_{i\in \partial(j')}\frac{\X_i(t_r^{-})}{\alpha_i}>0$; otherwise the service token is wasted.
	(If there are ties when determining the argmax, it serves the buffer with highest index.\footnote{Our analysis and results are unchanged if any other deterministic tie-breaking rule is employed instead.})
\end{defn}

As may be expected, SMW policies tend to equalize the scaled queue lengths if CRP holds. The following fact is formalized later in Proposition~\ref{prop:smw-resting-state} in Section~\ref{sec:smw_analysis}.
\begin{rem}[Resting state under \textup{SMW}($\balpha$)]
\label{rem:SMW-converges-to-w}
If Assumptions~\ref{asm:connectivity}, \ref{asm:non_trivial} and \ref{asm:strict_hall} hold then for any $\balpha \in \textup{relint}(\Omega)$, the SMW$(\balpha)$ policy has a ``resting state'' $\balpha$: 
Specifically, consider using SMW$(\balpha)$ on a sequence of systems indexed by the number of jobs $K$. Then there exists $T_0 = T_0(\alpha) >0$ which does not depend on $K$, such that for any $T>T_0$,
\begin{align*}
    \limsup_{K\to\infty}\left(
    \max_{\bX^{K}(0)\in\Omega_{K}}
    \big \lVert \tfrac{1}{K} \bX^{K, \balpha}(T) - \balpha \big \rVert_2
    \right)
    =\ 0\quad \textup{almost surely}\, ,
\end{align*}
where $\bX^{K, \balpha}(T)$ is the state of the $K$-th system at time $T$. 
\end{rem}

\section{Main Result}\label{sec:main_results}

In this section we present our main result, which says that for any network such that CRP holds: (i) All Scaled Maxweight (SMW) policies yield exponential decay of throughput loss in the number of jobs $K$, with an exponent which we explicitly specify. (ii) For scaling parameter vector $\balpha$ which maximizes the exponent among SMW policies, the SMW$(\balpha)$ policy is exponent optimal among all possible policies. 
In sharp contrast, we show in Section~\ref{subsec:neg_result_state_independent} that 
that no state-independent scheduling policy can achieve loss which decays exponentially in $K$, and moreover that if service token arrival rates are not perfectly known, then  the loss of a state-independent policy (generically) does not vanish as $K \to \infty$. 

Recall from Section~\ref{subsec:CRP} the set of subsets of servers with limited flexibility
\begin{equation}\label{eq:set_drainable}
	\mathcal{J} = 
	\left\{
	J\subsetneq V_S:
	\sum_{j'\in J}\sum_{k\notin \partial(J)}\phi_{j'k}>0
	\right\}.
\end{equation}
The following is our main result.

\begin{thm}[\bf Main Result]\label{thm:main_tight}
For any network $(G, \bphi)$ satisfying Assumptions~\ref{asm:connectivity}, \ref{asm:non_trivial} and \ref{asm:strict_hall}, we have:
\begin{compactenum}[label=\arabic*.,leftmargin=*]
		\item \textbf{Exponentially small loss under any SMW policy}: For any $\balpha\in\textup{relint}(\Omega)$, \textup{SMW($\balpha$)} achieves exponential decay of the throughput-loss probability with exponent\,
		\footnote{Note that the argument of the logarithm has a strictly larger numerator than denominator for every $J \subsetneq V_S$ since Assumption \ref{asm:strict_hall} holds, implying that $\gamma(\balpha)$ is the minimum of finitely many positive numbers, and hence is positive.}
\begin{align}
\gamma(\balpha)
=
\min_{J\in\mathcal{J}} B_J
\log\left( \frac{\lambda_J}{\mu_J} \right)
>0\, ,
\label{eq:gamma_w}
		\end{align}
\begin{align*}
	\textup{where}\quad
	B_J\triangleq \mathbf{1}_{\partial(J)}^{\textup{T}}\balpha\, ,\qquad
	\lambda_J \triangleq \sum_{j'\notin J}\sum_{k\in \partial(J)}\phi_{j'k}\, ,\quad \textup{and} \quad
	\mu_J \triangleq \sum_{j'\in J}\sum_{k\notin \partial(J)}\phi_{j'k}\, .
\end{align*}
		\item \textbf{There is an exponent optimal SMW policy}: Under \emph{any} policy $U$, it must be that
\begin{equation}
\gamma(U) \leq \gammaopt \, ,
\qquad
\textup{where}\quad\gammaopt = \sup_{\balpha\in\textup{relint}(\Omega)} \gamma(\balpha) \, .
\end{equation}
\end{compactenum}
Thus, there is an \textup{SMW} policy that achieves an exponent arbitrarily close to the optimal one.
\end{thm}

The first part of the theorem states that for any SMW policy with $\balpha$ in the relative interior of $\Omega$, the policy achieves an explicitly specified positive throughput-loss exponent $\gamma(\balpha)$, i.e., the throughput-loss probability decays as $e^{-(\gamma(\balpha)-o(1))\N}$ as $\N\to\infty$.
The second part of the theorem provides a universal upper bound $\gammaopt$ on the exponent that \emph{any} policy can achieve, i.e., for any scheduling policy $U$, the throughput-loss probability is at least $e^{-(\gammaopt+o(1)) \N}$.
Crucially, $\gammaopt$ is identical to the supremum over $\w$ of $\gamma(\balpha)$. In other words, \emph{there is an (almost) exponent optimal \textup{SMW} policy}, and moreover, the scaling parameters for this policy can be obtained as the solution to the explicit problem: $\textup{maximize}_{\balpha \in \textup{relint}(\Omega)} \gamma(\balpha)$.

We note that Theorem~\ref{thm:main_tight} is qualitatively different from the numerous results showing near optimality of (vanilla) maximum weight matching in various open queueing network settings \citep[e.g.,][show that vanilla MaxWeight asymptotically minimizes workload in heavy-traffic in certain open queueing networks under the CRP condition]{stolyar2004maxweight,dai2008asymptotic}. 
Despite our objective (maximize throughput) being symmetric in all the $m$ queues, our result says that there is an optimal \emph{scaled} maximum weight policy, that is \emph{not} symmetric in the $m$ queues; rather, it is uses asymmetric scaling factors that optimally account for the network primitives. 

\medskip
{\bf Intuition for $\gamma(\balpha)$.} Consider the expression for $\gamma(\balpha)$ in \eqref{eq:gamma_w}. It is a minimum over subsets $J \in\mathcal{J}$ of servers of a certain ``robustness'' of the subset to throughput loss. For subset $J$, the robustness  of $\textup{SMW}(\balpha)$'s ability to utilize service tokens originating in $J$ is the product of two terms $B_J\times \log\left(\frac{\lambda_J}{\mu_J}
\right)$ (see Figure \ref{fig:exponent_visualize} for an illustration of the quantities involved):
\begin{compactitem}[leftmargin=*]
  \item ``Protection'' due to $\balpha$: At the resting point $\balpha$ (see Remark \ref{rem:SMW-converges-to-w}) of $\textup{SMW}(\balpha)$, the jobs at neighboring buffers is $B_J=\mathbf{1}^{\T}_{\partial(J)} \balpha$, and the larger that is, the more unlikely it is that the subset will be deprived of jobs.
  \item ``Inherent robustness'' arising from excess of arrival rate over service rate: The logarithmic term $\xi_J \triangleq \log(\lambda_J/\mu_J)$ captures the \emph{inherent robustness} of that subset is to being drained of jobs.
Recall that $\lambda_J$ is the (optimistic) net job arrivals coming in to $\partial(J)$, and that $\mu_J$ is the net service taking jobs out of $\partial(J)$. The larger the ratio $\lambda_J/\mu_J$, the more oversupplied and hence robust $J$ is. 
\end{compactitem}

\begin{figure}[ht]
	\centering
	\includegraphics[width=0.5\textwidth]{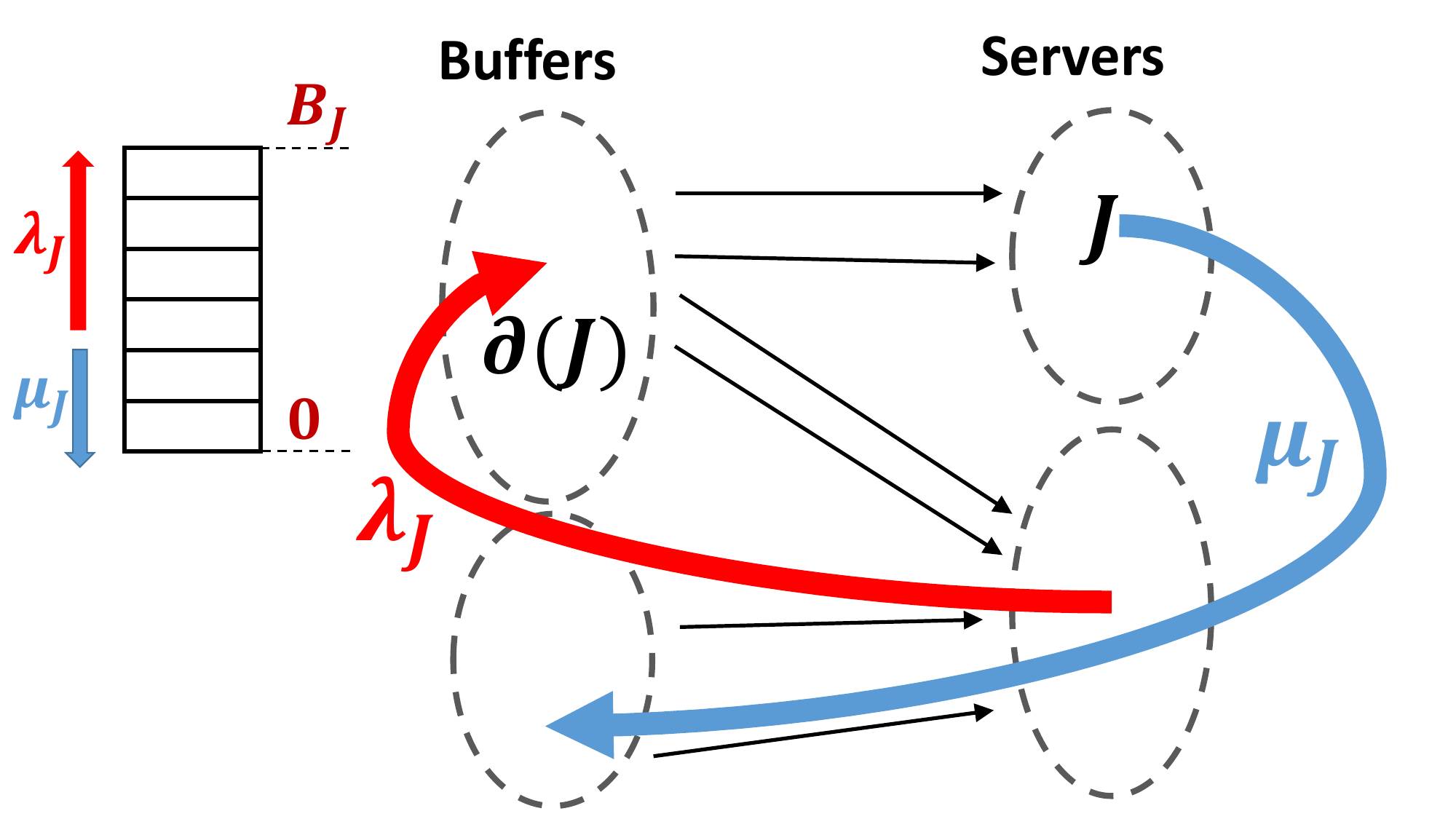}
	\caption{An illustration of the terms $B_J$, $\lambda_J$, and $\mu_J$ in Theorem \ref{thm:main_tight}. \yka{Move the queue on the right to the left beside $\partial(J)$. Scale it down to be the same height as $\partial(J)$. Edit/augment the caption as well if it improves clarity. } 
	}\label{fig:exponent_visualize}
\end{figure}

Remarkably, the expression for robustness of subset $J$ under SMW$(\balpha)$ is as large (i.e., as good) as the throughput-loss exponent for subset $J$ alone would be, with starting state $\balpha$, under a ``protect-$J$'' policy which \emph{exclusively protects $J$ at the expense of all other servers}. (Similar to standard buffer overflow probability calculations, the likelihood of the $KB_J$ jobs at $\partial (J)$ being depleted under a protect-$J$ policy is $\Theta((\lambda_J/\mu_J)^{-KB_J}) = \Theta( \exp(- K B_J \log (\lambda_J/\mu_J)))$. We then set $B_J$ to the starting scaled queue lengths at $\partial (J)$, i.e., $B_J = \mathbf{1}_{\partial (J)}^{\T} \balpha$, to establish the claim.) Thus,  Theorem~\ref{thm:main_tight} part 1 says that given the resting state $\balpha$, $\textup{SMW}(\balpha)$ achieves an exponent such that \emph{it suffers no loss from the need to protecting multiple subsets $J$ simultaneously}. Given this remarkable property, it is intuitive that the globally optimal exponent can be achieved via an SMW policy by choosing $\balpha$ suitably (part 2 of the theorem). \yka{Check for notation conflicts with $\mu_J, \lambda_J, B_J$.}

\medskip
{\bf Proof approach.} We establish Theorem~\ref{thm:main_tight} via a \emph{novel Lyapunov analysis for a closed queueing network.} A key technical challenge we face in our closed queueing network setting is that it is a priori unclear what the ideal state for the system is. This is in contrast to open queueing network settings in which the ideal state is typically the one in which all queues are empty, and the Lyapunov functions considered typically achieve their minimum at this state. We overcome the challenge of unknown ideal state via an innovative approach as follows: We define a \emph{policy-specific} Lyapunov function that achieves its minimum at the resting point of the SMW policy we are analyzing, and use this Lyapunov function to characterize its exponent $\gamma(\balpha)$. Moreover, given the optimal choice of $\balpha$, our tailored Lyapunov function corresponding to this choice of $\balpha$ helps us establish our converse result. In particular, the ideal state is finally revealed as a byproduct of our analysis to be equal to the optimal choice of $\balpha$. Our technical machinery may be broadly useful in deriving large-deviation optimal controls in settings where the appropriate target state is apriori unclear. Our analysis is described in Section \ref{sec:smw_analysis}.

\medskip
{\bf Knowledge requirements.} We remark that choosing the exponent optimal $\balpha$ requires exact knowledge of $\bphi$. However, if a noisy estimate of the service token type distribution is employed to choose $\balpha$ (by maximizing the exponent for the estimated distribution), the resulting SMW policy will nevertheless perform well: (i) it will achieve exponentially small loss (as long as the true $\bphi$ satisfies our assumptions), (ii) if the estimate of $\bphi$ is close to the true distribution, then the exponent achieved by the chosen $\balpha$ will be close to the estimated exponent based on the estimated distribution, since $\gamma(\balpha)$ given by \eqref{eq:gamma_w} varies continuously in $\bphi$ for each $\balpha \in \textup{relint}(\Omega)$.

\subsection{State-independent policies and naive state-dependent policies are inferior}\label{subsec:neg_result_state_independent}

\paragraph{State-independent policies.} Previous works studying control of circulating resources in networks, e.g., \cite{ozkan2016dynamic} and \cite{banerjee2016pricing}, have proposed state-independent control policies. 
We show that in our setting, such policies are not competitive with the SMW policies we have proposed. 
We first formally define state-independent policies.
\begin{defn}[State independent policy]\label{defn:state_independent}
We call a scheduling policy $U$ \yka{Please replace $\pi$ by $U$ throughout.}\emph{state independent} if, for each\footnote{We suppress the dependence on $K$ in our notation.} $K \geq 1$, it maps each $j'\in V_S$, $k\in V_B$, $r\in \mathbb{Z}_+$ to a distribution $u_{j'k}(t_r^{-})$ over $\partial(j')\cup \{\emptyset\}$; for the $r$-th service token arrival with origin $j'$ and destination $k$, the platform serves buffer $i$ drawn independently from distribution $u_{j'k}(t_r^{-})$, ignoring the current state {$\X(t_r^{-})$} and the history.
If $i = \emptyset$ or there is no job at the served buffer, the service token is lost. 
\end{defn}

The next proposition formalizes that for any state independent policy: (i) Exponentially small loss is impossible (even if service token arrival rates are exactly known), (ii) Given a compatibility graph $G$ and a state independent policy, for ``almost all'' service token type distributions $\bphi$ the loss incurred under the policy does not vanish as $K \to \infty$; informally, asymptotic optimality fails if $\bphi$ is not exactly known. The proof is in Appendix \ref{appen:sec:crp}. 
\begin{prop}[All state independent policies have inferior performance]\label{prop:state_ind_no_exp}
Fix a compatibility graph $G$ and any state-independent scheduling policy $U$. We have:
	\begin{compactenum}[label=\arabic*.,leftmargin=*]
	    \item (Exponentially small loss is impossible.) For any service token type distribution $\bphi$, $\mathbb{P}^{\N,U}	= \Omega\left(\frac{1}{\N^2}\right)$. In particular, $\gamma(U)=0$, where $\gamma(\cdot)$ is the exponent defined in \eqref{eq:lb_measure_rate}.
	    \item (For almost all $\bphi$, asymptotic optimality fails.) Let $\textup{Supp}(\bphi) \triangleq \{(j',k)\in V_S\times V_B: \phi_{j'k}>0\}$. Fix any subset of service token types $S \subseteq V_S \times V_B$ such that each server $j' \in V_S$ has at least one service token type in $S$. Let $D(S) \triangleq \{\bphi: \textup{Supp}(\bphi) = S \}$ be the set of service token type distributions with support $S$. 
	    Then, then there is a subset of $D(S)$ which is open and dense in $D(S)$ such that for all $\bphi$ in this subset it holds that $\liminf_{K \to \infty} \mathbb{P}^{K,U} > 0$. \yka{This claim is slightly weaker than $\Omega(1)$ loss. I think $\Omega(1)$ loss may be false, because I can define my SI policy to be based, for different $K$, on different $\bphi$s which form a dense set in $D(S)$.}
	\end{compactenum}
\end{prop}

Proposition~\ref{prop:state_ind_no_exp} shows that as $K$ grows, any state independent policy suffers from inferior performance. 
There are two possibilities regarding what is known about the service token type distribution $\bphi$:
\begin{compactenum}[label=\arabic*.,leftmargin=*]
    \item {\bf $\bphi$ exactly known.} In this case, part 1 of Proposition~\ref{prop:state_ind_no_exp} tells us that any state independent policy has loss $\Omega(\tfrac{1}{K^2})$ whereas any SMW policy produces exponentially small loss (Theorem~\ref{thm:main_tight} part 1) and moroever SMW($\balpha$) is exponent optimal for $\balpha$ chosen to maximize $\gamma(\balpha)$ in \eqref{eq:gamma_w}.
    \item {\bf $\bphi$ is not exactly known.} In this case, any state independent policy typically fails to achieve asymptotic optimality (part 2 of Proposition~\ref{prop:state_ind_no_exp}) whereas vanilla MaxWeight (or any fixed SMW policy) achieves exponentially small loss. 
\end{compactenum}

\section{Analysis of Scaled MaxWeight Policies: Proof of Theorem~\ref{thm:main_tight}}\label{sec:smw_analysis}

In this section, we prove that any SMW policy is exponent optimal given its resting state, and derive explicitly the throughput-loss exponent achieved, and the most likely sample paths leading to throughput loss. In Section \ref{subsec:FSP_FL}, we follow the standard approach for large deviations analyses and characterize the system behavior in the fluid scale through fluid sample paths and fluid limits. In Section \ref{subsec:lyap_functions} we take a novel approach to define 
a family of Lyapunov functions parameterized by the desired state, since we do not know the ideal state for the system.
In Section \ref{subsec:sufficient_condition_opt} we follow  \cite{venkataramanan2013queue} and show that if the Lyapunov function (centered at the starting state) is scale-invariant and sub-additive, a policy that performs steepest descent on this Lyapunov function is exponent optimal. In Section \ref{subsec:smw_exponent_subset} we prove that each SMW policy performs steepest descent on the Lyapunov function centered at its resting state and is hence exponent optimal given its resting state. We also explicitly characterize the optimal exponent, the most likely sample paths leading to throughput loss, and the critical subsets (i.e., the subsets that are most likely to be depleted of jobs). Finally, we deduce Theorem~\ref{thm:main_tight}.

\subsection{Fluid Sample Paths and Fluid Limits}\label{subsec:FSP_FL}
For any stationary scheduling policy $U\in\mathcal{U}$ defined in Section \ref{sec:model}, we define the scaled service token and queue-length sample paths by (the former was defined in \eqref{eq:demand-sample-path})
\begin{equation}\label{eq:fluid_scale}
\fA_{j'k}^{\N}(t)\triangleq \frac{1}{K}\bA_{j'k}^K(t)\, ,\quad
\fX_{i}^{\N,U}(t)\triangleq \frac{1}{K}\bX^{\N,U}_{i}(t)\, ,
\end{equation}
Note that for a fixed policy (with specified tie-breaking rules), each given service token sample path and initial state uniquely determines the state sample path.
We denote this correspondence by $\Psi^{K,U}:(\bar{\mathbf{A}}^{\N}(\cdot),\fX^{\N,	U}(0)) \mapsto \bar{\X}^{\N, U}(\cdot)$. \yka{Shall we write $\Psi^U$ instead of $\Psi$?}

To obtain a large deviation result, we need to study the service token process and the queue-length process in the fluid scaling, as captured in \eqref{eq:fluid_scale}. We take the standard approach of \textit{fluid sample paths} (FSP) (see \citealt{stolyar2003control,venkataramanan2013queue}).\yka{Everything implicitly depends on the policy. This should be discussed and clarified.}
\begin{defn}[Fluid sample paths]
\label{def:FSP}
	We call a pair $(\bar{\mathbf{A}}(\cdot),\bar{\X}^U(\cdot))_{T}\triangleq (\bar{\mathbf{A}}(\cdot),\bar{\X}^U(\cdot))_{t\in[0,T]}$\yka{It seems like $T$ should appear right here, e.g., $(\bar{\mathbf{A}}(t),\bar{\X}^U(t))_{t \in [0,T]}$. Also, at the end of section 2 there was repeated mention of $L^\infty[0,T]$ but here there is not a single mention, though you are trying to be \emph{more} formal here. This seems incongruous.} a fluid sample path on $[0,T]$ (under stationary policy $U$) if there exists a sequence
	$$
	(\,(\bar{\mathbf{A}}^{\N}(\cdot))_{t\in[0,T]},\, \allowbreak \bar{\X}^{\N,U}(0),\, \allowbreak (\Psi^{K,U}(\bar{\mathbf{A}}^{\N}(\cdot),\bar{\X}^{\N,U}(0)))_{t\in[0,T]}\,)
	$$
	where $\bar{\mathbf{A}}^{\N}(\cdot)$  are scaled service token sample paths and $\bar{\X}^{\N,U}(0)\in \Omega$, such that it has a subsequence which converges to $((\bar{\mathbf{A}}(\cdot))_{t\in[0,T]},\bar{\X}^U(0),(\bar{\X}^U(\cdot))_{t\in[0,T]})$ uniformly on $[0,T]$.\yka{Shouldn't we define FSPs for any $U$ instead of restricting to $SMW(\balpha)$ only? I imagine we need it for the converse.}
\end{defn}

In short, FSPs include both typical and atypical sample paths. Recall Fact 1, which gives the likelihood for an unlikely event to occur based on the most likely fluid sample path that causes the event.
Accordingly, the large deviations analysis in Section \ref{subsec:smw_exponent_subset} will identify the most likely FSP that leads to throughput loss.

\textit{Fluid limits} are fluid sample paths that characterize \textit{typical} system behavior, as they are the formal limits in the Functional Law of Large Numbers \citep{dai1995positive}.
\begin{defn}[Fluid limits]
	We call a pair $(\bar{\mathbf{A}}(\cdot),\bar{\X}^U(\cdot))_{T}$ a fluid limit on $[0,T]$ (under stationary policy $U$) if (i) the pair $(\bar{\mathbf{A}}(\cdot),\bar{\X}^U(\cdot))_{T}$ is a fluid sample path; (ii) we have $\bar{\mathbf{A}}_{j'k}(t)={\phi}_{j'k} t$,\yka{Do you mean $\hat{\phi}_{j'k}$?} for all $j'\in V_S,k\in V_B$ and all $t\in[0,T]$.
\end{defn}

\subsection{A family of Lyapunov functions}\label{subsec:lyap_functions}
Lyapunov functions are a useful tool for analyzing complex stochastic systems.
In open queuing networks the ideal state is one in which all queues are empty, and correspondingly the Lyapunov function is chosen to achieve its minimum value in the ideal state, e.g., the sum of squared queue lengths Lyapunov function is a popular choice \citep[][etc.]{tassiulas1992stability,eryilmaz2012asymptotically}, while others have also used piecewise linear Lyapunov functions (\citealt{bertsimas2001performance,venkataramanan2013queue}, etc.).
Since our setting is a closed queueing network and ideal state is unknown, we instead construct a novel approach. We define a family of piecewise linear Lyapunov functions, parameterized by the desired state $\balpha$, such that the function achieves its minimum at $\balpha$.
\begin{defn}\label{defn:lyap_func}
	For each $\balpha\in \textup{relint}(\Omega)$, define Lyapunov function $\V_{\balpha}(\mathbf{x}):\Omega \rightarrow [0,1]$ as
	$
	\V_{\balpha}(\mathbf{x})\triangleq 1-\min_{i}\frac{x_i}{\w_i}.
	$
\end{defn}

The intuition behind our definition is as follows.
The Lyapunov function value is jointly determined by the desired state $\balpha$ of the system (under some policy) and our objective of avoiding throughput loss, and can be interpreted as the energy of the system at each state.
The desired state should have minimum energy, and the most undesirable states should have maximum energy.
In our case the boundary $\partial \Omega$ of $\Omega$ is most undesirable since throughput loss only happens there, and correspondingly, $L_{\balpha}(\bx)=1$ for $\bx \in \partial \Omega$, whereas $L_{\balpha}(\balpha)=0$ as we want. In general, for $\bx \in \Omega$, $L_{\balpha}(\bx)$ is one minus the smallest scaled queue length, given scaling factors $\balpha$.
See Figure \ref{fig:lyapunov_level} for an illustration.

These functions moreover have the following properties which play a key role in our analysis: 
\begin{lem}[Key properties of $L_{\balpha}(\cdot)$]\label{lem:key_property_lyap}
	For $L_{\balpha}(\cdot)$ with $\alpha\in\textup{relint}(\Omega)$, we have:
	\begin{compactenum}[label=\arabic*.,leftmargin=*]
		\item Scale-invariance (about $\balpha$). $L_{\balpha}(\balpha+c\Delta\bx)=cL_{\balpha}(\balpha+\Delta\bx)$ for any $c>0$ and $\Delta\bx \in \mathbb{R}^m$ such that $\mathbf{1}^{\T} \Delta\bx = 0$ and $\balpha+\Delta\bx \in \Omega, \balpha+ c\Delta\bx \in \Omega$. \yka{We should use some other notation like $\Delta \bx$ and $\Delta \bx'$ instead of $\bx$ since this quantity (the deviation from $\balpha$) does not belong to $\Omega$.}
		\item Sub-additivity (about $\balpha$).
		$L_{\balpha}(\balpha+\Delta\bx+\Delta\bx')\leq L_{\balpha}(\balpha+\Delta\bx) + L_{\balpha}(\balpha+\Delta\bx')$ for any $\Delta\bx , \Delta\bx' \in \mathbb{R}^m$ such that $\mathbf{1}^{\T} \Delta\bx = \mathbf{1}^{\T} \Delta\bx' = 0$ and $\balpha+\Delta\bx +\Delta\bx', \balpha+\Delta\bx, \balpha+\Delta\bx' \in \Omega$.
	\end{compactenum}
\end{lem}
The proof of Lemma \ref{lem:key_property_lyap} is in Appendix \ref{appen:sec:fsp}.
\yka{Why is this the right Lyapunov fn? A: given starting state and objective it is the right one. talk abt scale invariance.}

\begin{figure}[ht]
	\centering
	\includegraphics[width=0.4\columnwidth]{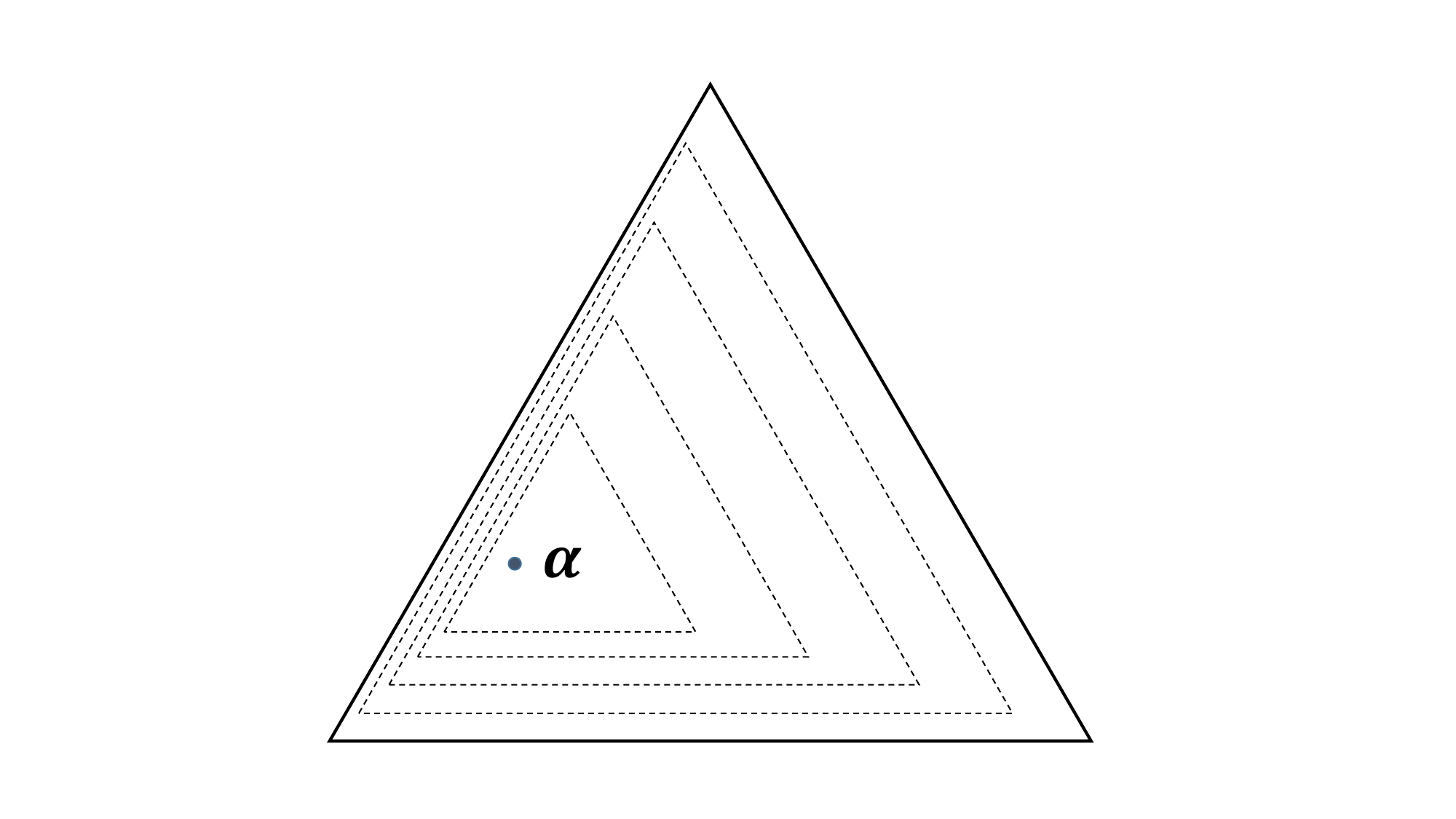}
	\caption{Sub-level sets of $\V_{\balpha}$ when $|V_B|=|V_S|=3$. State space $\Omega$ is the probability simplex in $\mathbb{R}^3$, and its boundary coincides with
		$\{\mathbf{x}:\V_{\balpha}(\mathbf{x}) = 
1,\mathbf{1}^{T}\bx=1\}$. The minimum value is achieved at $\balpha$; $L_{\balpha}(\balpha) = 0$. }\label{fig:lyapunov_level}
\end{figure}

A time $t\in(0,T)$ is said to be a \emph{regular point} of an FSP $(\fA(\cdot),\fX^U(\cdot))_T$ if $\fA(\cdot),\fX^U(\cdot),L_{\balpha}(\fX^U(\cdot))$ are all  differentiable at time $t$. 
Because of the Large Deviations Principle (Fact \ref{fact:sample_path_ldp}), it will suffice in our analysis to consider only the FSPs that have absolutely continuous service token sample paths.
Now, if $\fA(\cdot)$ is absolutely continuous, then so are $\fX^U(\cdot)$ and $L_{\balpha}(\fX^U(\cdot))$, and as a result almost all $t$ are regular:
For any policy $U\in\cU$ and FSP $(\fA(\cdot),\fX^U(\cdot))_T$, it holds that for any $t,t'$, $||\fX^U(t)-\fX^U(t')||_1 \leq 2||\fA(t)-\fA(t')||_1$ because jobs relocate only when service token arrives, and $|L_{\balpha}(\fX^U(t))-L_{\balpha}(\fX^U(t'))| \leq \frac{1}{\min_{i\in V_B}\alpha_i}||\fX(t)-\fX(t')||_{\infty}\leq \frac{1}{\min_{i\in V_B}\alpha_i}||\fX(t)-\fX(t')||_{1}$ (see Appendix \ref{appen:sec:fsp} for the short proof). As a result, if $\fA(\cdot)$ is absolutely continuous then so are $\fX(\cdot)$ and $L_{\balpha}(\fX^U(\cdot))$.

\subsection{Sufficient conditions for exponent optimality}\label{subsec:sufficient_condition_opt}
In this section, given a starting state, we provide a converse bound on the exponent for any stationary policy $U\in\mathcal{U}$, and derive sufficient conditions for a policy to achieve this bound.
\yka{At some point here or later we should remark on the fact that considering this ``weak'' adversary who only chooses between time invariant arrival patterns already gives us a tight converse. I imagine this has to do with the subadditivity and scale invariance of $\mathbf{f}$ and the state invariance of the adversary's cost function.}

We use the intuition from differential games (see, e.g., \citealt{atar2003escape}) to informally illustrate the interplay between the control and the most likely sample path leading to throughput loss.
Consider a zero-sum game between the adversary (nature) who chooses the fluid-scale service token arrival process $\fA(\cdot)$, and the controller who decides the scheduling rule $U$, where the adversary minimizes the large-deviation ``cost'' of a service token sample path that leads to throughput loss. Specifically, the adversary's cost for a service token sample path $\fA(\cdot)$ is the rate function defined in (\ref{eq:ld_rate}), i.e., the exponent. The converse bound we will obtain next will correspond to the adversary playing first and choosing the minimum cost \emph{time-invariant} service token sample path that ensures throughput loss.
The following pleasant surprises will emerge subsequently: (i) we will find an equilibrium in pure strategies to the aforementioned zero-sum game, (ii) the converse will turn out to be tight, i.e., the adversary's equilibrium service token sample path will be time invariant, (iii) the controller's equilibrium scheduling strategy will be an SMW policy with specific $\balpha$ (this simple policy will satisfy the sufficient conditions for achievability we will state immediately after our converse, in Proposition~\ref{prop:tight_converse}).

We provide a policy-independent upper bound on the exponent that only depends on the starting state.
First, for any $\mathbf{f} \in \mathbb{R}_+^{n \times m}$, define
\begin{align}
\mathcal{X}_{\mathbf{f}}\triangleq
\left\{\Delta \bx \left|
\begin{array}{ll}
\Delta x_i = \sum_{j'\in V_S}f_{j'i} - \sum_{j'\in \partial(i)}d_{ij'}\left(\sum_{k\in V_B}f_{j'k}\right),&\forall i\in V_B\\
\sum_{i\in\partial(j')}d_{ij'}=1,\ d_{ij'}\geq 0, &\forall i\in V_B,j'\in V_S
\end{array}
\right.\right\},\label{eq:fluid_control_polytope}
\end{align}
which is the attainable change of (normalized) state in unit time, given that the average service token arrival rates during this period are $\mathbf{f}$ and assuming no service token is wasted. (Here $(d_{ij'})_{i \in \partial(j')}$ is the chosen \emph{scheduling distribution} over buffers neighboring $j'$ for scheduling jobs to be served by service tokens originating at $j'$.)
Then given starting state $\balpha$, the attainable states at time $T$ belong to $\balpha + T\mathcal{X}_{\mathbf{f}}\triangleq\{\by\in\mathbb{R}^m:\by=\balpha+T\bx,\bx\in\mathcal{X}_{\bof}\}$\pqa{note to self: define $T\mathcal{X}_{\mathbf{f}}$.}\yka{delete ``(using the scale-invariance of $L_{\balpha}(\cdot)$)''}, if no service token is wasted during $[0,T]$ and the average service token arrival rate is $\bof$.
We obtain an upper bound on the throughput-loss exponent  by considering the most likely $\mathbf{f}$ and $T$ such that $\balpha + T\mathcal{X}_{\mathbf{f}}$ {lies entirely outside the state space $\Omega$.
Because the true state must lie in $\Omega$, there must be service token wasted during $[0,T]$, no matter the scheduling rule $\mathbf{d}$ used by the controller.}

\begin{lem}[Converse bound on the exponent]\label{lem:point_wise_converse}
	For any stationary policy $U\in\mathcal{U}$, it holds that 
	\begin{align}\label{eq:rate_cb_bound}
		-\liminf_{K\to\infty}\frac{1}{K}\log\mathbb{P}^{K,U}
		\leq \textup{sup}_{\balpha \in \textup{relint}(\Omega)} \gamma_{\textup{\tiny CB}}(\balpha) \, ,
	\end{align}
	\begin{align*}
	\textup{where, for $\Lambda^*(\cdot)$ given by \eqref{eq:kl_divergence},}\ \
	\gamma_{\textup{\tiny CB}}(\balpha) \triangleq
	 \inf_{\mathbf{f} \in \mathbb{R}_+^{nm}:v_{\alpha}(\mathbf{f})>0}\,
	 \frac{\Lambda^*(\mathbf{f})}{v_{\alpha}(\mathbf{f})}\, ,
	\ \;
	\textup{and}\ \;
	v_{\alpha}(\mathbf{f})\triangleq \min_{\Delta \bx\in\mathcal{X}_{\mathbf{f}}} L_{\balpha}(\balpha+\Delta \bx)\, .
	\end{align*}
\end{lem}

We now provide an informal explanation for the form of this key lemma. The $\balpha$ in \eqref{eq:rate_cb_bound} captures the most frequently visited (normalized) state (the ``resting'' state) in steady state under $U$, and $\gamma_{\textup{\tiny CB}}(\balpha)$ is an upper bound on the exponent given the most frequent state $\balpha$. Let us informally describe the expression for $\gamma_{\textup{\tiny CB}}(\balpha)$. Suppose the system starts in state $\balpha$. Then $v_{\balpha}(\mathbf{f})$ is the minimum rate of increase of $L_{\balpha}(\cdot)$ under service token arrival rates $\mathbf{f}$, no matter the scheduling distributions $\mathbf{d}$. So, starting at $\balpha$ and under time-invariant service token arrival rates $\mathbf{f}$, the state hits $\partial\Omega$ and service token is wasted in time at most $1/v_{\balpha}(\mathbf{f})$, implying a throughput-loss exponent of at most $\frac{\Lambda^*(\mathbf{f})}{v_{\balpha}(\mathbf{f})}$. The upper bound $\gamma_{\textup{\tiny CB}}(\balpha)$ follows from
minimizing over $\mathbf{f}$ since nature can choose any $\mathbf{f}$.
Finally, the bound in \eqref{eq:rate_cb_bound} takes the supremum over $\balpha$ since the policy can choose its resting state.
The proof of Lemma \ref{lem:point_wise_converse} is in Appendix \ref{appen:sec:point_wise_converse}.

Recall that for a function $g(\cdot):\mathbb{R}_+\to\mathbb{R}^d$ for some positive integer $d$, we use $\dot{g}(t)$ to denote the derivative of $g$ at time $t$ when the derivative exists.

The following proposition provides sufficient conditions for a policy to achieve the converse bound exponent $\gamma_{\textup{\tiny CB}}(\balpha)$. 
The conditions are requirements on the time derivative of $L_{\balpha}(\fX^U(t))$.
\begin{prop}[Sufficient conditions]\label{prop:tight_converse}
	Fix $\balpha\in\textup{relint}(\Omega)$. Let $U\in\mathcal{U}$ be a stationary, non-idling policy. Suppose that for each regular point $t$,\yka{Where is ``regular'' defined?} the following hold:
	\begin{compactenum}[label=\arabic*.,leftmargin=*]
		\item (Steepest descent). For any \yka{delete ``fluid sample path $(\fA(\cdot),\fX^{U}(\cdot))$''}service token fluid sample path $\fA(\cdot)$, we have
		\begin{align*}
		\dot{L}_{\balpha}(\fX^U(t))
			 = \inf_{U'\in\mathcal{U}_{\rm {\tiny ni}}}\left\{\dot{L}_{\balpha}(\fX^{U'}(t))
			 \left|
			 \fX^{U'}(t) = \fX^{U}(t)
			 \right.
			 \right\}\, ,
		\end{align*}
				for corresponding queue-length sample paths satisfying $\fX^U(t)\neq\balpha$ and $L_{\balpha}(\fX^U(t))<1$, where $\mathcal{U}_{\rm {\tiny ni}}$ is the set of non-idling policies;
\item (Negative drift). There exists $\eta>0$ and $\epsilon>0$ such that for all FSPs $(\fA(\cdot),\fX^U(\cdot))$ satisfying $\dot{\fA}(t)\in B(\bphi,\epsilon)$ and $\fX(t)\neq \balpha$, we have
		$
		\dot{L}_{\balpha}(\bar{\X}^U(t))
		\leq -\eta$.
Here $B(\bphi,\epsilon)$ is a ball with radius $\epsilon$ centered on the typical service token type distribution $\bphi$.
		 \yka{What is $B(\bphi,\epsilon)$? Why do we need property 2 if we have property 3? I suggest we write only property 3 and call it negative drift. Explain that $B$ is a ball centered on the typical service token arrival rate $\phi$, so in words, we require negative drift for near typical service token arrival rate, no matter what the current state is. Mention also that this property forces the state to return to $\balpha$!}
	\end{compactenum}
Then we have $\gamma(U) = \gamma_{\textup{\tiny CB}}(\balpha)$.
\end{prop}

Informally, the negative drift property requires the policy to have negative Lyapunov drift for near typical service token arrival rates, as long as the current state is not $\balpha$. This property forces the state to return to $\balpha$.

The full proof of Proposition \ref{prop:tight_converse} is quite technical and is included in Appendix \ref{appen:sec:sufficient}, but the key idea is straightforward.
Given starting state $\balpha$, the (i) steepest descent property of $U$ and (ii) the scale-invariance and sub-additivity of $L_{\balpha}(\cdot)$, together ensure that the speed at which $L_{\balpha}(\cdot)$ increases under $U$ cannot exceed the minimum speed $v_{\balpha}(\mathbf{f})$ in the converse construction (Lemma \ref{lem:point_wise_converse}) for $\mathbf{f}\triangleq \dot{\fA}(t)$.
Mathematically,
\begin{align}
&\quad \dot{L}_{\balpha}(\fX^U(t))\left|_{\dot{\fA}(t) = \mathbf{f}} \right.\nonumber\\[2pt]
&= \inf_{U'\in\mathcal{U}_{\rm {\tiny ni}}}\left\{\dot{L}_{\balpha}(\fX^{U'}(t))
\left|\dot{\fA}(t) = \mathbf{f} \right.
\right\}&\textup{(steepest descent)}\nonumber\\
&=\min_{\Delta\bx\in\mathcal{X}_{\mathbf{f}}}
\lim_{\Delta t\to 0}\frac{L_{\balpha}(\bar{\bX}^{U}(t)+\Delta\bx\Delta t)-
L_{\balpha}(\bar{\bX}^{U}(t))
}{\Delta t}&\textup{(definition of $\mathcal{X}_{\mathbf{f}}$)}\nonumber\\
&\leq \min_{\Delta\bx\in\mathcal{X}_{\mathbf{f}}}\lim_{\Delta t\to 0}\frac{L_{\balpha}(\balpha+\Delta\bx\Delta t)
}{\Delta t} &\textup{(sub-additivity of $L_{\balpha}$, Lemma \ref{lem:key_property_lyap})}\label{eq:sub_add_ineq}\\
&= \min_{\Delta\bx\in\mathcal{X}_{\mathbf{f}}}L_{\balpha}(\balpha+\Delta\bx)=v_{\balpha}(\mathbf{f})\,.&\textup{(scale-invariance of $L_{\balpha}$, Lemma \ref{lem:key_property_lyap})}\nonumber
\end{align}
As a result, the throughput-loss exponent under $U$ is no worse than $\gamma_{{\textup{\tiny CB}}}(\balpha)$.

\subsection{Optimality of SMW policies, explicit exponent, and critical subsets}\label{subsec:smw_exponent_subset}
In this section, we verify that SMW policies satisfy the sufficient conditions in Proposition \ref{prop:tight_converse}.
In doing so, we reveal the critical subset structure of the most-likely sample paths for throughput loss and derive the explicit exponent for SMW($\balpha$).
Proofs for this section are in Appendix \ref{appen:sec:smw}.

The following lemma shows that the Lyapunov drift only depends on the nodes with shortest scaled queue lengths, its service of jobs from these queues.

\begin{lem}[SMW($\balpha$) causes steepest descent]\label{lem:lyapunov_derivative}
	Let $(\fA,\bar{\X}^{U})$ be any FSP under any non-idling policy $U$ on $[0,T]$, and consider any $\balpha\in \textup{relint}(\Omega)$.
	For a regular $t\in[0,T]$, define:
	\begin{align*}
	S_1(\fX^U(t))&\triangleq\left\{k\in V_B:k\in\textup{argmin}
	\frac{\fX_k^U(t)}{\w_k}\right\},\\
	S_2\left(\fX^U(t),\dot{\fX}^U(t)\right)&\triangleq\left\{k\in S_1(\bar{\X}^U(t)):k\in\textup{argmin}
	\frac{\dot{\fX}_k^U(t)}{\w_k}\right\}.
	\end{align*}
	All the derivatives are well defined since $t$ is regular. We have
	\begin{align}
	\dot{L}_{\balpha}(\fX^U(t)) &= -\frac{\dot{\fX}_k^U(t)}{\w_k} \quad \textup{ for any } k\in S_2(\bar{\X}^U(t),\dot{\fX}^U(t))\label{eq:lyap_FSP_1}\\
	&\geq -\frac{1}{\mathbf{1}^{\T}_{S_2}\balpha}
	\left(\sum_{j'\in V_S,k\in S_2}\dot{A}_{j'k}(t)
	-
	\sum_{j'\in V_S:\partial (j')\subseteq S_2,k\in V_B}\dot{A}_{j'k}(t)\right)\label{eq:lyap_FSP_2}
	\end{align}
	for $\fX^U(t)\neq \balpha$ and $L_{\balpha}(\fX^U(t))< 1$.
	Inequality (\ref{eq:lyap_FSP_2}) holds with equality under \textup{SMW($\balpha$)}, i.e., \textup{SMW($\balpha$)} satisfies the steepest descent property in Proposition \ref{prop:tight_converse}.
\end{lem}

In Lemma \ref{lem:lyapunov_derivative_fluid}, we prove that SMW($\balpha$) satisfies the negative drift \yka{delete ``and robust drift''}property.
In particular, the drift $\eta$ is related to the Hall's gap (i.e., the slack in the CRP condition) of the network; see Appendix~\ref{appen:sec:smw} for details.

\begin{lem}[SMW($\balpha$) satisfies negative drift]\label{lem:lyapunov_derivative_fluid}
	For any $\balpha \in \textup{relint}(\Omega)$, under Assumptions~\ref{asm:connectivity}, \ref{asm:non_trivial} and \ref{asm:strict_hall}, the policy \textup{SMW}($\balpha$)  satisfies the negative drift condition in Proposition \ref{prop:tight_converse}.
\end{lem}

Before proceeding with our analysis, we point out that Lemma~\ref{lem:lyapunov_derivative_fluid} implies that $\balpha$ is the unique resting state of SMW($\balpha$) policy.
\begin{prop}[Resting state of SMW($\balpha$)]\label{prop:smw-resting-state}
Suppose Assumptions \ref{asm:connectivity}, \ref{asm:non_trivial} and \ref{asm:strict_hall} hold.
For any $\balpha \in \textup{relint}(\Omega)$, there exists $T_0>0$ such that any fluid limit ($\fA,\fX$) on $[0,T]$ (where $T>T_0$) under SMW($\balpha$) satisfies
$\fX(t)=\balpha$
for all $t \in [T_0, T]$.
\end{prop}

Combining Proposition~\ref{prop:tight_converse} with Lemmas~\ref{lem:lyapunov_derivative} and \ref{lem:lyapunov_derivative_fluid}, we immediately deduce that SMW$(\balpha)$ achieves the best possible exponent given resting state $\balpha$.\yka{We don't need to cite Lemma~\ref{lem:key_property_lyap} here, right? It got used in the proof of Prop \ref{prop:tight_converse} and that's it, no?}
\begin{cor}\label{cor:smw_exp_opt}
	For any $\balpha\in\textup{relint}(\Omega)$, we have $ \gamma(\balpha) = \gamma_{\textup{\tiny CB}}(\balpha)$.
\end{cor}

We argued in Section \ref{subsec:sufficient_condition_opt} that the most likely queue-length sample path leading to throughput loss with initial state $\balpha$ should be radial:
when the controller chooses an exponent-optimal policy,
the adversary picks a constant arrival rate $\bof$ such that the sample path of queue lengths is radial starting at $\balpha$, and the Lyapunov function increases at a constant rate.
From Lemma \ref{lem:lyapunov_derivative} we see that the rate at which the Lyapunov function increases depends on the (scaled) inflow and outflow rate of jobs in each subset.
Since the most likely queue-length sample path is radial, this sample path should drain the jobs of one subset (the critical subset), and that subset will determine the throughput-loss exponent.
We next lemma obtains an explicit expression for $\gamma_{{\textup{\tiny CB}}}(\balpha)$ and the most likely service token FSP forcing throughput loss.
\begin{lem}\label{lem:explicit_gamma}
	Recall the definitions of $\mathcal{J}$ in \eqref{eq:set_drainable} and $B_J,\lambda_J$ and $\mu_J$ in \eqref{eq:gamma_w}.
	For any $\balpha \in \textup{relint}(\Omega)$, we have $
		\gamma_{{\textup{\tiny CB}}}(\balpha)
		 = \min_{J \in\mathcal{J}}B_J \log(\lambda_J/\mu_J)$. 
		 Moreover, the infimum in the definition of $\gamma_{{\textup{\tiny CB}}}(\balpha)$ in Lemma \ref{lem:point_wise_converse} is achieved by the following $\mathbf{f}^*$: for any $J^* \in \textup{argmin}_{J\in\mathcal{J}}
	B_{J}\log(\lambda_J/\mu_J)$,
	\begin{align}	
	f^{*}_{j'k}\triangleq
	\left\{
	\begin{array}{ll}
	\hat{\phi}_{j'k}\lambda_{J^*}/\mu_{J^*}&\textup{ for }j'\in J^*,k\notin\partial(J^*)\, ,\\
	\hat{\phi}_{j'k}\mu_{J^*}/\lambda_{J^*}&\textup{ for }j'\notin J^*,k\in\partial(J^*)\, ,\\
	\hat{\phi}_{j'k}&\text{ otherwise}\, .
	\end{array}
	\right.
\end{align}
\end{lem}

\begin{rem}[Critical subset property]
\label{rem:critical-subset}
Lemma \ref{lem:explicit_gamma} provides the most likely service token sample path that leads to throughput loss under \textit{any} scheduling policy that is exponent optimal, starting at state\footnote{Remark~\ref{rem:critical-subset} applies to throughput lost over a (long) finite horizon given starting state $\balpha$. SMW$(\balpha)$ further forces the state to return to $\balpha$ (negative drift), so our observations carry over to the steady state as well under that policy.} $\balpha$.
We observe the \textit{critical subset property}: 
\begin{compactitem}[leftmargin=*]
 \item (Adversary's strategy) For each starting state $\balpha\in\textup{relint}(\Omega)$, there is (are) corresponding critical subset(s) $J^* \in \textup{argmin}_{J\in\mathcal{J}}	B_J\log(\lambda_J/\mu_J)
$, such that the most likely service token sample path forcing throughput loss drains a critical subset.
 \item (Controller's strategy) If the current state $\bx$ is on the most likely sample path forcing throughput loss in critical subset $J^*$ starting at $\balpha$, an exponent optimal policy (for given $\balpha$) will maximally protect $J^*$ at $\bx$, {i.e., the policy will serve jobs in $\partial(J^*)$ exclusively using service tokens originating in $J^*$}. Lemma~\ref{lem:lyapunov_derivative} tells us that SMW$(\balpha)$ is such a policy.
\end{compactitem}
\end{rem}
We can now prove the main theorem.

\begin{proof}[\textbf{Proof of Theorem \ref{thm:main_tight}}]
Lemma \ref{lem:point_wise_converse} along with the explicit expression for 	$\gamma_{{\textup{\tiny CB}}}(\balpha)$ provided by Lemma~\ref{lem:explicit_gamma} yields the converse result (part 2 of the theorem).

Achievability (part 1 of the theorem) follows from Corollary~\ref{cor:smw_exp_opt} along with the explicit expression for 	$\gamma_{{\textup{\tiny CB}}}(\balpha)$ provided by Lemma~\ref{lem:explicit_gamma}.
\end{proof}

\section{Buffer Space Planning and Dynamic Scheduling in Open Networks}\label{sec:open_network}

In this section, we consider a parallel-server system \cite[see, e.g.,][]{mandelbaum2004scheduling,shi2015process} with fixed total buffer space $K$, and a queue associated with each job type. 
At the beginning, the system controller needs to make a (static) buffer space planning decision which determines the buffer size of each queue. 
Then the controller needs to dynamically schedule waiting jobs to compatible servers, with the objective of minimizing the steady-state buffer overflow probability (of any queue).
We show that with only slight modifications, our model and results translate fully to this open network model, thus providing novel prescriptive insights into buffer space planning and dynamic scheduling control of such systems. 

\subsection{The parallel-server system}
We now provide a detailed description of the model.

\textbf{Parallel servers.} The set of graph primitives is the same as in the previous model, i.e., it consists of a compatibility graph $G(V_B\cup V_S,E)$. We let $m=|V_B|,n=|V_S|$.
Here $V_B$ is the finite set of buffers, one associated with each job type, and $V_S$ is the finite set of servers.
The neighborhood of $i\in V_B$ in $G$ consists of the servers that can serve buffer $i$, which is denoted by $\partial(i)\subseteq V_S$.
The neighborhood of $j'\in V_S$ in $G$ consists of the buffers that server $j'$ can serve, which is denoted by $\partial(j') \subseteq V_B$.
Type $i$ jobs arrive at buffer $i\in V_B$ according to a Poisson process with rate $\lambda_i$. Type $j'$ service tokens arrive at server $j'\in V_S$ according to a Poisson process with rate $\mu_{j'}$.

\textbf{Finite buffers with fixed total size.} 
The sum of the buffer sizes of the $m$ queues is given by $K$ in the $K$-th system.
At the beginning, the system operator needs to decide the buffer size of each queue, denoted by $\alpha_i K$ for queue $i$, subject to $\sum_{i\in V_B}\alpha_i=1$.
Denote the queue lengths at time $t$ as $\bX^{K}(t)=[X_1^{K}(t),\cdots,X_{m}^{K}(t)]$.
If a buffer is full, the jobs arriving to it will be lost.

\textbf{Dynamic scheduling.} 
After making the static decision of buffer sizing, the system controller's control lever is \emph{dynamic scheduling}: when a service token of type $j'$ arrives, the controller chooses the queue it serves. 
If all compatible buffers are empty, then the service token is lost.
As in the previous model, it suffices to consider stationary policies $U$, which is formally defined as a sequence of mappings, indexed by the total buffer size $K$, that map the current queue length $\bX^{K}$ and service token type $j'\in V_D$ to $\partial(j') \cup \{\emptyset\}$.
Let $t_{r}$ be the $r$-th service token arrival epoch after time $0$. Denote the state of the system just before $t_r$ by $\X^{K}(t_r^{-})$ (the initial state is $\X^{\N}(0)$).
Now suppose the controller uses a scheduling policy $U$, and
the $r$-th service token is of type $o[r]$.
Let $S[r] \triangleq U^{K}[\X^{K}(t_r^-)](o[r])$ be the chosen queue (potentially $\emptyset$). Formally, when a service token arrives,
\begin{equation*}
\X^{\N}(t_r)
	\triangleq
\left\{
\begin{array}{ll}
\X^{\N}(t_r^-) - \e_{S[r]}  &\text{if $S[r]\in V_S\, ,$}\\
\X^{\N}(t_r^-)&\text{if $S[r]=\emptyset\, .$}
\end{array}
\right.
\end{equation*}

\textbf{Performance measure.}
We consider a system controller who tries to minimize the buffer overflow probability. 
We define the performance measure in a similar way as in \eqref{eq:lb_performance_measure}:
\begin{equation}\label{eq:lb_performance_measure_open}
{\bbp}^{\N,U}
\triangleq
\lim_{T\to\infty}\frac{1}{T}\ 
\mathbb{E}\left(
\int_{t=0}^{T}
\ind\left\{\X^{\N,U}(s) \geq \lfloor\alpha_{i}K \rfloor \right\}
\textup{ds}
\right)\, ,
\end{equation}

Similarly, for policy $U$, we define \emph{buffer-overflow exponent} $\gamma(U)$ as below:
\begin{equation}\label{eq:lb_measure_rate_open}
\gamma(U) \triangleq -\limsup_{\N\to\infty}\frac{1}{\N}\log {\bbp}^{\N,U}\, ,
\end{equation}

\textbf{Complete Resource Pooling condition (for parallel-server systems).}
We require the following CRP condition on the network primitives $G$, $\blambda$ and $\bmu$ for our main result in this section.
\begin{asm}\label{asm:strict_hall_open}
We assume that for all subsets $I\subset V_B$ where $I\neq\emptyset$, it holds that
$
\lambda_I < \mu_I$ 
	for  ${\lambda_I} \triangleq \sum_{i\in I}\lambda_{i}$ and $\mu_I \triangleq \sum_{j'\in \partial(I)}\mu_{j'}\, .$
\end{asm}
Intuitively, Assumption \ref{asm:strict_hall_open} assumes that for each subset $I\subsetneq V_B$ of buffers, the (optimistic) service rate of jobs in $I$ is faster than the job arrival rate to $I$.

\subsection{Optimal buffer sizing, SMW, and main result}
Leveraging the similarity between the current model and the previous model introduced in Section \ref{sec:model}, we show that for this model SMW also achieves exponentially decaying job loss.
Crucially, if the (static) buffer sizing decision is $\balpha\in \textup{relint}(\Omega)$, it suffices to consider the SMW policy with the same parameter $\balpha$.
The formal definition of SMW in this setting is as follows:
\begin{defn}[SMW($\balpha$) for parallel-server systems]
Fix $\balpha \in \textup{relint}(\Omega)$, i.e., $\balpha \in \mathbb{R}^m$ such that $\alpha_i > 0 \ \forall i \in V_S$ and $\sum_{i \in V_S} \alpha_i =1$.	Given system state $\X(t_r^-)$ just before the $r$-th service token and for demand with type $j'$, \textup{SMW}($\balpha$) serves queue
	\begin{equation*}
	\textup{argmax}_{k\in \partial(j')}\frac{X_k(t_r^-)}{\alpha_k}
	\end{equation*}
	if $X_i(t_r^-)>0$; otherwise the request is lost.
	(If there are ties when determining the argmax, it assigns from the queue with highest index.) \yka{Is this $\alpha$ vector the same as the $\alpha$ which defines the buffer sizes? If so, what's the reason it suffices to identify the 2 $\alpha$s? Please add some commentary and please make sure the reader leaves with a clear understanding of what you find regarding optimal buffer sizing. Right now it's a bit buried.}
\end{defn}

The following performance guarantee similar to Theorem \ref{thm:main_tight} holds for SMW($\balpha$) under the CRP condition (Assumption \ref{asm:strict_hall_open}).

\begin{thm}[Result for parallel-server systems]\label{thm:main_tight_open}
	For any parallel-server system $(G, \blambda,\bmu)$ satisfying Assumption \ref{asm:strict_hall_open}, we have:
	\begin{compactenum}[label=\arabic*.,leftmargin=*]
		\item \textbf{Exponentially small loss under any SMW policy}: For any $\balpha \in\textup{relint}(\Omega)$, if the buffer size of queue $i$ is chosen to be $\alpha_i K$, \textup{SMW($\balpha $)} achieves exponential decay of the buffer-overflow probability in $K$ with exponent,
		\begin{align}
		\gamma(\balpha )
		=
		\min_{I\subset V_S, I \neq \emptyset}
		B_I
		\log\left(
		\frac{\mu_I}{\lambda_I}
		\right)
		>0\, ,
		\end{align}
		\begin{align}
		\textup{where}\quad
		B_I\triangleq \mathbf{1}_{I}^{\textup{T}}\balpha\, ,\qquad
		\lambda_I \triangleq \sum_{i\in I}\lambda_{i}
		\, ,\quad \textup{and} \quad
		\mu_I \triangleq \sum_{j'\in \partial(I)}\mu_{j'}
		\, .\label{eq:lambdaJ_defn_scrip}
		\end{align}
		\yka{Should we write $\gamma(\alpha)$ with $B_J, \lambda_J, \mu_J$? It may help in the exposition.}
		\item \textbf{There is an exponent optimal SMW policy}: Under \emph{any} buffer sizing and scheduling policy $U$, it must be that
		\begin{equation}
		\gamma(U) \leq \bar{\gamma} \, ,
		\qquad
		\textup{where}\quad\bar{\gamma} = \sup_{\balpha \in\textup{relint}(\Omega)} \gamma(\balpha ) \, .
		\end{equation}
	\end{compactenum}
	Thus, there is a buffer sizing decision and corresponding \textup{SMW} rule that achieves an exponent arbitrarily close to the optimal one.
\end{thm}

The proof of Theorem \ref{thm:main_tight_open} is almost identical to that of Theorem \ref{thm:main_tight}, hence we do not write a separate proof for it, and limit ourselves to summarizing how the proof of Theorem \ref{thm:main_tight} can be reused:
Part 1 of Theorem \ref{thm:main_tight_open} states that given a buffer sizing decision $\balpha$, SMW($\balpha$) maximizes the buffer-overflow exponent. 
To establish this, we define Lyapunov function $L_{\balpha}(\bx) = \max_{i\in V_B}\frac{x_i}{\alpha_i}$ which achieves its minimum at where queues are identically zero, and achieves maximum if and only if some buffer is full. 
Similar to the proof of Theorem \ref{thm:main_tight}, we can show that SMW($\alpha$) executes steepest descent on this Lyapunov function, hence it is exponent-optimal.
We then optimize over $\balpha$ and obtain the optimal buffer sizing decision, thus proving part 2 of Theorem \ref{thm:main_tight_open}.

\section{Discussion}\label{sec:conclu}
In this paper we study dynamic scheduling control of a closed queueing network.
We introduce a family of state-dependent scheduling policies called Scaled MaxWeight (SMW) and prove that they have superior performance in terms of maximizing throughput, compared with state-independent policies including previously proposed policies.
In particular, we construct an SMW policy that (almost) achieves the optimal large deviation rate of decay of throughput loss.
Our analysis also uncovers the structure of the problem: given system state, throughput loss is most likely to happen within state-dependent critical subsets of servers.
The optimal SMW policy protects all critical subsets simultaneously.
SMW policies are simple and explicit, and hence have the potential to influence practice in applications such as shared transportation systems.
Our work provides the first large deviations analysis under complete resource pooling, yielding sharp control insights. We also show how the methodology we introduce may inspire similar analyses in open networks, e.g., obtaining exponent optimal controls when there is a shared finite buffer for multiple queues.

{
\bibliographystyle{ormsv080}
\bibliography{bib_closed_qnet}{}
}
\appendix

\newpage

\vskip35pt
\begin{center}
  \Large {\bf Appendix}
\end{center}
\vskip10pt

\noindent This technical appendix is organized as follows.
We prove our main result, Theorem \ref{thm:main_tight}, in Appendices~\ref{appen:sec:fsp}-\ref{appen:sec:smw}. In particular:
\begin{compactitem}[leftmargin=*]
    \item Appendix~\ref{appen:sec:fsp}
    establishes key properties of our Lyapunov functions, including the proof of Lemma \ref{lem:key_property_lyap}.
      \item Appendix~\ref{appen:sec:point_wise_converse} includes the proof of Lemma
      \ref{lem:point_wise_converse}, a converse bound on the throughput-loss exponent.
      \item Appendix~\ref{appen:sec:sufficient} includes the proof of Proposition \ref{prop:tight_converse}, containing sufficient conditions for a policy to achieve the optimal exponent.
      \item Appendix~\ref{appen:sec:smw} shows that the SMW policy satisfies the sufficient conditions for exponent optimality, and derives explicitly the optimal exponent and most-likely sample paths, including the proofs of Lemma \ref{lem:lyapunov_derivative}, Lemma \ref{lem:lyapunov_derivative_fluid}, and Lemma \ref{lem:explicit_gamma}. 
\end{compactitem}
      %
      %

Appendix~\ref{appen:sec:crp} shows the necessity of the assumptions and state-dependent control, including the proofs of Propositions \ref{prop:NT-is-necessary}, \ref{prop:hall_is_necessary} and \ref{prop:state_ind_no_exp}. 

\section{Properties of the Lyapunov functions $L_{\balpha}(\bx)$}\label{appen:sec:fsp}
\yka{Awkward title. Also: Typically, each appendix section should be pointed to in the main text.}

\subsection{Scale-invariance and sub-additivity (about $\balpha$): proof of Lemma \ref{lem:key_property_lyap}}
\begin{proof}[Proof of Lemma \ref{lem:key_property_lyap}]
	(i) For $c>0$, $\balpha\in\text{relint}(\Omega)$, we have
	\begin{align*}
	L_{\balpha}(\balpha + c\Delta \bx)
	= 1 - \min_{i}\frac{\alpha_i + c\Delta x_i}{\alpha_i}
	= - \min_{i}\frac{c\Delta x_i}{\alpha_i}
	= - c\min_{i}\frac{\Delta x_i}{\alpha_i}
	= c L_{\balpha}(\balpha + \Delta \bx)\, .
	\end{align*}
	
	\noindent (ii) For $\balpha\in\text{relint}(\Omega)$, we have
	\begin{align*}
	&L_{\balpha}(\balpha + \Delta \bx + \Delta \bx')
	= 1 - \min_{i}\frac{\alpha_i + \Delta x_i + \Delta x_i'}{\alpha_i}
	=  - \min_{i}\frac{\Delta x_i + \Delta x_i'}{\alpha_i}\\
	\leq\ &
	- \min_{i}\frac{\Delta x_i}{\alpha_i}
	-
	\min_{i}\frac{\Delta x_i'}{\alpha_i}
	=
	L_{\balpha}(\balpha + \Delta \bx)
	+
	L_{\balpha}(\balpha + \Delta \bx')\, .
	\end{align*}
	
\end{proof}

\subsection{Regularity properties}
The following lemma is a collection of regularity properties of $L_{\balpha}(\bx)$ that are useful in the following proofs.
\begin{lem}\label{lem:tech_lems}
	For $\balpha\in\textup{relint}(\Omega)$ and $\V_{\balpha}(\bx)$ specified in Definition \ref{defn:lyap_func},  we have
	\begin{compactenum}[label=\arabic*.,leftmargin=*]
		\item $\V_{\balpha}(\bx)\geq 0$ for all $\bx\in \Omega$, and $\V_{\balpha}(\bx)=0$ if and only if $\bx=\balpha$.
		\item $\V_{\balpha}(\bx)$ is globally Lipschitz on $\Omega$, i.e. for any $\bx_1,\bx_2\in\Omega$, we have
		\begin{equation*}
		|\V_{\balpha}(\bx_1)-\V_{\balpha}(\bx_2)|\leq \frac{1}{\min_{i}\alpha_i}||\bx_1 - \bx_2||_{\text{\tiny $\infty$}}\, .
		\end{equation*}
	\end{compactenum}
\end{lem}
\begin{proof}[Proof of Lemma \ref{lem:tech_lems}]
	Property 1 is easy to verify hence we omit the proof.
	
	For property 2, note that
	\begin{align*}
		|L_{\balpha}(\bx_1)-L_{\balpha}(\bx_2)|
		=\ 
		\left|\min_i \frac{x_{1,i}}{\alpha_i} - \min_i \frac{x_{2,i}}{\alpha_i} \right| 
		\leq\ 
		\min_{i}\frac{|x_{1,i}-x_{2,i}|}{\alpha_i}
		\leq\ 
		\frac{1}{\min_i \alpha_i} ||\bx_1 - \bx_2||_{\infty}\, .
	\end{align*}
	
\end{proof}
\yka{Please include the proof of the last property.}

\section{Converse bound on the exponent: proof of Lemma~\ref{lem:point_wise_converse}}\label{appen:sec:point_wise_converse}
In this section, we prove Lemma \ref{lem:point_wise_converse}, the converse bound on the exponent for any policy $U\in\cU$.
The proof consists of three steps:
\begin{compactitem}[leftmargin=*]
	\item Step 1: For each stationary policy $U\in \cU$ we define a state $\tilde{\balpha}\in\textup{relint}(\Omega)$ such that the state visits the neighborhood of $\tilde{\balpha}$ frequently enough. In the following steps we will bound the throughput-loss exponent of $U$ by $\gamma_{\textup{\tiny CB}}(\tilde{\balpha})$.
	\item Step 2: Given that the system's initial state is close to $\tilde{\balpha}$, we construct a set of service token sample paths that are guaranteed to lead to a throughput loss regardless of the policy used.
	To this end, we compute $v_{\tilde{\balpha}}(\mathbf{f})$, which the minimum rate of increase of $L_{\tilde{\balpha}}(\cdot)$ under service token arrival rates $\mathbf{f}$ no matter the scheduling distributions. This step is used to lower bound the ``one-shot'' probability of throughput loss.
	\item Step 3: We use renewal-reward theorem to translate the one-shot throughput-loss probability to steady-state throughput-loss probability. The final bound in \eqref{eq:rate_cb_bound} takes the supremum over $\balpha$ since the policy can choose its resting state.
\end{compactitem}
The technique used in step 2 follows from Proposition 9 in \cite{venkataramanan2013queue}. The approach in steps 1 and 3 is novel (to the best of our knowledge) and tackles the key challenge of our closed network model, i.e., the policy has the flexibility to choose a resting state, as opposed to open network settings where the resting state is always $\bzero$.

\begin{proof}[Proof of Lemma~\ref{lem:point_wise_converse}]
	\noindent \textit{Step 1: Find the ``frequently visited'' state $\balpha$.}
	Fix a stationary policy $U\in\cU$.
	For each $K$, the $K$-th system under policy $U$ is a finite-state Markov chain, whose state space has cardinality smaller than $K^m$.
	Since we are deriving the converse bound, let the $K$-th system start within a communication class that minimizes steady state throughput loss among all communication classes.
	Denote the stationary distribution (henceforth it refers to the stationary distribution of the communication class where the initial state belongs to) of (normalized) states as $\pi^K(\fKX)$.
	Then there must exist a (normalized) state $\tilde{\bX}^K$ such that $\pi^K(\tilde{\bX}^K) \geq K^{-m}$.
	Take a subsequence $\{K_r\}$ of $\{K\}$ such that
	\begin{align*}
		\lim_{r\to\infty}\frac{1}{K_r}\log\mathbb{P}^{K_r,U}=
		\liminf_{K\to\infty}\frac{1}{K}\log\mathbb{P}^{K,U}\, .
	\end{align*}
	By compactness of $\Omega$, there must exist a further subsequence of $\{K_r\}$, which we denote by $\{K_{r'}\}$, and $\balpha\in \Omega$ such that
	$\lim_{r'\to \infty} \tilde{\bX}^{K_{r'}}=\balpha$.
	
	For any $0<\epsilon_1<\frac{1}{2}\left(\min_{j:\alpha_j>0}\alpha_j\right)$, define $\tilde{\balpha}\in\textup{relint}(\Omega)$ such that
	\begin{align*}
	0<\tilde{\alpha}_j <\epsilon_1/2\quad &\text{for $j$ such that }\alpha_j=0\, ,\\
	|\tilde{\alpha}_j-\alpha_j| <\epsilon_1/2\quad &\text{for $j$ such that }\alpha_j>0\, .
	\end{align*}
	Since $\balpha$ is the limit point of $\tilde{\bX}^{K_{r'}}$, there exists $r'_0(\epsilon)>0$ such that $\forall r'\geq r'_0(\epsilon)$,
	\begin{align}
	0 \leq \tilde{X}_j^{K_{r'}} < \tilde{\alpha}_j \quad &\text{for $j$ such that }\alpha_j=0\, ,\label{eq:alpha_cluster_zero}\\
	|\tilde{X}_j^{K_{r'}} -\alpha_j| < \epsilon_1/2\quad &\text{for $j$ such that }\alpha_j>0\, .\label{eq:alpha_cluster}
	\end{align}
	Inequalities \eqref{eq:alpha_cluster_zero} and \eqref{eq:alpha_cluster} imply that for $r' \geq r'_0(\epsilon)$
	\begin{align*}
	|\tilde{X}_j^{K_{r'}} -\tilde{\alpha}_j|
	\leq\ 
	\tilde{\alpha}_j
	<\epsilon_1,\quad &\text{for $j$ such that }\alpha_j=0 \\
	|\tilde{X}_j^{K_{r'}} -\tilde{\alpha}_j|
	\leq\ 
	|\tilde{X}_j^{K_{r'}} -\alpha_j|
	+
	|\tilde{\alpha}_j - \alpha_j|
	<\epsilon_1,\quad &\text{for $j$ such that }\alpha_j>0\, .
	\end{align*}
	Hence $||\tilde{\bX}^{K_{r'}} - \tilde{\balpha}||_{\infty} < \epsilon_1$ for $r' \geq r'_0(\epsilon)$.
	
	We quantify the fact that $\tilde{\balpha}$ is a ``frequently visited'' state in the following claim.
	
	\noindent\textit{Claim:} Fix $K=K_{r'}$ that comes from the subsequence defined above. In the $K$-th system, define
	\begin{align}
		\tau^{K} \triangleq \inf\left\{t>0:\fX^{K}(t)= \tilde{\bX}^{K} | \fKX(0)=\tilde{\bX}^{K}\right\}\, , \label{eq:cycle-time-converse-bound}
	\end{align}
	then we have
	$
		\mathbb{E}[\tau^{K}] \leq\ \frac{K^m}{\mathbf{1}^{\T}{\bphi}\mathbf{1}}\, .
	$
	
	\noindent\textit{Proof of claim:} Consider the discrete-time embedded chain of $\{\fX^{K}(\cdot)\}$. Since the initial state $\tilde{\bX}^{K}$ is positive recurrent within its communication class, the expected number of jumps between two consecutive visits to $\tilde{\bX}^{K}$ is inversely proportional to its steady state measure $\pi^{K}(\tilde{\bX}^K)$. By definition of $\tilde{\bX}^K$, the expected number of jumps must be no larger than $K^m$. Since the time between two jumps are i.i.d. exponential variables with mean $(\mathbf{1}^{\T}\bphi\mathbf{1})^{-1}$, this concludes the proof.
		
	\medskip
	
	\noindent \textit{Step 2: Lower bound on the ``one-shot'' throughput loss probability.}
	Fix $K_{r'}$ and a service token sample path $\fA^{K_{r'}}(\cdot)$.
	For $t>0$, define $f_{j'k}(t) \triangleq \frac{1}{t}\fA^{K_{r'}}(t)$, i.e. the average arrival rate of type ($j',k$) service token during $[0,t]$.
	For stationary policy $U$, denote the average fraction of service tokens arriving at $j'$ that serve jobs at $i$ during this period as $d^U_{ij'}(t)$ (we omit the superscript $U$ in the following for notational simplicity).
	For $t\geq 0$, if $\fX^{K_{r'}}(0) = \tilde{\bX}^{K_{r'}}$ and no service token is wasted prior to $t$, we have for any $i\in V_B$
	\begin{align*}
	\bar{X}_i^{K_{r'}}(t) - \tilde{X}_i^{K_{r'}}
	=\ & 
	t \left(\sum_{j'\in V_S}f_{j'i}(t) - \sum_{j'\in \partial(i)}d_{ij'}(t)\left(\sum_{k\in V_B}f_{j'k}(t)\right)
	\right)\, .
	\end{align*}
	Since $\tilde{\alpha}_j>0$ for any $j\in V_B$, the Lyapunov function $L_{\tilde{\balpha}}(\cdot)$ is well-defined.
	Evaluate the Lyapunov function at $\left(\tilde{\balpha} + \bar{\X}^{K_{r'}}(t) - \tilde{\bX}^{K_{r'}}  \right)$, we have:
	\begin{align}
	L_{\tilde{\balpha}}\left( \tilde{\balpha}  + \fX^{K_{r'}}(t) - \tilde{\bX}^{K_{r'}} \right) 
	=\; &
	L_{\tilde{\balpha}}\left(\tilde{\balpha} + t\left(\sum_{j'\in V_S}f_{j'i}(t) - \sum_{j'\in \partial(i)}d_{ij'}(t)\left(\sum_{k\in V_B}f_{j'k}(t)\right)
	\right)_{i\in V_B} \right)\nonumber\\
	%
	\stackrel{(a)}{=}\; & t L_{\tilde{\balpha}}\left(
	\tilde{\balpha} + 
	\left(\sum_{j'\in V_S}f_{j'i}(t) - \sum_{j'\in \partial(i)}d_{ij'}(t)\left(\sum_{k\in V_B}f_{j'k}(t)\right)
	\right)_{i\in V_B} \right)\nonumber\\
	\geq\; & t  \min_{\Delta \bx\in\mathcal{X}_{\mathbf{f}}}L_{\tilde{\balpha}}(\tilde{\balpha}+\Delta \bx).
	\label{eq:smallest_speed}
	\end{align}
	Equality $(a)$ holds because the Lyapunov function is scale-invariant with respect to $\tilde{\balpha}$.
	Here $\Delta \bx$ is the change of (normalized) state in unit time given average service token arrival rate during this period $\mathbf{f}$, and $\mathcal{X}_{\mathbf{f}}$ is defined in (\ref{eq:fluid_control_polytope}).
	
	Define $v_{\tilde{\balpha}}(\bof)\triangleq \min_{\Delta \bx\in\mathcal{X}_{\mathbf{f}}}L_{\tilde{\balpha}}(\tilde{\balpha}+\Delta \bx)$, which is the minimum rate the Lyapunov function increases under any policy, given service token arrival rate $\mathbf{f}$.
	Now we construct a set of service token sample paths that must lead to throughput loss before the system returns to the starting state.
	First note that $\{\mathbf{f}:v_{\tilde{\balpha}}(\mathbf{f})>0\}$ is non-empty. To see this, let $f'_{j'k}$ equal to $1$ for some $j'$ and $k \notin \partial(j')$, and $0$ otherwise (such a pair $(j',k)$ exists by Assumption \ref{asm:non_trivial}). This $\mathbf{f}'$ results in a strictly positive\footnote{To see this, notice that $L_{\tilde{\balpha}}(\bx)>0$ for any $\bx\in\Omega\backslash\{\tilde{\balpha}\}$, hence it suffices to show that $\mathbf{0}\notin \mathcal{X}_{\bof'}$. Because for any $\Delta\bx\in\mathcal{X}_{\bof'}$, we have $\Delta x_k=f'_{j'k}>0$, hence $\mathbf{0}\notin \mathcal{X}_{\bof'}$. This concludes the proof.}
	$v_{\tilde{\balpha}}(\mathbf{f}')$.
	Therefore for any $\epsilon_2>0$ there exists service token arrival rate $\tilde{\mathbf{f}}$ such that	
	\begin{align*}
	v_{\tilde{\balpha}}(\tilde{\mathbf{f}})>0
	\qquad\text{and}\qquad
	\frac{\Lambda^*(\tilde{\mathbf{f}})}{v_{\tilde{\balpha}}(\tilde{\mathbf{f}})}\leq
	\inf_{\mathbf{f}:v_{\tilde{\balpha}}(\mathbf{f})>0}\frac{\Lambda^*(\mathbf{f})}{v_{\tilde{\balpha}}(\mathbf{f})}+\epsilon_2.
	\end{align*}
	It is not hard to show that $v_{\tilde{\balpha}}(\mathbf{f})$ is continuous in $\mathbf{f}$,	hence there exists $\epsilon_3>0$ such that for any $\hat{\mathbf{f}}:||\hat{\mathbf{f}}-\tilde{\mathbf{f}}||_{\infty}<\epsilon_3$, we have
	\begin{align*}
	v_{\tilde{\balpha}}(\hat{\mathbf{f}}) > (1-\epsilon_2)v_{\tilde{\balpha}}(\tilde{\mathbf{f}}) > 0 \, .
	\end{align*}
	Denote $T\triangleq \frac{1+\frac{\epsilon_1}{\min_{j:\alpha_j>0}\alpha_j}}{(1-\epsilon_2)v_{\tilde{\balpha}}(\tilde{\mathbf{f}})}$, define
	\begin{equation*}
	\mathcal{B}_{\tilde{\balpha}} \triangleq \left\{
	\fA(\cdot)\in C\left[0,T\right]\left|
	\sup_{t\in\left[0,T\right]}
	||\fA(t) - t\tilde{\mathbf{f}}||_{\infty} \leq \epsilon_3
	\right.\right\}.
	\end{equation*}
	For any service token arrival sample path $\fA(\cdot)\in\mathcal{B}_{\tilde{\balpha}}$, we will show that for $t\in[0,T]$ the followings are true: (i) normalized state $\fX^{K_{r'}}(t)$ does not hit $\tilde{\bX}^{K_{r'}}$ before any service token is wasted; (ii) at least one service token is wasted.
	
	To prove (i), define function $\tilde{L}_{\tilde{\balpha}}(\fX)\triangleq L_{\tilde{\balpha}}\left(\tilde{\balpha} + \fX - \tilde{\bX}^{K_{r'}} \right)$.
	By definition, we have $L_{\tilde{\balpha}}(\bx)>0$ for any $\bx\in\{\bx\in\mathbb{R}^m:\mathbf{1}^{\T}\bx=1\}\backslash\{\tilde{\balpha}\}$, hence we have that $\tilde{L}_{\tilde{\balpha}}(\fX)>0$ for any $\fX\in\Omega\backslash\{\tilde{\bX}^{K_{r'}}\}$.
	By inequality (\ref{eq:smallest_speed}), if no service token is wasted during $[0,T]$ we have:
	\begin{align*}
	\tilde{L}_{\tilde{\balpha}}\left(\bar{\X}^{K_{r'}}(t)\right)
	\geq
	t v\left(\frac{1}{t}\fA(t)\right)
	\geq
	t \min_{\fA(\cdot)\in\mathcal{B}}v\left(\frac{1}{t}\fA(t)\right)
	>
	t (1-\epsilon_2)v_{\tilde{\balpha}}(\tilde{\mathbf{f}})
	>0.
	\end{align*}
	
	We prove (ii) by contradiction. Suppose no service token is wasted given (fluid scale) service token arrival sample path $\fA(\cdot)\in\mathcal{B}$, then
	\begin{align*}
	\tilde{L}_{\tilde{\balpha}}\left(\bar{\X}^{K_{r'}}
	(T)\right)
	\geq
	T
	\min_{\fA(\cdot)\in\mathcal{B}}v\left(\frac{1}{T}\fA(T)\right)
	>
	\frac{1+\frac{\epsilon_1}{\min_{j:\alpha_j>0}\alpha_j}}{(1-\epsilon_2)v_{\tilde{\balpha}}(\tilde{\mathbf{f}})}
	(1-\epsilon_2)v_{\tilde{\balpha}}(\tilde{\mathbf{f}})
	=
	1+\frac{\epsilon_1}{\min_{j:\alpha_j>0}\alpha_j}.
	\end{align*}
	Expand the expression of $\tilde{L}_{\tilde{\balpha}}\left(\bar{\X}^{K_{r'}}
	(T)\right)$ on the LHS, we have
	\begin{align*}
	1 - \min_j \frac{\bar{X}_j^{K_{r'}}(T) + \left(\tilde{\alpha}_j - \tilde{x}_j^{K_{r'}}\right)}{\tilde{\alpha}_j}
	>
	1+\frac{\epsilon_1}{\min_{j:\alpha_j>0} \alpha_j}\, .
	\end{align*}
	Therefore
	\begin{align}
	\min\left\{\min_{j:\alpha_j=0}\frac{\bar{X}_j^{K_{r'}}(T)}{\tilde{\alpha}_j},
	\min_{j:\alpha_j>0}\frac{\bar{X}_j^{K_{r'}}(T)-\epsilon_1/2}{\tilde{\alpha}_j}
	\right\}
	\leq\ &
	\min_j \frac{\bar{X}_j^{K_{r'}}(T) + \left(\tilde{\alpha}_j - \tilde{x}_j^{K_{r'}}\right)}{\tilde{\alpha}_j}\nonumber\\
	<\ &
	-\frac{\epsilon_1}{\min_{j:\alpha_j>0} \alpha_j}.\label{eq:converse_aux_1}
	\end{align}
\yka{The equation is spilling over.}
	Note that the first inequality in (\ref{eq:converse_aux_1}) holds because of (\ref{eq:alpha_cluster_zero}) and (\ref{eq:alpha_cluster}).
	Inequality (\ref{eq:converse_aux_1}) implies that $\min_j \bar{X}_j^{K_{r'}}(T) < 0$, which is impossible as queue lengths must be non-negative.
	\ \\
	
	\noindent\textit{Step 3: Asymptotic steady-state lower bound on throughput-loss probability.}
	We use renewal-reward theorem to lower bound the throughput loss probability.\yka{What is the renewal-reward theorem? Is there a citation?}
	Consider the regenerative process that restarts each time $\fX^{K_{r'}}(t) = \tilde{\bX}^{K_{r'}}$.
	Without loss of generality, let $\bar{\X}^{K_{r'}}(0) = \tilde{\bX}^{K_{r'}}$. 
	Recall the definition of $\tau^K$ in \eqref{eq:cycle-time-converse-bound}.
	Using the claim in step 1 and the result in step 2, we have:
	\begin{align*}
	\mathbb{P}^{K_{r'},U}
	& =\ \frac{\mathbb{E}\left[\#\{\textup{service token wasted during $[0,\tau]$}\}\right]}{\mathbb{E}[\tau]}
	\\
	&\geq\
	{K_{r'}}^{-m} (\mathbf{1}^{\T}{\bphi}\mathbf{1})
	\mathbb{E}\left[\#\{\textup{service token wasted during $[0,\tau]$}\}\right]\\
	&\geq\
	{K_{r'}}^{-m} (\mathbf{1}^{\T}{\bphi}\mathbf{1})
	\mathbb{P}\left(\#\{\textup{service token wasted during $[0,\tau]$}\}\geq 1\right)
	\\
	&\geq\
	{K_{r'}}^{-m} (\mathbf{1}^{\T}{\bphi}\mathbf{1})
	\mathbb{P}\left(\fA^{K_{r'}}(\cdot)\in\mathcal{B}_{\tilde{\balpha}}\right).
	\end{align*}
	Take asymptotic limit on both sides, we have:
	\begin{align*}
	\liminf_{r'\to\infty}\frac{1}{K_{r'}}\log\mathbb{P}^{K_{r'},U}
	&\geq\
	\liminf_{r'\to\infty}\frac{1}{K_{r'}}\log\mathbb{P}\left(\fA^{K_{r'}}(\cdot)\in\mathcal{B}_{\tilde{\balpha}}\right)\\
	&\stackrel{(a)}{\geq}\ -\inf_{\fA(\cdot)\in\mathcal{B}_{\tilde{\balpha}}^{o}\cap\textup{AC}[0,T]}\int_{0}^{T}\Lambda^*\left(\dot{\fA}(t)\right)dt\\
	&\stackrel{(b)}{\geq}\ -T \Lambda^*(\tilde{\mathbf{f}})\\
	&=\  -\frac{1+\frac{\epsilon_1}{\min_{j:\alpha_j>0}\alpha_j}}{(1-\epsilon_2)v_{\tilde{\balpha}}(\tilde{\mathbf{f}})}\Lambda^*(\tilde{\mathbf{f}})\\
	&\geq\  -\frac{1+\frac{\epsilon_1}{\min_{j:\alpha_j>0}\alpha_j}}{1-\epsilon_2}\left(\inf_{\mathbf{f}:v_{\tilde{\balpha}}(\mathbf{f})>0}\frac{\Lambda^*(\mathbf{f})}{v_{\tilde{\balpha}}(\mathbf{f})}+\epsilon_2\right).
	\end{align*}
	Here (a) holds because of Mogulskii's Theorem (Fact \ref{fact:sample_path_ldp}), (b) holds because service token sample path $\fA(t)=t\tilde{\mathbf{f}}\in \textup{AC}[0,T]$ is a member of $\mathcal{B}_{\tilde{\balpha}}$.
	For any $\delta>0$, by choosing small enough $\epsilon_1(\delta),\epsilon_2(\delta) > 0$, we have
	\begin{align*}
	-\liminf_{r'\to\infty}\frac{1}{K_{r'}}\log\mathbb{P}^{K_{r'},U}
	\leq
	(1+\delta)(\gamma_{\textup{\tiny CB}}(\tilde{\balpha}(\delta)) + \delta).
	\end{align*}
Here the choice of $\tilde{\balpha}$ depends on $\delta$.
	To get rid of the multiplicative term $(1+\delta)$, it suffices to show that $\sup_{\balpha\in\textup{relint}(\Omega)}\gamma_{\textup{\tiny CB}}(\balpha)<\infty$.
	This can be proved by the following construction: let $\fA(t)=t\mathbf{f}'$ for $t\in[0,1]$ where $f_{j'k}=1$ for some $j'\in V_S$ and $k\notin \partial(j')$. 
	Because $\gamma_{\textup{\tiny CB}}(\balpha)$ is defined by an infimum $\gamma_{\textup{\tiny CB}}(\balpha) \triangleq
	 \inf_{\mathbf{f} \in \mathbb{R}_+^{nm}:v_{\alpha}(\mathbf{f})>0}\,
	 \frac{\Lambda^*(\mathbf{f})}{v_{\balpha}(\mathbf{f})}$, we have $\gamma_{\textup{\tiny CB}}(\balpha)\leq \frac{\Lambda^*(\bof')}{v_{\tilde{\balpha}}(\mathbf{f}')}$. By definition, $v_{\tilde{\balpha}}(\bof')
	 =
	 1-\max_{\Delta\bx\in\mathcal{X}_{\bof'}}\min_{i} \frac{\tilde{\alpha}_i+\Delta x_i}{\tilde{\alpha}_i}
	 =
	 -\max_{\Delta\bx\in\mathcal{X}_{\bof'}}\min_{i} \frac{\Delta x_i}{\tilde{\alpha}_i}
	 $.
	 Note that
	 \begin{align*}
	 	\mathcal{X}_{\bof'}=\ 
	 	\{
	 	\Delta\bx\in\mathbb{R}^{|V_B|}:
	 	\sum_{i\in\partial(j')}\Delta x_i=-1\, ,
	 	\Delta x_i\leq 0\textup{ for }i\in\partial(j')\, ,
	 	\Delta x_k=1\,  \\
	 	\Delta x_i=0\textup{ for }i\notin \partial(j')\cup\{k\}
	 	\}\, .
	 \end{align*}
	 Therefore
	 \begin{align*}
	 	\max_{\Delta\bx\in\mathcal{X}_{\bof'}}\min_{i\in V_B} \frac{\Delta x_i}{\tilde{\alpha}_i}
	 	=\ \max_{\Delta\bx\in\mathcal{X}_{\bof'}}\min_{i\in\partial(j')} \frac{\Delta x_i}{\tilde{\alpha}_i}
	 	\leq\ \max_{\Delta\bx\in\mathcal{X}_{\bof'}}\min_{i\in\partial(j')}\Delta x_i
	 	\leq\ -\frac{1}{|\partial(j')|}
	 	\leq\ -\frac{1}{m}\, .
	 \end{align*}
	 Hence
	 $
	 v_{\tilde{\balpha}}(\bof')
	 \geq 
	 \frac{1}{m}$. \yka{I'm not clear why the first inequality holds.}
	Hence $\gamma_{\textup{\tiny CB}}(\alpha)\leq \frac{\Lambda^*(\bof')}{v_{\tilde{\balpha}}(\mathbf{f}')}\leq m\Lambda^*(\mathbf{f}') <\infty$.
	Therefore by choosing a small enough $\delta$, we have
	\begin{align*}
	-\liminf_{r'\to\infty}\frac{1}{K_{r'}}\log\mathbb{P}^{K_{r'},U}
	\leq
	\gamma_{\textup{\tiny CB}}(\tilde{\balpha}(\epsilon)) + \epsilon.
	\end{align*}
	By the definition of subsequence $\{K_{r'}\}$, we have
\begin{align*}
	-\liminf_{K\to\infty}\frac{1}{K}\log\mathbb{P}^{K,U}
	\leq
	\gamma_{\textup{\tiny CB}}(\tilde{\balpha}(\epsilon)) + \epsilon.
	\end{align*}
As a result, for any $\epsilon>0$ there exists $\balpha\in\Omega$ such that $-\liminf_{K\to\infty}\frac{1}{K}\log\mathbb{P}^{K,U}
		\leq  \gamma_{\textup{\tiny CB}}(\balpha) + \epsilon$, therefore
$-\liminf_{K\to\infty}\frac{1}{K}\log\mathbb{P}^{K,U}
		\leq \sup_{\balpha \in \textup{relint}(\Omega)} \gamma_{\textup{\tiny CB}}(\balpha)$.
\end{proof}

\section{Sufficient conditions for optimality: proof of Proposition \ref{prop:tight_converse}}\label{appen:sec:sufficient}
The proof of Proposition \ref{prop:tight_converse} consists of two parts.
We first derive an achievability bound for policies that, {for a given $\balpha \in \textup{relint}(\Omega)$, satisfy the negative drift property in Proposition \ref{prop:tight_converse}; we then show it matches the converse bound in Lemma~\ref{lem:point_wise_converse} for that specific $\balpha$ (i.e., $\gamma_{{\textup{\tiny CB}}}(\balpha)$)} if the steepest descent property in Proposition~\ref{prop:tight_converse} is also satisfied.

\yka{Delete ``Recall the link between the exponent and the value of the corresponding differential game mentioned earlier in this section.
Since $\gamma_{{\textup{\tiny AB}}}=\gamma_{{\textup{\tiny CB}}}$, inequality (\ref{eq:sub_add_ineq}) should hold with equality for any exponent optimal policy and its corresponding most likely sample path.
Note that \textit{radial} sample paths $\fX^{U}$ starting at $\balpha$ makes (\ref{eq:sub_add_ineq}) equality,
hence a `weak' adversary who only chooses between time invariant arrival patterns should be already optimal under an exponent optimal policy.
We will formalize the above statement in Section \ref{subsec:smw_exponent_subset} and explicitly characterize the most likely sample paths.'' Where was $\gamma_{{\textup{\tiny AB}}}$ defined?}
\yka{Doesn't this argument imply/suggest why the radial converse in Lemma~\ref{lem:point_wise_converse} is tight? The second inequality should be equality for a radial sample path. The first inequality should be = above and $\geq$ for an arbitrary policy. I want to provide intuition here reg. \emph{why} radial sample paths give a tight converse. (see my earlier comment.)}

\yka{The requirement $\V_{\balpha}(\bar{\X}[t])=v$ seems unnecessary and could be confusing. If it suffices that $\V_{\balpha}(\bar{\X}[t])<1$ then why not write it that way?}
\yka{There is a notation conflict between $\gamma(\balpha)$ here and that in the achievability part of the main theorem, currently (11). I suggest we use some other notation here, e.g., $\gamma_{{\textup{\tiny LD}}}(\balpha)$ or $\gamma_{{\textup{v}}}(\balpha)$ (v for variational), and later establish that $\gamma(\balpha)=\gamma_{{\textup{\tiny LD}}}(\balpha)$.}

\subsection{An achievability bound}\label{appen:subsec:achievability}
The following lemma is an adaptation of Theorem 5 and Proposition 7 in \cite{venkataramanan2013queue} to our setting.
It gives the achievability bound for the exponent of the steady state throughput-loss probability, for any policy such that the negative drift condition in Proposition \ref{prop:tight_converse} is met for $L_{\balpha}(\cdot)$ where $\balpha\in\textup{relint}(\Omega)$.
The main technical difficulty comes from the fact that it characterizes the \textit{steady state} of the system.
The analysis uses Freidlin-Wentzell theory and follows from \cite{stolyar2003control,venkataramanan2013queue}.
While the main proof idea follows that in \cite{venkataramanan2013queue}, we refine the results there by dropping the assumption that all FSPs are Lipschitz continuous with a universal Lipschitz constant. This allows us to deal with Poisson-driven service token arrival processes which does not satisfy this assumption.

\yka{But the lemma statement doesn't assume anything about the policy $U$ satisfying the conditions of Prop 3!}
\yka{Suddenly why does $\Omega$ depend on $K$? We are working with scaled states throughout so write ``given scaled initial state $\bz$''.}


\begin{lem}[Achievability bound]\label{lem:lower_bound_vj}
	For the system being considered, if policy $U$ satisfies the negative drift condition in Proposition \ref{prop:tight_converse} for $L_{\balpha}(\cdot)$ where $\balpha\in\textup{relint}(\Omega)$,
we have (the subscript ``AB'' stands for achievability bound)\yka{What are the assumptions on $U$? What is $\alpha$?}\yka{The notation is inconsistent. In \eqref{eq:ub_performance_measure} you used $P_{\textup{p}}^{K,U}$. Later in the proof, you use $\prob$ to denote the steady state probability but you don't clarify that this is what you mean.}
	\begin{equation}\label{eq:lower_bound_vj}
	-\limsup_{\N\to\infty} \frac{1}{\N}\log
	\mathbb{P}^{K,U}
	\geq
	\gamma_{\textup{\tiny AB}}(\balpha)\, .
	\end{equation}
	Here for fixed\,\footnote{The definition of quantity $\gamma_{{\textup{\tiny AB}}}(\balpha)$ is based on the local behavior of $\fA$ and $\fX$ for times close to $t$. In particular, the value of $T$ plays no role.} $T>0$,
	\begin{equation*}
	\begin{split}
	\gamma_{\textup{\tiny AB}}(\alpha)\triangleq & \inf_{v>0,\bof,\fA,\fX}\frac{\Lambda^*(\bof)}{v}\, ,\\
	where \ &(\fA,\fX)\text{ is a FSP on $[0,T]$ under $U$ such that for some regular }t\in[0,T]\\
	&\dot{\fA}(t) = \bof\, ,\quad L_{\balpha}(\fX(t))<1 \, ,\quad
	\dot{L}_{\balpha}(\fX(t)) = v\, .
	\end{split}
	\end{equation*}
\end{lem}
The proof of Lemma \ref{lem:lower_bound_vj} is similar to the proofs of Theorem 5 and Proposition 7 in \cite{venkataramanan2013queue}.
In the interest of space, we refer the interested readers to Supplementary Appendix \ref{online_appen_main_proof} for the full proof.

\subsection{Converse bound matches achievability bound}
{In Lemma \ref{lem:point_wise_converse} we obtain a converse bound which holds for any state-dependent policy. However, for a given policy $U$ can we obtain a tighter policy-specific converse bound? In the following Lemma, we show that for policies that satisfy the negative drift property in Proposition \ref{prop:tight_converse} for Lyapunov function $L_{\balpha}(\cdot)$ where $\balpha \in \textup{relint}(\Omega)$, there is a tighter converse bound given by $\gamma_{{\textup{\tiny CB}}}(\balpha)$.}

\begin{lem}\label{lem:fluid_resting_point}
	For policies $U\in\mathcal{U}$ that satisfy the negative drift condition in the statement of Proposition \ref{prop:tight_converse} for $\balpha \in\textup{relint}(\Omega)$, we have\yka{This is very confusing. For which $\alpha$ does this hold? Given that we already have Lemma \ref{lem:point_wise_converse}, why do we need this lemma?}
	\begin{align*}
		-\liminf_{K\to\infty} \frac{1}{K}\log \mathbb{P}^{K,U} \leq \gamma_{{\textup{\tiny CB}}}(\balpha)\, .
	\end{align*}
\yka{Really? You have equality without needing the steepest descent condition? More likely you intended to claim an inequality $ \leq$.}
\end{lem}

\begin{proof}
	The following proof is very similar to the proof of Lemma \ref{lem:point_wise_converse}.
We will emphasize the parts that are different and skip the repetitive arguments.
In the proof of Lemma \ref{lem:point_wise_converse}, we divide the process into cycles and apply the renewal-reward theorem. We follow the same approach here except that we define the cycles differently.

\noindent \textit{Step 1: Show that $\balpha$ is the ``resting point'' of $U$.}
Fix $\epsilon_1>0$ and define
\begin{align*}
	\tau^{K} \triangleq \inf\left\{ t\geq0: L_{\balpha}(\fKX(t))\leq\epsilon_1 \right\}\, .
\end{align*}
Using the argument in Step 2b(i) of the proof of Lemma \ref{lem:lower_bound_vj}, we can show that there exists $K_0 = K_0(\epsilon_1)>0$ and constant $C>0$ such that for $K\geq K_0$,
	\begin{align*}
	\sup_{\bx\in\Omega}\mathbb{E}\left(\tau^{K}|\fKX(0)=\bx \right) \leq C\, .
	\end{align*}
	In other words, starting from any state, the expected time for the system state to reach the $O(\epsilon_1)$-neighborhood of $\balpha$ is bounded from above by a constant.

\noindent\textit{Step 2: Lower bound the throughput-loss probability.} Proceed exactly as Step 2 and Step 3 in the proof of Lemma \ref{lem:point_wise_converse}, we explicitly construct a service token sample path that guarantees throughput loss within $\Theta(1)$ units of time given the starting state satisfies $L_{\alpha}\left(\fKX(T+\tau^K)\right) < \epsilon_1$.
Then we obtain the result.
\end{proof}

Now we combine Lemma \ref{lem:lower_bound_vj} and Lemma \ref{lem:fluid_resting_point} to prove Proposition \ref{prop:tight_converse} by showing that $\gamma_{{\textup{\tiny AB}}}(\balpha) = \gamma_{{\textup{\tiny CB}}}(\balpha)$.
Lemma \ref{lem:key_property_lyap} and the steepest descent property in Proposition \ref{prop:tight_converse} are crucial in showing $\gamma_{{\textup{\tiny AB}}}(\balpha) \geq \gamma_{{\textup{\tiny CB}}}(\balpha)$ (the other direction is obvious).

\begin{proof}[Proof of Proposition~\ref{prop:tight_converse}]
Let $U\in\mathcal{U}$ satisfy the conditions in Proposition \ref{prop:tight_converse}.
Then for regular $t$ we have
\begin{align*}
\dot{L}_{\balpha}(\fX(t))
&\leq\; \inf_{U'\in\mathcal{U}}\left\{\dot{L}_{\balpha}(\fX^{U'}(t))
\left|\dot{\fA}'(t) = \mathbf{f} \right.
\right\}&\textup{(steepest descent)}\\
&=\; \min_{\Delta\bx\in\mathcal{X}_{\mathbf{f}}}
\lim_{\Delta t\to 0}\frac{L_{\balpha}(\bar{\bX}^{U'}(t)+\Delta\bx\Delta t)-
	L_{\balpha}(\fX^{U'}(t))
}{\Delta t}\\
&\leq\; \min_{\Delta\bx\in\mathcal{X}_{\mathbf{f}}}\lim_{\Delta t\to 0}\frac{L_{\balpha}(\balpha+\Delta\bx\Delta t)
}{\Delta t} &\textup{(sub-additivity, Lemma \ref{lem:key_property_lyap})}\\
&=\; \min_{\Delta\bx\in\mathcal{X}_{\mathbf{f}}}L_{\balpha}(\balpha+\Delta\bx)=v_{\balpha}(\mathbf{f})\, .&\textup{(scale-invariance, Lemma \ref{lem:key_property_lyap})}
\end{align*}
Let $v=\dot{L}_{\balpha}(\fX(t))$,
from $v \leq v_{\balpha}(\mathbf{f})$ we have $\{v>0\}\subset \{v_{\balpha}(\mathbf{f})> 0\}$, hence using Lemma \ref{lem:lower_bound_vj} we have
\begin{align*}
	\gamma_{{\textup{\tiny AB}}}(\balpha)
	=
	\inf_{v>0,\bof,\fA,\fX}\frac{\Lambda^*(\bof)}{v}
	\geq
	\inf_{\mathbf{f}:v_{\balpha}(\mathbf{f})>0}\frac{\Lambda^*(\mathbf{f})}{v_{\balpha}(\mathbf{f})}
	=
	\gamma_{{\textup{\tiny CB}}}(\balpha)\, .
\end{align*}
But since by Lemma \ref{lem:fluid_resting_point} we know $\gamma_{{\textup{\tiny CB}}}(\balpha)$ is a converse bound for policy $U$, hence $\gamma_{{\textup{\tiny AB}}}(\balpha)\leq \gamma_{{\textup{\tiny CB}}}(\balpha)$. Therefore $\gamma_{{\textup{\tiny AB}}}(\balpha)= \gamma_{{\textup{\tiny CB}}}(\balpha)$.
\end{proof}

\section{SMW policies and explicit exponent}\label{appen:sec:smw}

Appendix~\ref{appen:sec:smw} shows that the SMW policy satisfies the sufficient conditions for exponent optimality, and derives explicitly the optimal exponent and most-likely sample paths, including the proofs of Lemma \ref{lem:lyapunov_derivative}, Lemma \ref{lem:lyapunov_derivative_fluid}, and Lemma \ref{lem:explicit_gamma}. 

\subsection{Lyapunov drift of FSPs under SMW: proof of Lemma \ref{lem:lyapunov_derivative}}
{In this subsection we prove Lemma \ref{lem:lyapunov_derivative} which establishes that SMW($\balpha$) policies perform steepest descent on $L_{\balpha}(\cdot)$.}
\begin{proof}[Proof of Lemma \ref{lem:lyapunov_derivative}]
	For notation simplicity, we will write $S_1(\fX(t))$ as $S_1$, $S_2\left(\fX(t),\dot{\fX}(t)\right)$ as $S_2$, and 
	$
	\min_{k\in S_1}\frac{\dot{\bar{X}}_k(t)}{\alpha_k}
	$ as $c$
	in the following.
	Let $(\fA,\fX)$ be an FSP under policy $U\in\mathcal{U}$.
	\begin{compactitem}[leftmargin=*]
		\item \textit{Proof of (\ref{eq:lyap_FSP_1}).} 
		Note that $t$ is a \emph{regular} time, hence $L_{\balpha}(\fX(\cdot))$ and $\fX(\cdot)$ are differentiable at $t$. It follows from the definition of derivatives that $\dot{L}_{\balpha}(\fX(t))$ is determined by the queues in $S_2$ alone, hence we have $\dot{L}_{\balpha}(\fX(t))=\ -\min_{k\in S_1}\frac{\dot{\bar{X}}_k(t)}{\alpha_k}= -c$.
		\item \textit{Proof of (\ref{eq:lyap_FSP_2}).}	
		For the $K$-th system, define auxiliary processes:
		\begin{align*}
			 \bar{E}_{ij'k}^{K,U}(t) \triangleq\ \# &\left\{ \textup{Type $(j',k)$ service tokens that arrive during $[0,t]$} \right. \\
			&\left. \textup{and serve jobs at $i$ under policy $U\in \cU$} \right\}\quad i,k\in V_B\, ,\, j'\in V_S\, .
		\end{align*}
		Using standard argument \citep[see, e.g.,][]{dai2005maximum}, we can extend the definition of FSP (Definition \ref{def:FSP}) to $(\fA(\cdot),\fX(\cdot),\fE(\cdot))$, where a subsequence of $\fE^{K,U}(\cdot)$ converges u.o.c. to $\fE(\cdot)$.
		We focus on the regular times $t$ where $\dot{\fE}(t)$ exists, which includes almost all regular times because $\dot{\fE}(t)$ is differentiable almost everywhere.
		
		Consider any non-idling policy $U'\in\cU$, and $\fX^{U'}(t)$ such that $\fX^{U'}(t)\neq \balpha$, $L_{\balpha}(\fX^{U'}(t))<1$.
		The flow of jobs entering $S_2$ is $\sum_{j'\in V_S,k\in S_2}\dot{\bar{A}}_{j'k}(t)$ because $U'$ is non-idling.
		The flow of jobs leaving $S_2$ is at least $\sum_{j'\in V_S:\partial(j')\subset S_2,k\in V_B}\dot{\bar{A}}_{j'k}(t)$ because $U'$ is non-idling and that the jobs in $V_B\backslash S_2$ cannot be served by service tokens originating from $\{j'\in V_S:\partial(j')\subset S_2\}$. Therefore,
		\begin{align}\label{eq:deri_lyap_ub}
		\sum_{k\in S_2}\dot{\bar{X}}_k^{U'}(t)
		\leq\ 
		\sum_{j'\in V_S,k\in S_2}\dot{\bar{A}}_{j'k}(t)
		-
		\sum_{j'\in V_S:\partial(j')\subset S_2,k\in V_B}\dot{\bar{A}}_{j'k}(t)\, .
		\end{align}
		
		Now we consider SMW($\balpha$) policies and $\fX^{\textup{SMW}(\balpha)}(t)$ such that $\fX^{\textup{SMW}(\balpha)}(t)\neq \balpha$.
		For the process $\fE(t)$ (resp. $\fX(t)$), we use notation $\Delta\fE(t)$ (resp. $\Delta\fX(t)$) to denote $\fE(t+\Delta t) - \fE(t)$ (resp. $\fX(t+\Delta t) - \fX(t)$).
		It holds that
		\begin{align*}
		\sum_{k\in S_2}
		\Delta\bar{X}_k^{K,U}(t)
		=\ 
		&\sum_{j'\in V_S,k\in S_2}
		\sum_{i\in \partial(j')}
		\Delta\bar{E}^{K,U}_{ij'k}(t)
		-
		\sum_{i\in S_2,k\in V_B}
		\sum_{j'\in \partial(i)} \Delta\bar{E}^{K,U}_{ij'k}(t)\, .
		\end{align*}
		For regular $t$, it follows from the definition of derivative that
		\begin{align*}
		\sum_{k\in S_2}
		\dot{\bar{X}}_k^{U}(t)
		=\ 
		&\sum_{j'\in V_S,k\in S_2}
		\sum_{i\in \partial(j')}
		\dot{\bar{E}}^{U}_{ij'k}(t)
		-
		\sum_{i\in S_2,k\in V_B}
		\sum_{j'\in \partial(i)} \dot{\bar{E}}^{U}_{ij'k}(t)\, .
		\end{align*}
		
		For SMW($\balpha$) policy, using exactly the same argument as in Lemma 4 of \cite{dai2005maximum}, we have
		\begin{align}\label{eq:smw-fluid-eq}
			\dot{\bar{E}}_{ij'k}^{\textup{SMW}(\balpha)}(t)=\ 0
			\qquad
			\textup{if }\frac{\bar{X}^{\textup{SMW}(\balpha)}_i(t)}{\alpha_i} <\ \max_{\ell\in\partial(j')}\frac{\bar{X}^{\textup{SMW}(\balpha)}_{\ell}(t)}{\alpha_{\ell}}\, .
		\end{align}
		
		By definition of $S_2$, there exists $\epsilon>0$ such that any (scaled) queue length in $S_2$ is strictly smaller than all (scaled) queue lengths in $V_B\backslash S_2$ in $(t,t+\epsilon)$, which also implies that the queue lengths in $V_B\backslash S_2$ remain strictly positive during $(t,t+\epsilon)$.
		Apply \eqref{eq:smw-fluid-eq}, we know that the system will serve the jobs within $V_B\backslash S_2$ by service tokens arriving at $\partial(V_B\backslash S_2)$ during $(t,t+\epsilon)$.
		Hence we have
		\begin{align*}
		\sum_{k\in V_B\backslash S_2}
		\dot{\bar{X}}_k^{\textup{SMW($\alpha$)}}(t)
		=\ 
		&\sum_{j'\in V_S,k\in V_B\backslash S_2} \sum_{i\in \partial(j')}\dot{\bar{E}}^{\textup{SMW($\alpha$)}}_{ij'k}(t)
		-
		\sum_{j'\in \partial(V_B\backslash S_2),k\in V_B} \dot{\bar{A}}_{j'k}(t)\\
		\leq\ 
		&\sum_{j'\in V_S,k\in V_B\backslash S_2} \dot{\bar{A}}_{j'k}(t)
		-
		\sum_{j'\in \partial(V_B\backslash S_2),k\in V_B} \dot{\bar{A}}_{j'k}(t)\, .
		\end{align*}
		
		Since it is a closed system, we have:
		\begin{align}
		\sum_{k\in  S_2}\dot{\bar{X}}_k^{\textup{SMW($\alpha$)}}(t)		
		=\; &-\sum_{k\in V_B\backslash S_2} \dot{\bar{X}}_k^{\textup{SMW($\alpha$)}}(t) \nonumber\\
		\geq\; &\sum_{j'\in \partial(V_B\backslash S_2),k\in V_B} \dot{\bar{A}}_{j'k}(t)
		-
		\sum_{j'\in V_S,k\in V_B\backslash S_2}\dot{\bar{A}}_{j'k}(t)\, .\label{eq:deri_lyap_lb}
		\end{align}
		\medskip

		Note that
		\begin{align*}
		\textup{RHS of \eqref{eq:deri_lyap_lb}}
		=&\; \sum_{j'\in \partial(V_B\backslash S_2),k\in V_B}\dot{\bar{A}}_{j'k}(t)
		-
		\sum_{j'\in V_S,k\in V_B\backslash S_2}\dot{\bar{A}}_{j'k}(t)\\
		= &\; \left(
		\sum_{j'\in V_S,k\in V_B}\dot{\bar{A}}_{j'k}(t)
		-
		\sum_{j':\partial(j')\subset S_2,
			k\in V_B}
		\dot{\bar{A}}_{j'k}(t)
		\right)
		-
		\left(
		\sum_{j'\in V_S,k\in V_B}\dot{\bar{A}}_{j'k}(t)
		-
		\sum_{j'\in V_S,
			k\in S_2}
		\dot{\bar{A}}_{j'k}(t)
		\right)\\
		=&\; \sum_{j'\in V_S,
			k\in S_2}
		\dot{\bar{A}}_{j'k}(t)
		-
		\sum_{j':\partial(j')\subset S_2,
			k\in V_B}
		\dot{\bar{A}}_{j'k}(t) \\
		=&\; \textup{RHS of \eqref{eq:deri_lyap_ub}}\, .
		\end{align*}
		Finally, observe that for any $k\in S_2$,
		\begin{align}\label{eq:deri_lyap_final_step}
		\dot{L}_{\balpha}(\fX(t))
		=
		-\frac{\dot{\bar{X}}_k^{U'}(t)}{\alpha_k}
		=
		-\frac{1}{\mathbf{1}_{S_2}^{\T}\balpha}
		\sum_{k\in S_2}\alpha_k\frac{\dot{\bar{X}}_k^{U'}(t)}{\alpha_k}
		=
		-\frac{1}{\mathbf{1}_{S_2}^{\T}\balpha}
		\sum_{k\in S_2}\dot{\bar{X}}_k^{U'}(t)\, .
		\end{align}
		Plug (\ref{eq:deri_lyap_lb}) and (\ref{eq:deri_lyap_ub}) into (\ref{eq:deri_lyap_final_step}), we know that
		inequality (\ref{eq:lyap_FSP_1}) holds, and it becomes equality for SMW($\balpha$) policy.
	\end{compactitem}
\end{proof}

\subsection{Lyapunov drift of Fluid Limits under SMW: proof of Lemma~\ref{lem:lyapunov_derivative_fluid} }

\begin{proof}[Proof of Lemma~\ref{lem:lyapunov_derivative_fluid}]
	\noindent\textbf{Negative drift.}
	Let $(\fA,\fX)$ be a fluid limit of the system under SMW($\balpha$), and $t$ be its regular point.
	Simply plug in Lemma \ref{lem:lyapunov_derivative}, and replace $\dot{\bar{A}}_{j'k}(t)$ with ${\phi}_{j'k}$, we have ($S_2$ is defined in Lemma \ref{lem:lyapunov_derivative}, $S_2\neq\emptyset$)
	\begin{equation*}
	\begin{split}
	\dot{L}_{\balpha}(\fX(t))
	=\; &
	-\frac{1}{\mathbf{1}^{\T}_{S_2}\balpha}
	\left(\sum_{j'\in V_S,k\in S_2}\dot{\bar{A}}_{j'k}(t)
	-
	\sum_{j'\in V_S:\partial(j')\subseteq S_2,k\in V_B}\dot{\bar{A}}_{j'k}(t)\right)\\
	\leq\; &
	-\min_{S_2\subsetneq V_B,S_2\neq\emptyset}
	\frac{1}{\mathbf{1}^{\T}_{S_2}\balpha}
	\left(\sum_{j'\in V_S,k\in S_2}{\phi}_{j'k}
	-
	\sum_{j'\in V_S:\partial(j')\subseteq S_2,k\in V_B}{\phi}_{j'k}\right)\\
	\leq\; &
	-\min_{S_2\subsetneq V_B,S_2\neq\emptyset}
	\left(\sum_{j'\in V_S,k\in S_2}{\phi}_{j'k}
	-
	\sum_{j'\in V_S:\partial(j')\subseteq S_2,k\in V_B}{\phi}_{j'k}\right)\\
	\stackrel{\textup(a)}{\leq}\; & -\min\{\xi,{\phi}_{\textup{min}}\}\, .
	\end{split}
	\end{equation*}
Here (a) holds for the following reason. First note that when $\fX(t)\neq\balpha$, we have $S_2\neq V_B$.
	Let $J\triangleq \{j'\in V_S:\partial(j')\subset S_2;\exists k\in V_B\backslash S_2\textup{ s.t. }\phi_{j'k}>0\}$. 
	If $J=\emptyset$, we have
	\begin{align*}
	\sum_{j'\in V_S,k\in S_2}{\phi}_{j'k}
	-
	\sum_{j'\in V_S:\partial(j')\subseteq S_2,k\in V_B}{\phi}_{j'k}
	=\ &\sum_{j'\in V_S:\partial(j')\cap(V_B\backslash S_2)\neq\emptyset,k\in S_2}{\phi}_{j'k}
	-
	\sum_{j'\in V_S:\partial(j')\subseteq S_2,k\in V_B\backslash S_2}{\phi}_{j'k}\\
	\geq\ &
	\sum_{j'\in V_S:\partial(j')\cap(V_B\backslash S_2)\neq\emptyset,k\in S_2}{\phi}_{j'k}
	\geq
	{\phi}_{\textup{min}}\, ,
	\end{align*}
	where
	$
	{\phi}_{\textup{min}}\triangleq \min_{j'\in V_B,k\in V_B,{\phi}_{j'k}>0}{\phi}_{j'k}
	$
	is the minimum positive arrival rate for any service token type $(j',k)$ (the last inequality holds because of Assumption 1).
	If $J\neq\emptyset$, we must have $J\in\mathcal{J}$, hence
	$$
	\sum_{j'\in V_S,k\in S_2}{\phi}_{j'k}
	-
	\sum_{j'\in V_S:\partial(j')\subseteq S_2,k\in V_B}{\phi}_{j'k}
	\geq
	\sum_{j'\in V_S,k\in \partial(J)}{\phi}_{j'k}
	-
	\sum_{j'\in J,k\in V_B}{\phi}_{j'k}
	\geq
	\xi\, ,
	$$
	where
	$
	\xi
	\triangleq \min_{J\in\mathcal{J}}
	\left(\sum_{i\in \partial(J)}\mathbf{1}^{\T}{\bphi}_{(i)} - \sum_{j'\in J}\mathbf{1}^{\T}{\bphi}_{j'}\right)>0
	$
	is the Hall's gap  of the system.

		\noindent \textbf{Robustness of drift. }
		Define
		\begin{equation*}
		G(\bof) \triangleq \min_{S\subsetneq V_B,S\neq\emptyset}
		\left(\sum_{j'\in V_S,k\in S} f_{j'k}
		-
		\sum_{j':\partial(j')\subseteq S,k\in V_B} f_{j'k}\right)\, .
		\end{equation*}
		Note that $G(\bof)$ is continuous in $\bof$.
		Since $G({\bphi})\leq -\min\{\xi,{\phi}_{\textup{min}}\}<0$, by continuity there exists $\epsilon$ such that for any $\dot{\fA}(t)\in B({\bphi},\epsilon)$,
		\begin{equation*}
		\dot{L}_{\balpha}(\fX(t))=
		G\left(\dot{\fA}(t)\right) \leq -\frac{1}{2}\min\{\xi,{\phi}_{\textup{min}}\}\, .
		\end{equation*}
\end{proof}

\subsection{Explicit exponent and most likely sample path: proof of Lemma \ref{lem:explicit_gamma}}
\begin{proof}[Proof of Lemma \ref{lem:explicit_gamma}]
	\noindent\textbf{Explicit exponent.}
Let $(\fA(\cdot), \fX(\cdot))$ be a fluid sample path under SMW($\balpha$).
For a regular point $t$ of this FSP, denote $\bof\triangleq \dot{\fA}(t)$.

For notation simplicity, for $S\subset V_B$ denote
\begin{align*}
	\textup{gap}_{S}(\bof)
	\triangleq
	\sum_{j':\partial(j')\subseteq S,k\in V_B} f_{j'k}
	-\sum_{j'\in V_S,k\in S} f_{j'k}\, .
\end{align*}
In words, $\textup{gap}_{S}(\bof)$ is the minimum net rate at which jobs in $S$ is drained given current service token arrival rate $\bof$, assuming no service token is wasted.
Using the result of Lemma \ref{lem:lyapunov_derivative}, we have:
\begin{align}
	\dot{L}_{\balpha}(\fX(t))
	=\
	\frac{\textup{gap}_{S_2}(\bof)}{\mathbf{1}_{S_2}^{\T}\balpha}\, ,
\label{eq:barv-def}
\end{align}
where $S_2\triangleq S_2(\fX(t),\dot{\fX}(t))$ and the latter is defined in Lemma \ref{lem:lyapunov_derivative}.
Given $\dot{\fA}(t)=\bof$, we define
\begin{align*}
    \bar{v}(\bof)
	\triangleq\ 
	\sup_{\fX(t)\in\Omega\backslash\{\balpha\}}\dot{L}_{\balpha}(\fX(t))
    =\ 
    \max_{S\neq\emptyset,S\subsetneq V_B}
	\frac{\textup{gap}_{S}(\bof)}{\mathbf{1}_{S}^{\T}\balpha}\, .
\end{align*}
Recall the definition of $\gamma_{{\textup{\tiny AB}}}(\balpha)$ in Lemma \ref{lem:lower_bound_vj}, we have
%
	\begin{align}
	\gamma_{{\textup{\tiny AB}}}(\balpha)
	=\;
	\inf_{\bof\geq\bzero:\bar{v}(\bof)>0}
	\frac{\Lambda^*(\bof)}
	{\bar{v}(\bof)}
	&
	=\;
	\inf_{\bof\geq\bzero:\max_{S\subseteq V_B}
		\textup{gap}_{S}(\bof)>0}
	\frac{\Lambda^*(\bof)}
	{\max_{S\subseteq V_B}
		\frac{\textup{gap}_{S}(\bof)}{\mathbf{1}_{S}^{\T}\balpha}
}\nonumber\\
&
=\;
\inf_{\bof\geq\bzero:\max_{S\subseteq V_B}
	\textup{gap}_{S}(\bof)>0}
\left\{\min_{S \subseteq V_B:\textup{gap}_S(\bof)>0}
\left(\mathbf{1}_{S}^{\T}\balpha\right)
\frac{\Lambda^*(\bof)}
{\textup{gap}_{S}(\bof)}
\right\}\label{eq:gammaAB_technical_1}\\
&
\stackrel{(a)}{=}\;
\min_{S \subseteq V_B}
\left\{\inf_{\bof\geq\bzero:\textup{gap}_S(\bof)>0}
\left(\mathbf{1}_{S}^T\balpha\right)
\frac{\Lambda^*(\bof)}
{\textup{gap}_{S}(\bof)}
\right\}\, .\label{eq:gammaAB_technical_2}
\end{align}
For completeness, define the minimum over the empty set as $+\infty$.
Here $(a)$ holds because: For a minimizer $\bof^*\geq\bzero$ of the outer problem of (\ref{eq:gammaAB_technical_1}) and a minimizer $S^*\subseteq V_B$ of the inner problem of (\ref{eq:gammaAB_technical_1}),
$S^*\subseteq V_B$ is feasible for the inner problem of
(\ref{eq:gammaAB_technical_2}) while $\bof^*\geq\bzero$ is feasible for the outer problem of (\ref{eq:gammaAB_technical_2}), hence $(\ref{eq:gammaAB_technical_1})\geq (\ref{eq:gammaAB_technical_2})$.
Similarly we can show $(\ref{eq:gammaAB_technical_1})\leq (\ref{eq:gammaAB_technical_2})$.

We claim that
\begin{align}
	(\ref{eq:gammaAB_technical_2})
	=\;
	\min_{J \in\mathcal{J}}
	\left\{\inf_{\bof\geq\bzero:\textup{gap}_{\partial(J)}(\bof)>0}
	\left(\mathbf{1}_{\partial(J)}^{\T}\balpha\right)
	\frac{\Lambda^*(\bof)}
	{\textup{gap}_{\partial(J)}(\bof)}
	\right\}\, .\label{eq:gammaAB_technical_3}
\end{align}
Recall that the definition of $\mathcal{J}$:
\begin{equation*}
	\mathcal{J} = 
	\left\{
	J\subsetneq V_S:
	\sum_{j'\in J}\sum_{k\notin \partial(J)}\phi_{j'k}>0
	\right\}.
\end{equation*}
To see \eqref{eq:gammaAB_technical_3}, first note that for $S\subseteq V_B$ where $\{j'\in V_S:\partial(j')\subset S\}$ is empty, $\textup{gap}_{S}(\bof)$ is non-positive regardless of $\bof\geq\bzero$,
hence such $S$ can never be the minimizer.
For other $S$, let $J\triangleq\{j'\in V_S:\partial(j')\subset S \}$, then $\partial(J)\subset S$.
Note that
\begin{align*}
\textup{gap}_{\partial(J)}(\bof)
&=\;
\sum_{j'\in J, k\in V_B} f_{j'k} - \sum_{j' \in V_S,k\in \partial(J)} f_{j'k}\\
&=\;
\sum_{j':\partial(j')\subset S, k\in V_B} f_{j'k}
- \sum_{j' \in V_S,k\in S} f_{j'k}
+ \sum_{j' \in V_S,k\in S\backslash\partial(J)} f_{j'k}\\
&=\; \textup{gap}_{S}(\bof) + \sum_{j' \in V_S,k\in S\backslash\partial(J)} f_{j'k}\\
&\geq\;  \textup{gap}_{S}(\bof)\, .
\end{align*}
As a result, for $\bof$ such that $\textup{gap}_{S}(\bof)>0$, we have
\begin{align*}
	\left(\mathbf{1}_{S}^{\T}\balpha\right)
	\frac{\Lambda^*(\bof)}
	{\textup{gap}_{S}(\bof)}
	\geq\;
	\left(\mathbf{1}_{\partial(J)}^{\T}\balpha\right)
	\frac{\Lambda^*(\bof)}
	{\textup{gap}_{\partial(J)}(\bof)}\, .
\end{align*}
Hence only those $S\subseteq V_B$ where $S=\partial(J)$ for $J\subseteq V_S$ can be the minimizer.
If $J\notin\mathcal{J}$, then $\textup{gap}_{\partial(J)}(\bof)\leq 0$ regardless of $\bof\geq\bzero$, so these sets are also ruled out.
Therefore (\ref{eq:gammaAB_technical_3}) holds.

\yka{The last equality holds because:  For any $A\subseteq V_B$ let $J_A \triangleq \{i' \in V_S: \partial i' \subseteq A\}$. By definition, $\partial(J_A)  \subseteq A$. In finding $\min_{A\subseteq V_B}(\cdot)$ we can restrict attention to $A \subseteq V_B$ such that $\partial(J_A) = A$, since, given $J_A$ and for any $\bof$, adding more nodes to $A$ only reduces $\textup{gap}_A(\bof)$ (by increasing only the \emph{inflow} to $A$), and hence does not affect $\min_{A\subseteq V_B}(\cdot)$ since $\textup{gap}_A(\bof)$ appears in the denominator. This allows us to move from considering subsets $A$ of buffers, to subsets $J$ of servers with their associated buffer neighborhood $\partial(J)$. Further, we can safely ignore subsets $J \subseteq V_S$ such that $\sum_{i'\in J}\sum_{j\notin \partial(J)}\phi_{i'j}=0$, given that $\textup{gap}_{\partial(J)}(\bof) \leq 0$ for such $J$. We are left with $J \in \mathcal{J}$ where $\mathcal{J} = \{ J \subseteq V_S: \sum_{i'\in J}\sum_{j\notin \partial(J)}\phi_{i'j}>0\}$.}

Suppose the outer minimum of \eqref{eq:gammaAB_technical_3} is achieved by $J^*\in\mathcal{J}$.
	Denote the optimal value of the inner infimum of (\ref{eq:gammaAB_technical_3}) as $(\mathbf{1}^{\T}_{\partial(J^*)}\balpha)g({\bphi},J)>0$, then we have:
	\begin{equation}\label{eq:fraction_to_equation}
	\inf_{\bof\geq\bzero:\textup{gap}_{\partial(J)}(\bof)>0}
	\Lambda^*(\bof)
	- g({\bphi},J) \left(\sum_{j'\in J,k\in V_B}f_{j'k}
	-\sum_{j'\in V_S,k\in \partial(J)}f_{j'k}\right)
	=0\, .
	\end{equation}
	We can get rid of the constraint on $\bof$ because
	for $\bof$ where $\textup{gap}_{\partial(J)}(\bof)\leq 0$,
	the argument of minimization in (\ref{eq:fraction_to_equation}) is negative;
	and for $\bof$ that has negative components, its rate function is $\infty$ by definition.
	Using Legendre transform, we have:
	\begin{align*}
	&\inf_{\bof}
	\Lambda^*(\bof)
	- g({\bphi},J) \left(\sum_{j'\in J,k\in V_B}f_{j'k}
	-\sum_{j'\in V_S,k\in \partial(J)}f_{j'k}\right)\\
	=\; &\inf_{\bof}
	\Lambda^*(\bof)
	-
	\bof^{\T}
	\left(g({\bphi},J)\sum_{j'\in J,k\in V_B}\e_{j'k}
	-g({\bphi},J)\sum_{j'\in V_S,k\in \partial(J)}\e_{j'k}\right)\\
	=\; &-\Lambda\left(g({\bphi},J)\sum_{j'\in J,k\in V_B}\e_{j'k}
	-g({\bphi},J)\sum_{j'\in V_S,k\in \partial(J)}\e_{j'k}\right)\\
	\stackrel{(b)}{=}\; &- \sum_{j'\in V_S,k\in V_B}{\phi}_{j'k} \left( e^{g({\bphi},J)\ind \{j'\in J\}-g({\bphi},J)\ind \{k\in \partial(J)\}} - 1 \right) \, .\\
	\end{align*}
	In (b) we use the fact that the dual function of $\Lambda^*(\bof)$ is $\Lambda(\bx)= \sum_{j' \in V_S,k\in V_B}{\phi}_{j'k} (e^{x_{j'k}} - 1) $ where $\bx\in\mathbb{R}^{n\times m}$.
	Hence
	Eq.~(\ref{eq:fraction_to_equation}) reduces to the nonlinear equation
	\begin{align*}
	\left(
	\sum_{j'\notin J,k\in \partial(J)}{\phi}_{j'k}
	\right)e^{-g({\bphi},J)}
	+
	\left(
	\sum_{j'\in J,k\notin \partial(J)}{\phi}_{j'k}
	\right)
	e^{g({\bphi},J)}
	=
	\sum_{j'\notin J,k\in \partial(J)}{\phi}_{j'k}
	+ \sum_{j'\in J,k\notin \partial(J)}{\phi}_{j'k}\, .
	\end{align*}
	Let $y\triangleq e^{g({\bphi},J)}$, this becomes a quadratic equation:
	\begin{align*}
	\left(
	\sum_{j'\in J,k\notin \partial(J)}{\phi}_{j'k}
	\right)
	y^2
	-
	\left(\sum_{j'\notin J,k\in \partial(J)}{\phi}_{j'k}
	+ \sum_{j'\in J,k\notin \partial(J)}{\phi}_{j'k}\right)y
	+
	\left(
	\sum_{j'\notin J,k\in \partial(J)}{\phi}_{j'k}
	\right)
	=0\, .
	\end{align*}
	Hence
	\begin{equation*}
	y = \frac{\sum_{j'\notin J,k\in \partial(J)}{\phi}_{j'k}}
	{\sum_{j'\in J,k\notin \partial(J)}{\phi}_{j'k}}\text{ or $1$}\, .
	\end{equation*}
	Since $g({\bphi},J)>0$, we have
	\begin{equation*}
	g({\bphi},J) =\ \log\left(
	\frac{\sum_{j'\notin J,k\in \partial(J)}{\phi}_{j'k}}
	{\sum_{j'\in J,k\notin \partial(J)}{\phi}_{j'k}}
	\right)
	=\ 
	\log\left(
	\frac{\sum_{j'\notin J,k\in \partial(J)}\phi_{j'k}}
	{\sum_{j'\in J,k\notin \partial(J)}\phi_{j'k}}
	\right)
	\, .
	\end{equation*}
\yka{What is the $\bof$ which allows to achieve this $g(\phi,J)$? Is it the one that sets $f_{i' \in J, j \in (\partial J)^c} = R \phi_{j'k}, f_{i' \in J^c, j \in (\partial J)} = \phi_{j'k}/R, f_{i' \in J, j \in (\partial J)} = \phi_{j'k}, f_{i' \in J^c, j \in (\partial J)^c} =  \phi_{j'k} $ for $R=\phi_{ J, (\partial J)^c}/\phi_{ J^c, \partial J}$?}
	Plugging into \eqref{eq:gammaAB_technical_3}, we have:
	\begin{equation*}
	\gamma_{{\textup{\tiny AB}}}(\balpha)
	=
	\min_{J\in\mathcal{J}}
	\left(\mathbf{1}^{\T}_{\partial(J)} \balpha\right)
	\log\left(
	\frac{\sum_{j'\notin J,k\in \partial(J)}\phi_{j'k}}
	{\sum_{j'\in J,k\notin \partial(J)}\phi_{j'k}}
	\right)\, .
	\end{equation*}
\yka{Intuitively, there is no slack in \eqref{eq:barv-def} because the adversary will implement a radial solution, so the subset $A$ attacked will be the $A_2(\bX[t], d\bX[t](\bA[t], \alpha)/dt)$. Must write this somewhere. Maybe in the converse section.}

\emph{Remark:} For $J\in\mathcal{J}$, if there exists $j'\in V_S$ such that $j'\notin J$ but $\partial(j') \subseteq \partial(J)$, then such subsets $J$ are ``spurious'' in the sense that they cannot achieve the minimum in the expression of $\gamma_{{\textup{\tiny AB}}}(\balpha)$ (the term corresponding to $J\cup\{j'\}$ is no larger than the term correpsonding to $J$). Therefore only the ``maximal'' $J$'s matter to the value of exponent.
\vspace{1pt}\\

\noindent\textbf{Most likely service token sample path leading to throughput loss.}
Denote
\begin{align*}
	\mathbf{c}\triangleq g({\bphi},J)\left(\sum_{j'\in J,k\in V_B}\e_{j'k}
	-\sum_{j'\in V_S,k\in \partial(J)}\e_{j'k}\right)\, ,
\end{align*}
denote $\bof_J$ as the minimizer of the inner minimization problem on the RHS of (\ref{eq:gammaAB_technical_3}). We have
\begin{align*}
	\bof_J =\ &\textup{argmin}_{\bof\geq \bzero}
	\sum_{j'\in V_S} \sum_{k\in V_B} \left(\Lambda_{j'k}^*( f_{j'k}) - c_{j'k} f_{j'k}\right) \\
	=\ &\textup{argmin}_{\bof\geq \bzero}
	\sum_{j'\in V_S} \sum_{k\in V_B} \left( f_{j'k} \log\frac{ f_{j'k}}{{\phi}_{j'k}} +  {\phi}_{j'k} -  f_{j'k} - c_{j'k} f_{j'k}\right)\, .
\end{align*}
First order condition implies: $(\bof_J)_{j'k}
	=
	{\phi}_{j'k} \frac{e^{c_{j'k}+1}}{\sum_{j',k}{\phi}_{j'k}e^{c_{j'k}+1}}
	=
	{\phi}_{j'k} \frac{e^{c_{j'k}}}{\sum_{j',k}{\phi}_{j'k}e^{c_{j'k}}}
	$.
Recall the definition of $\lambda_J$, $\mu_J$ in (\ref{eq:gamma_w}), we have
\begin{align*}
\sum_{j',k}{\phi}_{j'k}e^{c_{j'k}}
&=\;
\sum_{j'\in J,k\notin \partial(J)}{\phi}_{j'k}\frac{\lambda_J}{\mu_J}
+
\sum_{j'\notin J,k\in \partial(J)}{\phi}_{j'k}\frac{\mu_J}{\lambda_J}
+
\left(
1-
\sum_{j'\in J,k\notin \partial(J)}{\phi}_{j'k}
-
\sum_{j'\notin J,k\in \partial(J)}{\phi}_{j'k}
\right)\\
&=\;
\mu_J\frac{\lambda_J}{\mu_J}
+
\lambda_J\frac{\mu_J}{\lambda_J}
+
\left(
1-
\lambda_J
-
\mu_J
\right)
=\; 1\, .
\end{align*}
Hence
\begin{align*}
	(\bof_J)_{j'k}=
	\left\{
	\begin{array}{ll}
	{\phi}_{j'k}(\lambda_{J}/\mu_{J}),&\text{ for }j'\in J,k\notin\partial(J)\\
	{\phi}_{j'k}(\mu_{J}/\lambda_{J}),&\text{ for }j'\notin J,k\in\partial(J)\\
	{\phi}_{j'k},&\text{ otherwise}
	\end{array}
	\right..
\end{align*}
Let $J^* = \textup{argmin}_{J\in\mathcal{J}}B_J \log(\lambda_J/\mu_J)$, then service token sample path with constant derivative $\bof_{J^*}$ is the most likely sample path leading to throughput loss.
\end{proof}

\section{Necessity of our assumptions and the inferiority of state-independent control}\label{appen:sec:crp}

This section shows the necessity of our assumptions, and of state-dependent control, including the proofs of Propositions \ref{prop:NT-is-necessary}, \ref{prop:hall_is_necessary} and \ref{prop:state_ind_no_exp}. 

\subsection{Necessity of Assumption \ref{asm:non_trivial}: proof of Proposition \ref{prop:NT-is-necessary}}
\begin{proof}[\textbf{Proof of Proposition~\ref{prop:NT-is-necessary}}.]
We define the following policy $U$ which ensures no throughput loss in the long run, i.e.,  $\mathbb{P}^{\N,U} = 0$. Arbitrarily choose $n$ of the $K$ jobs and dedicate one of the chosen jobs to each of the servers. Suppose the job dedicated to server $j'$ is initially at buffer $i$. Since Assumption~\ref{asm:connectivity} is satisfied, there is a way to move the job from $i$ to a buffer compatible with $j'$ in a finite (random) time. Move the job to some node in $\partial(j')$. Similarly, move each of the $n$ dedicated jobs into the neighborhood of the corresponding server. All this is completed in an initial transient of finite (random) duration (the expected duration is also finite).  Thereafter, for each service token arrival, use it to serve the job dedicated to the origin of service token. We are guaranteed that the destination $k \in \partial(j')$, i.e., the job remains within the neighborhood of $j'$ after completing service (we are told that service token types with $k \notin \partial(j')$ have zero arrival rate $\phi_{j'k}=0$).
\end{proof}

\subsection{Necessity of CRP condition: proof of Proposition~\ref{prop:hall_is_necessary}}
\begin{proof}[Proof of Proposition~\ref{prop:hall_is_necessary}]
	There are two cases:
\medskip

\noindent\textbf{Case 1:} \textit{There exists $J \subsetneq V_S$ s.t. $ \lambda_J < \mu_J \Longleftrightarrow \sum_{i\in \partial(J)}\mathbf{1}^{\T}\phi_{(i)}<\sum_{j'\in J}\mathbf{1}^{\T} \phi_{j'}$.}

\smallskip The main proof idea in this case is simply that since the net job arrival rate to $\partial(J)$ is less than the net service rate in $J$, a positive fraction of service tokens must be wasted. \yka{I replaced $KT$ by $T$ and made a bunch of other changes.}
		
	    Consider the following balance equation:
		\begin{align*}
		& \;\#\{\text{service tokens originating in $J$ during $[0,T]$ which are wasted}\}\\
		=& \;
		\#\{\text{service tokens originating in $J$ during $[0,T]$}\}
		\\ &\quad -
		\#\{\text{service tokens originating in $J$ during $[0,T]$ which are utilized}\}\\
		\geq& \;
		\#\{\text{service tokens originating in $J$ during $[0,T]$}\}
		-
		\#\{\text{jobs served from $\partial(J)$ during $[0,T]$}\}\\
		\geq& \;
		\sum_{r:t_r\in[0,T]}\ind \{o[r]\in J\}
		-
		\sum_{r:t_r\in[0,T]}\ind \{d[r]\in \partial(J)\}
		-
		\#\{\text{initial jobs in $\partial(J)$}\}\, .
		\end{align*}
		The first inequality holds because the service tokens originating in $J$ can only serve jobs from $\partial(J)$. The second inequality holds because the total number of jobs served from $\partial(J)$ during $[0,T]$ cannot exceed the initial jobs there plus the number of service token arrivals with destination in  $\partial(J)$.
		Divide both sides by $A_{\Sigma}(T)$ which is the total number of service token arrivals during $[0,T]$, and let $T\to\infty$.
		By the strong law of large numbers, we have:
		\begin{align*}
		\liminf_{T\to\infty}
		\{\text{fraction of service token wasted in $[0,T]$}\}
		\geq
		\sum_{j'\in J}\mathbf{1}^{\T} \phi_{j'}
		-\sum_{i\in \partial(J)}\mathbf{1}^{\T}\phi_{(i)}
		>0\, .
		\end{align*}
		Hence a positive fraction of service token will be wasted in the long run, and the loss exponent is $0$.
		
		\bigskip
		
\noindent\textbf{Case 2:} \textit{We have $\lambda_{J'} \geq \mu_{J'}$ for all $J' \in \cJ$ but there exists $J\in\mathcal{J}$ such that $\lambda_J = \mu_J \Leftrightarrow \sum_{i\in \partial(J)}\mathbf{1}^{\T}\phi_{(i)}=\sum_{j'\in J}\mathbf{1}^{\T} \phi_{j'}$.}
\smallskip

The high-level idea in this case is that if all the service tokens originating in $J$ are utilized (if possible), then, at best, the total quantity of jobs in $\partial(J)$ follows an unbiased random walk on $0,1, \dots, K$. Such a random walk spends a positive fraction of time at $0$, and all service tokens originating in $J$ when there is zero job in $\partial(J)$ is lost. The proof is somewhat more intricate than this argument may suggest; in particular because we need to allow for idling policies (those which sometimes waste service tokens even though a job is available at a neighboring buffer).

		Divide the\yka{deleted ``discrete time periods''} service token arrivals into cycles with $M\N^2$ arrivals each, where
		$$M\triangleq  \frac{1}{\mu_J}\, ,$$
		for $ \mu_J = \sum_{j'\in J,k\notin \partial(J)}\phi_{j'k} > 0$ as before.
		Without loss of generality, consider the first cycle $t_1,\cdots,t_{MK^2}$.
		Define random walk $S_r$ with the following dynamics:
		\begin{compactitem}[leftmargin=*]
			\item $S_0  = \mathbf{1}_{\partial(J)}^{\T}\bX(0)$.
			\item $S_{r+1}  = S_r  + 1$ if $o[r]\notin J,d[r]\in \partial(J)$.
			\item $S_{r+1}  = S_r  - 1$ if $o[r]\in J,d[r]\notin \partial(J)$.
			\item $S_{r+1}  = S_r$ otherwise.
		\end{compactitem}
It is not hard to see that if no service token is wasted during $r \leq MK^2$ under some policy $U$, then $S_r $ is a pathwise upper bound on the number of jobs in $\partial(J)$, namely, $\mathbf{1}_{\partial(J)}^{\T} \X(t_r)$, for any $r \leq MK^2$.
	With this observation, we have:
		\begin{align}
		\mathbb{P}\left(\text{some token is wasted during } r \leq M\N^2\right)\ 
		\geq\ \mathbb{P}\left(S_{r'}=0\text{ for some }r' < M\N^2\right)\cdot (\mathbf{1}^T \phi_{j'})\, .
		\label{eq:demand-loss-lb}
		\end{align}
The above is true because when the event on RHS happens, either (1) some service token is wasted before $t_r'$, or (2) no service token is wasted before $t_{r'}$, then since $0=S_{r'}\geq\mathbf{1}_{\partial(J)}^{\T} \X(t_{r'})$ we have $\mathbf{1}_{\partial(J)}^{\T} \X(t_{r'})=0$ and so any service token with origin in $J$ is wasted at $t_{r'+1}$. Importantly, \eqref{eq:demand-loss-lb} holds for \emph{any} policy.

		For the given $J$ we have $\lambda_J = \mu_J > 0$ and so $S_r$ is a ``lazy'' simple random walk, which takes a step with probability $2 \mu_J$ independently at each $r$.
		Define the stopping time $\tau$ as 
		\begin{align*}
		    \tau  \triangleq \inf \left \{ r \in \mathbb{Z}_+: S_r \in \big\{\mathbf{1}_{\partial(J)}^{\T}\X(0)-\N,\mathbf{1}_{\partial(J)}^{\T}\X(0)+\N\big\} \right \} \, .
		\end{align*} 
		Using
		\cite[Example 4.1.6,][]{durrett2010probability} on the lazy simple random walk $S_r - \mathbf{1}_{\partial(J)}^{\T}\X(0)$, we obtain\footnote{Since $S_r- \mathbf{1}_{\partial(J)}^{\T}\X(0)$ is a lazy version of a simple random walk, which takes a step with probability $2 \mu_J$ independently at each time, the expectation of the time $\tau$ to hit $\pm K$ is inflated by a factor of $1/(2\mu_J)$ relative to that of a simple random walk (this follows from using the natural coupling between the steps in the two walks, and noting that the lazy walk takes expected time $1/(2\mu_J)$ between consecutive steps).}
		\begin{equation*}
		\mathbb{E}[\tau] = \frac{\N^2}{2\mu_J}\, .
		\end{equation*}
		Using Markov's inequality, we have
		\begin{align*}
		    \mathbb{P} \left(\, \tau \geq MK^2 \,\right) \leq \frac{	\mathbb{E}[\tau]}{MK^2} = \frac{1}{2}
		\end{align*}
		By symmetry
		\begin{align}
		    \mathbb{P} \left(\, S_\tau - \mathbf{1}_{\partial(J)}^{\T}\X(0)= -K \ \textup{and} \ \tau < mK^2 \right ) = \frac{1}{2} \mathbb{P} \left( \tau < mK^2 \right) \geq \frac{1}{2} \cdot \frac{1}{2} = \frac{1}{4}  \, .
		    \label{eq:RW-lb}
		\end{align}
		Now, $S_\tau - \mathbf{1}_{\partial(J)}^{\T}\X(0)= -K$ and $\tau < MK^2$, i.e., $S_r$ hits $\mathbf{1}_{\partial(J)}^{\T}\X(0)-K$ during $r< MK^2$, implies  that $S_{r}$ hits $0$ during $ t < MK^2$, since $S_r$ must hit $0$ (weakly) before it hits $\mathbf{1}_{\partial(J)}^{\T}\X(0)-K$. Hence, plugging \eqref{eq:RW-lb} into \eqref{eq:demand-loss-lb} we obtain that
		\begin{align*}
		   \mathbb{P}\left(\text{some service token is wasted during } r \leq M\N^2\right)
		\geq\ & \frac{\mathbf{1}^{\T} \phi_{j'}}{4}\, , 
		\end{align*}
		and this uniform and strictly positive lower bound holds for \emph{any} policy, during \emph{any} cycle consisting of $MK^2$ consecutive arrivals.
It follows that
	$
		{\bbp}^{\N,U} = \Omega\left(\frac{1}{\N^2}\right)$, and hence $\gamma(U)=0$ for any $U$.
\end{proof}

\subsection{Necessity of state-dependent control: proof of Proposition \ref{prop:state_ind_no_exp}}
\begin{proof}[Proof of Proposition \ref{prop:state_ind_no_exp}]
\ \\
\emph{Proof of first part.} For notation simplicity, denote $X(t_r)$ by $X[r]$, similar for another notations.
Denote the probability mass function of distribution $u_{j'k}[t]$ by $u_{j'k}[t](\cdot)$.
	We first define an ``augmented'' policy $\tilde{U}$ for any state-independent policy $U$.
	Policy $\tilde{U}$ is also state independent with distribution $\tilde{u}_{j'k}[t]$, where:
	\begin{align*}
	\tilde{u}_{j'k}[t](i)
	& =
	u_{j'k}[t](i)
	+
	\frac{1}{|\partial(j')|}u_{j'k}[t](\emptyset)\quad \textup{for }i\in\partial(j')\, ,\\
\tilde{u}_{j'k}[t](\emptyset)& =0\, .
	\end{align*}
	In the following analysis, we couple $U$ and $\tilde{U}$ in such a way that if $U$ serves job in $i$ with the $t$-th service token, then $\tilde{U}$ will do the same.
	
	Divide the service token arrivals into cycles with $K^2$ arrivals each.
	We will lower bound the probability of throughput loss in any cycle.
	Without loss of generality, consider the first cycle $[1,K^2]$.
	Suppose $\bX^{\N,U}[0]=\X_0$.
	By Assumption \ref{asm:non_trivial}, $\exists j' \in V_S,\ k\notin \partial(j')\subset V_B$ such that $\phi_{j'k}>0$.  \yka{Why can't we just take $J = \{i\}$? Also, I would rather name the o-d pair as $\phi_{j' k}>0$ within this proof, instead of $i',j$ which conflicts with the notation in the para above.}
	Consider the random walk $S_t$ with the following dynamics, which is the ``virtual'' net change of number of jobs in $\partial(j')$:
	\begin{compactitem}[leftmargin=*]
		\item $S_0  = 0$.
		\item $S_{t+1} = S_t  + 1$ if $d[t]\in \partial(j')$ and policy $\tilde{U}$ serves a job from outside of $\partial(j')$ using it (regardless of whether there is available job to serve).
		\item $S_{t+1} = S_t  - 1$ if $d[t]\notin \partial(j')$ and policy $\tilde{U}$ serves a job from $\partial(j')$ using it (regardless of whether there is available job to serve).
		\item $S_{t+1} = S_t $ if otherwise.
	\end{compactitem}
	
	Using similar argument as in eq. \eqref{eq:demand-loss-lb}, we have
	\begin{align}
	&\mathbb{P}\left(\text{some service token is wasted in epoch }[1,\N^2]\right) \nonumber\\
\geq\ &\mathbb{P}\left(S_{K^2}+\mathbf{1}_{\partial(j')}^{\T} \X_0> \N\text{ or }< 0\right)
    \geq \mathbb{P}\left(|S_{K^2}|>K\right) \label{eq:indep_mg}\, .
	\end{align}
	
	Note that $S_{K^2} $ is the sum of $\N^2$ independent random variables $Z_t$, where $Z_t=S_{t}-S_{t-1} $. Here independence holds because we ignore throughput losses in the definition of the process.
	Here $Z_t$ has support $\{-1,0,1\}$ and satisfies:
	\begin{equation}
	\mathbb{P}(Z_t=-1)\geq \delta\triangleq \phi_{j'k} >0\, ,
\label{eq:state-indep-delta-def}
	\end{equation}
where $k\notin \partial(j)$.
	There are two cases:\yka{Maybe number the cases 1 and 2, since you may need subcases 2(i) and 2(ii) and we want to avoid confusion.}
	\begin{compactenum}[label=\arabic*.,leftmargin=*]
		\item If $\mathbb{E}[S_{K^2} ]\leq -\frac{\N^2}{2}$, then for $\N\geq 8$, we have
		\begin{align*}
		1-\mathbb{P}\left(S_{K^2}
		\in [-\N,\N]\right)
		\geq& \;
		1-\mathbb{P}\left(S_{K^2}
		-
		\mathbb{E}[S_{K^2} ]
		\geq
		-\N + \frac{\N^2}{2}
		\right)\\
		\geq& \;
		1 - 2\exp \left(-\frac{\N^2}{32}\right)
		\qquad
		(\text{Hoeffding's inequality, $-K + K^2/2 \geq K^2/4$})\\
		\geq& \; \frac{1}{2}\, .
		\end{align*}
Plugging into \eqref{eq:indep_mg} establishes that service token is wasted with likelihood at least $1/2$.
		\item If $\mathbb{E}[S_{K^2} ]> -\frac{\N^2}{2}$, then using linearity of expectation and simple algebra we obtain that the number of $t$'s such that $\mathbb{E}[Z_t]\geq -\frac{3}{4}$ is at least $\frac{\N^2}{7}$.
		
		Denote the set of these $t$'s as $\mathcal{T}$.
		Hence
		\begin{equation}\label{eq:var_nonind}
		\begin{split}
		K^2\geq\text{Var}(S_{K^2} ) = \sum_{t=1}^{\N^2}\text{Var}(Z_t)
		\geq
		\sum_{t\in\mathcal{T}}\text{Var}(Z_t)
		\geq
		\frac{\N^2}{7}\cdot \delta \left(1-\frac{3}{4}\right)^2
		=
		\frac{\delta}{102}K^2\, ,
		\end{split}
		\end{equation}
using \eqref{eq:state-indep-delta-def}.
				
		Note from \eqref{eq:indep_mg} that to show a constant lower bound of throughput-loss probability on $[1,K^2]$,
		it suffices to derive a uniform upper bound on $\mathbb{P}\left(S_{K^2}
		\in [-\N,\N]\right)$ that is strictly smaller than $1$.
		To this end, apply Theorem 7.4.1 in \cite{chung2001course} (Berry-Esseen Theorem) to obtain:
		\begin{align}\label{eq:berry_esseen}
		 \;\sup_{x\in\mathbb{R}}\left|
		\mathbb{P}\left(S_{K^2} -\mathbb{E}[S_{K^2} ]\leq x\sqrt{\text{Var}[S_{K^2} ]}\right)
		-\Phi(x)
		\right|
		\leq  \;
		\frac{\sum_{t=1}^{\N^2}\mathbb{E}|Z_t-\mathbb{E}Z_t|^3}{\left(\text{Var}[S_{K^2} ]\right)^{3/2}}
		\leq  \; \frac{5000}{K \delta^{3/2}}\, ,
		\end{align}
		where $\Phi(\cdot)$ is the cumulative distribution function of the standard normal distribution.
		
		Denote $B(x,a)\triangleq [x-a,x+a]$.
		Note that there are two subcases (indexed 2(i) and 2(ii)):
		\begin{align*}
		[-\N,\N]\subset B\left(\mathbb{E}[S_{K^2} ],4\N\right),\qquad
			[-\N,\N]\cap B\left(\mathbb{E}[S_{K^2} ],2\N\right)=\emptyset\, .
		\end{align*}
\yka{This is confusing because the second event implies the third event. The preceding sentence instead suggests you are going to define mutually exclusive events. I think you wanted to delete the second event. Also, I think it would be better to call these ``subcases'' rather than ``events'' since they can be checked a priori instead of waiting for some random variables to be  exposed. With random events, one should work carefully with conditional probabilities etc.}
In subcase 2(i),
$
  \mathbb{P}\left(S_{K^2}
			\in [-\N,\N]\right)
			\leq
			\mathbb{P}\left(
			S_{K^2} \in B\left(\mathbb{E}[S_{K^2} ],4\N\right)
			\right)\, ,
$
whereas in subcase 2(ii),
\begin{align*}
  \mathbb{P}\left(S_{K^2}
			\in [-\N,\N]\right)
			\leq
			1-
			\mathbb{P}\left(
			S_{K^2} \in B\left(\mathbb{E}[S_{K^2} ],2\N\right)\right)\, .
\end{align*}
Hence
		\begin{align}
			\mathbb{P}\left(S_{K^2}
			\in [-\N,\N]\right)
			\leq
			\max\left\{
			\mathbb{P}\left(
			S_{K^2} \in B\left(\mathbb{E}[S_{K^2} ],4\N\right)
			\right)
			,
			1-
			\mathbb{P}\left(
			S_{K^2} \in B\left(\mathbb{E}[S_{K^2} ],2\N\right)\right)
			\right\}\, .
\label{eq:ub-S-mK-K}
		\end{align}
		Use \eqref{eq:berry_esseen} and $\text{Var}(S_{K^2} ) \leq K^2$ to obtain
		\begin{align*}
			\mathbb{P}\left(
			S_{K^2} \in B\left(\mathbb{E}[S_{K^2} ],4\N\right)
			\right)
		\leq 	& \;
		\mathbb{P}\left(
		S_{K^2}-\mathbb{E}[S_{K^2} ]
		\leq \sqrt{\text{Var}[S_{K^2} ]}\frac{4K}{\sqrt{\text{Var}[S_{K^2} ]}}
		\right)\\
		\leq & \;
		5000 \delta^{-3/2}K^{-1}
		+
		\Phi\left(\frac{4K}{\sqrt{\text{Var}[S_{K^2} ]}} \right)\\
		\leq & \;
		5000 \delta^{-3/2}K^{-1}
		+
		\Phi\left(50 \delta^{-1/2} \right)\, ,\\
		\ \\
		1-\mathbb{P}\left(
		S_{K^2} \in B\left(\mathbb{E}[S_{K^2} ],2\N\right)
		\right)
		= 	& \;
		\mathbb{P}\left(
	S_{K^2}-\mathbb{E}[S_{K^2} ]
	\leq \sqrt{\text{Var}[S_{K^2} ]}\frac{-2K}{\sqrt{\text{Var}[S_{K^2} ]}}
		\right)\\
& \;+\mathbb{P}\left(
		S_{K^2}-\mathbb{E}[S_{K^2} ]
		\geq \sqrt{\text{Var}[S_{K^2} ]}\frac{2K}{\sqrt{\text{Var}[S_{K^2} ]}}
		\right)\\
		\leq & \;
		10000 \delta^{-3/2}K^{-1}
		+
		2\Phi\left(\frac{-2K}{\sqrt{\text{Var}[S_{K^2} ]}} \right)\\
		\leq & \;
		10000 \delta^{-3/2}K^{-1}
		+
		2\Phi\left(-2 \right)\, .
		\end{align*}
	Hence for $K>\max\left\{\frac{10000\delta^{-3/2}}{\bar{\Phi}\left(50\delta^{-1/2}\right)},
	\frac{10000\delta^{-3/2}}{\frac{1}{2} - \Phi(-2)}
	 \right\}$, plugging into \eqref{eq:ub-S-mK-K} and then into \eqref{eq:indep_mg}, we obtain
	 \begin{align*}
	 	\mathbb{P}(\textup{some service token is wasted in $[1,K^2]$})
	 	\geq
	 	\min\left\{\frac{1}{2}\bar{\Phi}\left(50\delta^{-1/2}\right),
	 	\frac{1}{2}-\Phi(-2)
	 	\right\}
	 	>0\, .
	 \end{align*}
	\end{compactenum}
Since we obtained a uniform lower bound on the likelihood of wasting service token in both cases, we conclude that the steady state throughput-loss probability is $\Omega(1/K^2)$ as $K\to\infty$.

\emph{Proof of second part.} Consider any $k \in V_B$ such that $\exists j' \in V_S$ such that $(j',k) \in S$. Given a service token type distribution $\bphi \in D(S)$, suppose $U$ achieves asymptotic optimality $\mathbb{P}^{\N,U} = o(1)$, i.e., $1- o(1)$ fraction of service tokens are utilized. This implies that a fraction $\sum_{j' \in V_S: (j',k) \in S} \phi_{j'k} - o(1)$ of service tokens has destination $k$ \emph{and} is utilized under $U$. And that a fraction $\sum_{(j',i) \in S} \phi_{j'i} u_{j'k} - o(1)$ of service tokens are used to serve a job from $k$ \emph{and} are utilized under $U$. (Our proof will focus on the case where $u_{j'k}$ is time invariant and independent of $K$. 
The proof for the general case of time varying $u_{j'k}(t)$ which can depend on $K$ is very similar, though the latter fraction can now vary over time, increasing the notational burden. We omit the details.) But in steady state, the inflow of jobs to node $k$ must be equal to the outflow of jobs, i.e., it must be that  
\begin{align*}
\sum_{j' \in V_S: (j',k) \in S} \phi_{j'k} = \sum_{(j',i) \in S} \phi_{j'i} u_{j'k} \, .
\end{align*}
This is a knife edge requirement. In particular, the set of $\bphi \in D(S)$ which \emph{do not} satisfy this condition is clearly an open and dense subset of $D(S)$. For all such $\bphi$, the above argument implies that $\liminf_{K \to \infty} \mathbb{P}^{\N,U} > 0$, completing the proof. 
\end{proof}

\newpage
 \setcounter{page}{1}
{\LARGE
\begin{center}
Online appendix for ``Large Deviation Optimal Scheduling of Closed Queueing Networks''    
\end{center}
}

\section{Supplementary material to Appendix \ref{appen:subsec:achievability}: proof of Lemma \ref{lem:lower_bound_vj}}\label{online_appen_main_proof}

\begin{proof}[\textbf{Proof of Lemma \ref{lem:lower_bound_vj}.}]
	\noindent \textit{Step 1. Define stopping times and consider the sampling chain.}
	In this step, we mostly follow the approach in \cite{venkataramanan2013queue} (Freidlin-Wentzell theory) and decompose the expression for the likelihood of the Lyapunov function taking on a large value.
	There are minor differences between our proof and proof of Theorem 4 in \cite{venkataramanan2013queue} because of our closed queueing network setting, so we will write down each step for completeness.
	
	Let $\fX^{K,U}_{\bz}(\infty)$ be a random vector distributed as the stationary distribution of recurrent class associated with initial (normalized) state $\bz\in\Omega$. For notation simplicity, we suppress the dependence on $\bz$ and $U$ and keep them fixed. 
	We want to upper bound:
	\begin{equation*}
	\limsup_{\N\to\infty}
	\frac{1}{\N}\log
	\mathbb{P}
	\left(
	L_{\balpha}(\fX^{K}(\infty)) \geq 1
	\right)\, .
	\end{equation*}
	
	Choose positive constants $\delta,\epsilon$ such that  $0<\delta<\epsilon<1$. \yka{$\rho$ doesn't show up until step 2b. Why not define it there?}
	Consider the following stopping times defined on a sample path $\fX^{K}(\cdot)$:
	\begin{align*}
	\beta_1^{K} &\triangleq\
	\inf\{t\geq 0: L_{\balpha}(\fX^{K}(t))\leq \delta\},\\
	\eta_i^{K} &\triangleq\
	\inf\{t\geq \beta_i^{K}:L_{\balpha}(\fX^{K}(t))\geq \epsilon\},\quad i=1,2,\cdots\\
	\beta_i^{K} &\triangleq\
	\inf\{t\geq \eta_{i-1}^{\N,U}:L_{\balpha}(\fX^{K}(t))\leq \delta\},\quad
	i=2,3,\cdots
	\end{align*}
	
	Let the discrete-time Markov chain $\hat{\bX}^{K}[i]$ be obtained by sampling $\fX^{K}(t)$ at the stopping times $\eta_i^{K}$.
	Since $\fX^{K}(\cdot)$ is stationary, there must also exist a stationary distribution for Markov chain $\hat{\bX}^{K}[\cdot]$.
	Let $\Theta^{K}$ denote the state space of the sampled chain $\hat{\bX}^{K}[\cdot]$, \yka{which chain? $\hat{\bX}_{\bz}^{\N,U}[i]$ or $\fX^{K}[t]$?}$\hat{\pi}^{K}$ is the sampled chain's stationary distribution.

The above construction was based on the following idea: first divide time into cycles, where the $i$-th cycle is the interval of time between consecutive $\eta_i$'s, i.e., a cycle is completed each time the value of $L_\alpha(\fX^{K} )$ goes down below $\delta$ and then rises above $\epsilon$.
Then the fraction of time the Lyapunov function spent above $1$ is equal to the ratio
\begin{align*}
	\Ex [\textup{time for which }L_{\balpha}(\fX^{K} ) \geq 1\textup{ during a cycle}]/(\Ex [\textup{length of cycle}])
\end{align*}
in steady state. We sample the initial state as $\fKX(0)=\bx \sim \hat{\pi}^{K}$, hence the first cycle itself characterizes the steady state ratio.
Therefore, the stationary likelihood of event $\{L_{\balpha}(\fX^{K}) \geq 1\}$ can be expressed as (see Lemma 10.1 in \citealt{stolyar2003control}):
	\begin{equation}\label{eq:achievability_sample_chain}
	\mathbb{P}
	\left(
	L_{\balpha}(\fX^{K}) \geq 1
	\right)
	=
	\frac{\int_{\Theta^{K}}\hat{\pi}^{K}(d\bx)
		\cdot
		\mathbb{E}
		\left(
		\int_{0}^{\eta_{1}^{K}}\ind\left\{L_{\balpha}(\fX^{K}(t))\geq 1\right\}dt
		\left|\fX^{K}(0)=\bx\right.
		\right)}
	{\int_{\Theta^{K}}\hat{\pi}^{K}(d\bx)\cdot
		\mathbb{E}(\eta_1^{K}|\fX^{K}(0)=\bx)}\, .
	\end{equation}
	\ \\
	
	\noindent\textit{Step 2. Bounding the RHS of (\ref{eq:achievability_sample_chain})}. To upper bound $\mathbb{P}
	\left(
	L_{\balpha}(\fX^{K}) \geq 1
	\right)
$, we lower bound the denominator in the RHS of (\ref{eq:achievability_sample_chain}) and upper bound the numerator.
	\begin{compactitem}[leftmargin=*]
		\item \textit{Step 2a. Bounding the Denominator.}
To lower bound the denominator, we focus on the discrete-time embedded chain of $\{\fKX(\cdot)\}$. Note each exactly one service token arrives at each jump of the chain, therefore $||\fKX(\cdot)||_{\infty}$ change by at most $\frac{1}{K}$ at each jump.
Using property 2 of $L_{\balpha}(\cdot)$ in Lemma \ref{lem:tech_lems}, we further have that $L_{\balpha}(\fKX(\cdot))$ change by at most $\frac{1}{K\cdot\min_{i}\alpha_i}$ at each jump.
Since the Lyapunov function $L_{\balpha}(\fKX(\cdot))$ has to increase from $\delta$ to $\epsilon$ during $[0,\eta_{1}^K]$, there exists $K_1 = K_1(\epsilon,\delta)>0$ such that for any $K>K_1$, at least $\frac{K\cdot\min_{i}\alpha_i}{2}(\epsilon-\delta)$ jumps occur during $[0,\eta_{1}^K]$.
Because the times between two consecutive jumps follow i.i.d. exponential distribution with rate $K\mathbf{1}^{\T}\hat{\bphi}\mathbf{1}$, therefore for any $K>K_1$,
\begin{align}\label{eq:sample_chain_denom_lb}
	\mathbb{E}(\eta_1^{K}|\fKX(0)=\bx)
	\geq
	\frac{K\cdot\min_{i}\alpha_i}{2}(\epsilon-\delta)\frac{1}{K\mathbf{1}^{\T}\hat{\bphi}\mathbf{1}}
	=
	\frac{\min_{i}\alpha_i}{2\cdot\mathbf{1}^{\T}\hat{\bphi}\mathbf{1}}(\epsilon-\delta)\, .
\end{align}

		\item \textit{Step 2b. Bounding the Numerator.} This part is more complex, and we first decompose the numerator into several terms.
Let $\rho\in(\epsilon,1)$. Because each (normalized) queue length change by at most $\frac{1}{K}$ at each jump almost surely, and that $L_{\balpha}(\cdot)$ is Lipschitz continuous, there exists $K_2 = K_2(\epsilon,\rho)>0$, such that  for all $K\geq K_2$, we have $L(\fX^{K}(\eta_i^{K}))\leq \rho$.
		
		We define another stopping time:
		\begin{equation*}
		\eta^{K,\uparrow}\triangleq
		\inf\{t\geq 0:L_{\balpha}(\fX^{K}(t))\geq 1\}\, .
		\end{equation*}
		
		Then for any $\bx\in\Theta^{K}$, we must have:
		\begin{align*}
		\mathbb{E}
		\left(
		\int_{0}^{\eta_1^{K}}
		\ind \{L_{\balpha}(\fX^{K}(t))\geq 1\} dt
		\left|\fX^{K}(0)=\bx\right.
		\right)
		\leq\ 
		\mathbb{E}
		\left(
		\ind \{\eta^{K,\uparrow}\leq \beta_1^{K}\}
		(\beta_1^{K} - \eta^{K,\uparrow})
		\left|\fX^{K}(0)=\bx\right.
		\right)\, .
		\end{align*}
		
		The above inequality holds because:
		\begin{itemize}
			\item if $\beta_1^{K}\leq\eta^{K,\uparrow}$, then both sides are zero (because the Lyapunov function will hit $\epsilon$ before $1$);
			\item if $\beta_1^{K}>\eta^{K,\uparrow}$, then $L_{\balpha}(\fX^{K}(t))\geq 1$ can occur only for a subset of $t \in [\eta^{K,\uparrow}, \beta_1^{K}]$, and this time interval has length $\beta_1^{K}-\eta^{K,\uparrow}$.
		\end{itemize}
	Hence
		\begin{align*}
		& \;\mathbb{E}
		\left(
		\int_{0}^{\eta_1^{K}}
		\ind \left\{L_{\balpha}(\fX^{K}(t))\geq 1\right\} dt
		\left|\fX^{K}(0)=\bx\right.
		\right)\\
		\leq& \;
		\mathbb{E}
		\left(
		\beta_1^{K} - \eta^{K,\uparrow}\left|
		\eta^{K,\uparrow}\leq \beta_1^{K},
		\fX^{K}(0)=\bx
		\right.
		\right)
		\mathbb{P}
		\left(
		\eta^{K,\uparrow}\leq \beta_1^{K}
		\left|\fX^{K}(0)=\bx\right.
		\right)\, .
		\end{align*}
		
Define
\begin{align*}
  \beta^{K}(\bx) &\triangleq
	\inf\left\{t\geq 0: L_{\balpha}(\fX^{K}(t))\leq \delta \left|\fX^{K}(0)=\bx\right.\right\}\, .
\end{align*}
		Using the properties of Markov chains and conditional expectation, we have:
		\begin{align*}
		\mathbb{E}
		\left(
		\beta_1^{K} - \eta^{K,\uparrow}\left|
		\eta^{K,\uparrow}\leq \beta_1^{K},
		\fX^{K}(0)=\bx\right.
		\right)
		=& \;
		\mathbb{E}\left(
		\mathbb{E}\left(\beta^{K}
		\left(\fX^{K}(\eta^{K,\uparrow})\right)
		\right)\left|
		\eta^{K,\uparrow}\leq \beta_1^{K},
		\fX^{K}(0)=\bx\right.
		\right)\\
\leq&\;
		\sup_{\bx\in\Omega}\mathbb{E}\left(\beta_1^{K}
		\left|\fX^{K}(0)=\bx\right.
		\right)\, .
		\end{align*}
\yka{There is some notation problem above. I believe you want $\mathbb{E}\left(\beta_1^{K}
		\left|\fX^{\N,U}(\eta^{\N,U,\uparrow})\right.
		\right)$ to mean ``expected time it takes for $L_\alpha$ to dip below $\delta$ if you start at $\fX^{\N,U}(\eta^{\N,U,\uparrow})$'' but $\beta_1^{K}$ is already defined for starting at $\bx$, not at $\fX^{\N,U}(\eta^{\N,U,\uparrow})$.}
		Let $T$ be a positive number which will be chosen later.
		Recall that $L_{\balpha}(\bx)\leq \rho$ for all $\bx\in\Theta^{K}$ almost surely when $K\geq K_2$.
		Hence, for any such $\bx\in\Theta^{K}$, we have,
\yka{The numerator of (\ref{eq:achievability_sample_chain}) is an expectation over $\bx$. So the starting quantity  should be $\mathbb{E}
		\left(
		\int_{0}^{\eta_1^{K}}
		\ind \left\{L_{\balpha}(\fX^{\N,U}[t])\geq \alphatwo\right\} dt
		\left|\fX^{\N,U}[0]=\bx\right.
		\right)$ and not ``numerator of (\ref{eq:achievability_sample_chain})'', right?}		
\begin{align}
		& \;\mathbb{E}
		\left(
		\int_{0}^{\eta_1^{K}}
		\ind \left\{L_{\balpha}(\fX^{K}(t))\geq 1\right\} dt
		\left|\fX^{K}(0)=\bx\right.
		\right)\nonumber\\
		\leq& \;
		\mathbb{E}
		\left(
		\beta_1^{K} - \eta^{K,\uparrow}\left|
		\eta^{K,\uparrow}\leq \beta_1^{K},
		\fX^{K}(0)=\bx
		\right.
		\right)
		\mathbb{P}
		\left(
		\eta^{K,\uparrow}\leq \beta_1^{K}
		\left|\fX^{K}(0)=\bx\right.
		\right)\nonumber\\
		\leq& \;
		\left(
		\sup_{\bx\in\Omega}\mathbb{E}\left(\beta_1^{K}
		\left|\fX^{K}(0)=\bx\right.
		\right)
		\right)
		\left[
		\mathbb{P}
		\left(
		\eta^{K,\uparrow}\leq T
		\left|\fX^{K}(0)=\bx\right.
		\right)\right. \nonumber\\
		& \;\left. +
		\mathbb{P}
		\left(
		\beta_1^{K} \geq T
		\left|\fX^{K}(0)=\bx\right.
		\right)
		\right]\qquad \big(\textup{using $\,\eta^{K,\uparrow}\leq \beta_1^{K} \ \Rightarrow \ \eta^{K,\uparrow}\leq T \textup{ or } T \leq \beta_1^{K}$}\,\big)\nonumber\\
		\leq & \;
		\underbrace{\left(\sup_{\bx\in\Omega}\mathbb{E}\left(\beta_1^{K}
			\left|\fX^{K}(0)=\bx\right.
			\right)\right)}_{\text{(a)}}
		\left[
		\underbrace{\sup_{\bx:L_{\balpha}(\bx)\leq \rho}
			\mathbb{P}
			\left(
			\eta^{K,\uparrow}\leq T
			\left|\fX^{K}(0)=\bx\right.
			\right)}_{\text{(b)}}\right.\nonumber\\
		& \;\left. +\,
		\underbrace{\sup_{\bx:L_{\balpha}(\bx)\leq \rho}
			\mathbb{P}
			\left(
			\beta_1^{K} \geq T
			\left|\fX^{K}(0)=\bx\right.
			\right)}_{\text{(c)}}
		\right]\, \label{eq:exp_numerator}.
		\end{align}

		\begin{compactitem}[leftmargin=*]
			\item \textit{Step 2b(i). Bounding term (a).} Term (a) is the upper bound of the expected time for the Lyapunov function to hit a lower level $\delta$ starting from a higher level $\epsilon$. Because the policy $U$ satisfies the negative drift condition, it follows from standard argument (see Part B(1) of the proof of Theorem 4 in \cite{venkataramanan2013queue}, which applies the classical results in \cite{dai1995positive}) that there exists $K_3 = K_3(\delta,\epsilon)$ and constant $C>0$ such that for $K\geq K_3$, we have
			$
			(a)\leq C\, .
			$
			
			\item \textit{Step 2b(ii). Asymptotics for (b).}
			Let $K\to\infty$ and apply Proposition 2 in \cite{venkataramanan2013queue} to $\fKX(\cdot)$. We have:
			\begin{align*}
			&\limsup_{K\to \infty}
			\frac{1}{\N}
			\log\left(
			\sup_{\bx:L_{\balpha}(\bx)\leq \rho}
			\mathbb{P}
			\left(
			\eta^{K,\uparrow}\leq T
			\left|\fKX(0)=\bx\right.
			\right)
			\right)\\
			\leq\;
			-&\inf_{\fA,\fX}\int_{0}^{T}
			\Lambda^*\left(\dot{\fA}(t)\right)dt,
			\textit{ where $(\fA,\fX)$ is an FSP}\\ &\textit{such that $L_{\balpha}(\fX(0))\leq\rho$, $L_{\balpha}(\fX(t))\geq 1$ for some $t\in[0,T]$}\, .
			\end{align*}
			
			\item \textit{Step 2b(iii). Asymptotics for (c).}
Intuitively, term (c) is the tail probability of the duration of a cycle that terminates when the Lyapunov function hit $\delta$.
It remains to be shown that this term is negligible comparing to (b) as $T\to\infty$.
Let $K\to\infty$ and apply Proposition 2 in \cite{venkataramanan2013queue} to $\fKX(\cdot)$. We obtain:
			\begin{align*}
			&\limsup_{K\to \infty}
			\frac{1}{\N}
			\log\left(
			\sup_{\bx:L_{\balpha}(\bx)\leq \rho}
			\mathbb{P}
			\left(
			\beta_1^{K}\geq T
			\left|\fKX(0)=\bx\right.
			\right)
			\right)\\
			\leq\;
			-&\inf_{\fA,\fX}\int_{0}^{T}
			\Lambda^*\left(\dot{\fA}(t)\right)dt,
			\textit{ where $(\fA,\fX)$ is an FSP}\\ &\textit{such that $L_{\balpha}(\fX(0))\leq\rho$, $L_{\balpha}(\fX(t))\geq \delta$ for all $t\in[0,T]$}\, .
			\end{align*}
			
			We focus on the variational problem on the RHS. Note that any FSP that is feasible to the variational problem must satisfy:
			\begin{align*}
				\delta\leq\ L_{\balpha}(\fX(0)) + \int_{t=1}^{T}\dot{L}(\fX(t))dt
				\leq\ \rho + \int_{t=1}^{T}\dot{L}(\fX(t))dt\, .
			\end{align*}
			For any fixed FSP, define $\cT_0 \triangleq \{t\in[0,T]:\dot{L}(\fX(t)) > -\eta  \}$, where $\eta$ is the negative drift parameter in the statement of Proposition \ref{prop:tight_converse}.
			Denote the measure of $\cT_{0}$ by $t_{0}$. Therefore it must hold that:
			\begin{align*}
			\rho + \int_{t=1}^{T} \dot{L}(\fX(t))dt =\ &
			\rho
			+
			\int_{t\notin \cT_{0}}  \dot{L}(\fX(t)) dt
			+
			\int_{t\in \cT_{0}}  \dot{L}(\fX(t)) dt \\
			\leq\ &
			\rho - \eta(T-t_{0}) + \int_{t\in \cT_{0}}
			 \dot{L}(\fX(t)) dt\, .
			\end{align*}
			Hence
			\begin{align*}
				\int_{t\in \cT_{0}} \dot{L}(\fX(t)) dt
				\geq\
				\eta (T - t_{0}) + \delta - \rho
				\geq\
				\eta(T - t_{0}) - 1 \, .
			\end{align*}
			
			There are two cases:
			
			Case 1: When $t_{0} > \frac{T}{2}$. Define
			\begin{align}
				J_{\min} \triangleq \min\quad &\Lambda^*(\dot{\fA}(t)) \label{eq:cost-drift-perturbation} \\
				\textup{subject to}\quad & \dot{L}(\fA(t)) \geq -\eta\, ,\ t\in[0,T]\, ,\ (\fA(t),\fX(t))\textup{ is an FSP.} \nonumber
			\end{align}
			Note that $J_{\textup{min}} \geq \min_{\bof\notin B(\phi,\epsilon')}\Lambda^*(\bof)>0$ and $\epsilon'$ is the $\epsilon$ specified in condition (2) of Proposition \ref{prop:tight_converse}.
			Therefore a lower bound of the exponent of these sample paths is
			\begin{align*}
			\int_{0}^{T}
			\Lambda^*\left(\dot{\fA}(t)\right)dt
			\geq\ 
			\frac{T}{2} J_{\min}\, .
			\end{align*}
			
			Case 2: When $t_{0} \leq \frac{T}{2}$. We have
			\begin{align*}
			\int_{t\in \cT_{0}} \dot{L}(\fX(t)) dt
			\geq\
			\eta(T - t_{0}) - 1
			\geq
			\frac{\eta T}{2} - 1 \, .
			\end{align*}
			We choose $T>\frac{4}{\eta}$, therefore $\frac{\eta T}{2} - 1 \geq \frac{\eta T}{4}$. A lower bound of the exponent of these sample paths is the value of the following variational problem:
			\begin{align*}
				J(T)\triangleq -&\inf_{\fA,\fX}\int_{0}^{T}
				\Lambda^*\left(\dot{\fA}(t)\right)dt,
				\textit{ where $(\fA,\fX)$ is an FSP}\\ &\textit{such that $\int_{0}^{T}\max\{\dot{L}_{\balpha}(\fX(t)),0\} dt \geq \frac{\eta T}{4}$}\, .
			\end{align*}
			We claim that $J(T)\to\infty$ as $T\to\infty$ and prove the claim in step 3.
			
			Combine the two cases, we have:			
			\begin{align*}
			\limsup_{\N\to\infty}\frac{1}{\N}
			\log
			\left(
			\sup_{\bx:L_{\balpha}(\bx)\leq \rho}
			\mathbb{P}
			\left(
			\beta_1^{K} \geq T
			\left|\fKX(0)=\bx\right.
			\right)
			\right)
			\leq
			-
			\min\left\{\frac{T}{2}J_{\min}\, , J(T) \right\} \, .
			\end{align*}
It is not hard to see that as $T\to\infty$, the exponent of term (c) tends to $-\infty$ hence is negligible.
			
\yka{Do you suddenly move from bounding terms for the numerator to bounding the ratio? You haven't even ended two compactitems yet (for (i)-(iii) and for 2a-b).! I have no idea what abruptly happens below, allowing you to bound the pessimistic exponent. Even if you want to use \cite{venkataramanan2013queue}, at least provide some overview of what's going on. Do you take $\delta \rightarrow 0$? What about $\eps, \rho, w$?}
			
\yka{How is it that $w$ was replaced by 1 in \eqref{eq:achievability_multi_variational}, but then reappeared as $w$ in \eqref{eq:variational_problem_original}? How is it that the lower bound on $T$ is $0$. Note how the bound on term $c$ has a positive exponent for $T < 2(\rho-\delta)/\eta$?}
\yka{Need to insert \textbackslash, before punctuation in equations. I inserted several places but this page onwards there are still many pending. Similarly, need to use \textbackslash; after alignment character \&}
\end{compactitem}
\end{compactitem}

Now combine all the terms.
For fixed $\epsilon,\delta,\rho$, note that the denominator of \eqref{eq:achievability_sample_chain} and (a) in \eqref{eq:exp_numerator} are bounded by a constant term, so they have no contribution to the exponent of \eqref{eq:achievability_sample_chain}.
Since as $T\to\infty$, (c) in \eqref{eq:exp_numerator} have an exponent that is at most $-\liminf_{T\to\infty}J(T)$, we have
			\begin{align}
			\limsup_{\N\to\infty}
			\frac{1}{\N}\log \bbp^{K,U}
			\leq\; &
			-\liminf_{T\to\infty}J(T)\, ,\, 
			\limsup_{\N\to\infty}
			\frac{1}{\N}\log
			\left(
			\max_{\fKX(0)\in\Omega}
			\mathbb{P}
			\left(
			L_{\balpha}(\fX^{K}(\infty)) \geq 1
			\right)
			\right)
			\label{eq:achievability_multi_variational}\\
			\leq\; &-\inf_{T>0}
			\inf_{\fA,\fX}\int_{0}^{T}
			\Lambda^*\left(\dot{\fA}(t)\right)dt\nonumber\\
			&\quad\textit{where ($\fA,\fX$) is an FSP such that $L_{\balpha}(\fX(0))=\rho$, $L_{\balpha}(\fX(T))\geq 1$}\, .\label{eq:variational_problem_original}
			\end{align}
Finally, let $\delta,\epsilon, \rho\to 0$, we have
\begin{align*}
			&\limsup_{\N\to\infty}
			\frac{1}{\N}\log \bbp^{K,U}\nonumber\\
			\leq\; &-\inf_{T>0}
			\inf_{\fA,\fX}\int_{0}^{T}
			\Lambda^*\left(\dot{\fA}(t)\right)dt\nonumber\\
			&\quad\textit{where ($\fA,\fX$) is an FSP such that $L_{\balpha}(\fX(0))=0$, $L_{\balpha}(\fX(T))\geq 1$}\, .
			\end{align*}

We briefly summarize Step 2 and provide some intuition.
The goal is to upper bound the stationary likelihood that the Lyapunov function equals $1$.
To study the stationary behavior, we first divide time into cycles, where a cycle is completed each time the Lyapunov function goes down below $\delta$ then rises above $\epsilon$, where $\delta<\epsilon\ll 1$.
Then using a variant of renewal-reward theorem (equation \eqref{eq:achievability_sample_chain}), we only need to lower bound the expected cycle duration, and upper bound the expected time the Lyapunov function stays at $1$ during a cycle.
The Lipschitz property of the Lyapunov function ensures that the cycle duration is bounded away from $0$ hence has no contribution to the \textit{exponent} of the desired likelihood (Lemma \ref{lem:tech_lems}).
Meanwhile, the negative drift condition ensures the expected time until the Lyapunov function returns to $\delta$ after hitting $1$.
This leaves the exponent of the desired likelihood to be solely dependent on the probability that the Lyapunov function ever hit $1$ during a cycle.
Finally we apply the sample path large deviation principle (Fact \ref{fact:sample_path_ldp}) to bound this quantity.

	\noindent\textit{Step 3. Reduce (\ref{eq:achievability_multi_variational}) to an one-dimensional variational problem.}
	This rest of the proof is exactly the same as the proof of Theorem 5 and Proposition 7 in \cite{venkataramanan2013queue}; we provide the intuition and omit the details.
	
	The proof up until this point dealt with the \textit{steady state} of the system.
	Recall the link between the exponent and value of a differential game described in Section \ref{subsec:sufficient_condition_opt}.
	We now lower bound the exponent of the steady state throughput-loss probability by a variational problem (differential game), namely, (\ref{eq:variational_problem_original}).
	Since we are trying to lower bound the adversary's cost, we consider an ``ideal adversary'' who can increase $L_{\balpha}(\bx)$ at the minimum cost at \textit{each} level set.
	Mathematically,
	\begin{align}\label{eq:achievability_one_dimensional}
	\textup{The quantity in }(\ref{eq:variational_problem_original})
	\leq
	-\inf_{T>0}\theta_{T}\, ,
	\end{align}
	\begin{align*}
	\textup{where}\qquad\theta_T \, \triangleq \;&\inf_{L_{\balpha}(\cdot)}
	\int_{0}^{T}l_{\balpha,T}\left(L(t),\dot{L}(t)\right)dt\\
	\textup{s.t. }&L(\cdot)\text{ is absolutely continuous and }
	L(0)=0,\quad L(T)\geq 1\, .\\
	\vspace{3pt}\\
		l_{\balpha,T}(y,v)\triangleq &\inf_{\fA,\fX}\Lambda^*(\bof)\\
	\text{s.t. }&(\fA,\fX)\text{ is an FSP on $[0,T]$ such that for some regular }t\in[0,T]\\\
	&\dot{\fA}(t)=\bof,\quad
	L_{\balpha}(\fX(t))=y,\quad
	\dot{L}_{\balpha}(\fX(t))=v \, .
	\end{align*}
		
	Using the scale-invariance property of $L_{\balpha}(\bx)$ (Lemma \ref{lem:key_property_lyap}), we can show that $l_{\balpha,T}(y,v)$ is independent of $y$ (Proposition 7 in \citealt{venkataramanan2013queue}).
	As a result, the above variational problem reduces to an one-dimensional problem where the ``ideal adversary'' chooses a single rate (i.e., $v$ in the statement of Lemma \ref{lem:lower_bound_vj}) at which $L_{\balpha}(\bx)$ increases.
	This problem is exactly the one in the statement of Lemma \ref{lem:lower_bound_vj}.
	\yka{I like how you have written up step 3 with the intuition. Please do the same at the end of step 2.}
	(We prove the claim in step 2 that $\liminf_{T\to\infty}J(T)=\infty$ here. Using exactly the same argument as in step 3, we can show that $J(T)\geq \frac{\eta T}{4} \gamma_{{\textup{\tiny AB}}}(\balpha)$ where the RHS is defined in \eqref{eq:lower_bound_vj}. This concludes the proof.)

\end{proof}

\end{document}